\RequirePackage{fix-cm}
\documentclass[smallextended,referee,envcountsect]{svjour3}       
\smartqed  

\usepackage{graphicx}
\usepackage{amsmath, ntheorem,dsfont}

\usepackage{pifont}

\usepackage{amssymb}
\usepackage{stackrel}
\usepackage{marvosym, enumerate}
\usepackage{amstext, amscd, latexsym}
\usepackage{amsfonts, ragged2e}
\usepackage{mathrsfs, pgfplots, enumerate, setspace}
\usepackage{subfigure}
\usepackage{cases}
\smartqed
\usepackage{cite}
\usepackage{amstext, graphicx, amscd, latexsym, amssymb}
\usepackage{amsfonts, ragged2e}
\usepackage{pgfplots, setspace}
\usepackage{geometry}
\usepackage{latexsym,bm}
\usepackage{float}
\usepackage{epstopdf,caption}
 
\usepackage{amsfonts,amsmath,amsbsy,amssymb,amscd,mathrsfs}
\usepackage{graphicx}
\usepackage{epstopdf}
\usepackage{algorithmic} 
\usepackage{graphicx,epstopdf}
 \usepackage{mathtools}

 \newcommand{\nm}[1]{\left\lVert {#1} \right\rVert}
 
 \newcommand{\dual}[1]{\left\langle {#1} \right\rangle}

\usepackage[figuresright]{rotating}
\usepackage[sort&compress,numbers]{natbib}
\usepackage[ruled,linesnumbered]{algorithm2e} 
\usepackage[colorlinks,linkcolor=blue,anchorcolor=blue,citecolor=blue,urlcolor=blue]{hyperref} 

\newtheorem{assumption}{Assumption}

\journalname{}

\begin{document}

\title{Second-order Methods for Multiobjective Composite Optimization: Preconditioning Strategies, Subspace Variants and Inexact Solutions}

\author{Jian Chen \and  Xinmin Yang}

\institute{J. Chen \at College of Mathematics, Sichuan University, Chengdu 610065, China\\
	\href{mailto:chenjian_math@163.com}{chenjian\_math@163.com} \\
	\Letter X.M. Yang \at National Center for Applied Mathematics in Chongqing,  Chongqing Normal University, Chongqing 401331, China\\
	\href{mailto:xmyang@cqnu.edu.cn}{xmyang@cqnu.edu.cn}  \\}

\date{Received: date / Accepted: date}

\maketitle

\begin{abstract}
Multiobjective composite optimization problems arise in sparse regularization, constrained multiobjective models, and multi-task learning, but their numerical solution remains challenging when the smooth components are ill-conditioned. Proximal gradient methods are inexpensive per iteration but may converge slowly, while proximal Newton and quasi-Newton methods exploit curvature information at the cost of evaluating expensive metric proximal mappings. To address these issues, we propose a preconditioned proximal Barzilai--Borwein method for multiobjective composite optimization. The method combines objective-wise Barzilai--Borwein scaling, which reduces imbalance among objectives, with a common preconditioner that captures shared curvature information. To avoid non-diagonal metric proximal mappings, we develop a subspace variant in which the search direction is computed in a two-dimensional subspace generated by a proximal-gradient-type direction and a projected historical direction. By constructing a conjugate basis with respect to the preconditioning metric, the subspace model decomposes into tractable one-dimensional subproblems. The framework is further extended to nonsmooth terms of the form $g_i(Ax)$ through a linear-operator-aware preconditioner, yielding explicit proximal evaluations via dual subproblems. We also analyze an inexact version based on relaxed descent conditions. We establish the global convergence of the inexact algorithm in the nonconvex setting and prove a linear convergence rate under an error-bound condition. Numerical experiments on ill-conditioned $\ell_1$-regularized, structured $\ell_1$-regularized, and linearly constrained problems demonstrate the effectiveness of the proposed method.

\end{abstract}

\section{Introduction}
In this paper, we consider the following multiobjective composite optimization problem:
\begin{align*}
	\min\limits_{x\in\mathbb{R}^{n}} F(x), \tag{MCOP}\label{MCOP}
\end{align*}
where $F:\mathbb{R}^{n}\rightarrow(\mathbb{R}\cup\{+\infty\})^{m}$ is a vector-valued function. Each component $F_{i}$, $i=1,2,\cdots,m$, is defined by
$$F_{i}:=f_{i}+g_{i},$$ 
with continuously differentiable function $f_{i}:\mathbb{R}^{n}\rightarrow\mathbb{R}$ and proper convex and lower semicontinuous but not necessarily differentiable function $g_{i}:\mathbb{R}^{n}\rightarrow\mathbb{R}\cup\{+\infty\}$.
In multiobjective optimization, the primary goal is to optimize multiple objective functions simultaneously. Typically, finding a single solution that optimizes all objectives is not {possible}. Therefore, optimality is defined by Pareto optimality or efficiency. A solution {is} said to be Pareto optimal or efficient if none of the objectives can be improved without worsening at least one other objective. {Applications of such problems can be found in various fields}, including engineering \cite{MA2004}, economics \cite{FW2014}, management science \cite{E1984}, and machine learning \cite{SK2018}, among others.

Over the past two decades, descent methods have received increasing attention within the multiobjective optimization community. These methods generate descent directions by solving subproblems, eliminating the necessity for predefined parameters. To the best of our knowledge, the study of multiobjective gradient descent methods can be traced back to the pioneering works by Mukai \cite{M1980} and Fliege and Svaiter \cite{FS2000}. The latter elucidated that the multiobjective steepest descent direction reduces to the steepest descent direction when dealing with a single objective. This observation inspired researchers to extend ordinary numerical algorithms for solving multiobjective optimization problems (MOPs) (see, e.g., \cite{AP2021,BI2005,CL2016,CTY2023,FD2009,FV2016,GI2004,LM2023,LP2018,MP2019,P2014,PS2022,QG2011,SP2023} and references therein). 

For multiobjective composite optimization, Tanabe et al. \cite{TFY2019} proposed the proximal gradient method for multiobjective optimization (PGMO). Subsequent analysis \cite{TFY2023b} shows that PGMO achieves convergence rates of $\mathcal{O}(1/\sqrt{k})$, $\mathcal{O}(1/k)$, and $\mathcal{O}(r^{k})$ in the nonconvex, convex, and strongly convex cases, respectively. Tanabe et al. \cite{TFY2023a} further developed an accelerated proximal gradient method for MOPs, improving the convergence rate to $\mathcal{O}(1/k^{2})$ in the convex case. In the strongly convex case, however, the linear convergence factor of PGMO is given by $r=\sqrt{1-\mu_{\min}/L_{\max}}$ \cite{TFY2023b}, where $\mu_{\min}:=\min{\mu_i:i=1,\ldots,m}$ and $L_{\max}:=\max{L_i:i=1,\ldots,m}$, with $\mu_i$ and $L_i$ denoting the strong convexity and smoothness parameters of the smooth component of the $i$th objective, respectively. Chen et al. \cite{CTY2024} pointed out that objective imbalance, caused by substantially different curvature properties across objective functions, may lead to a small value of $\mu_{\min}/L_{\max}$ and hence slow convergence. To overcome this limitation, Chen et al. \cite{CTY2024} incorporated curvature information to scale each objective, thereby alleviating the adverse effect of objective imbalance. They showed that the resulting scaled proximal gradient method for MOPs converges linearly with the factor $\sqrt{1-\min_{i=1,\ldots,m}{\mu_i/L_i}}$. Nevertheless, despite this improvement, the performance of such methods remains sensitive to the conditioning of the underlying problem. 

To address this issue, second-order methods have been introduced into multiobjective optimization to exploit curvature information. Building on the Newton framework of Fliege et al. \cite{FD2009}, Ansary developed a Newton-type proximal gradient method for MOPs (NPGMO) \cite{A2023}. Chen et al. \cite{CTY2025a} further proposed a new analytical framework and established fast local convergence results for NPGMO under suitable assumptions. Alternatively, Peng et al. \cite{PRY2024} constructed Hessian approximations and developed proximal quasi-Newton methods for MOPs. Despite these advances, the practical implementation of proximal Newton-type methods remains numerically challenging for two main reasons.
\begin{itemize}
	\item First, these methods usually construct objective-wise second-order approximation matrices, which makes the resulting direction-finding subproblem difficult to solve efficiently \cite{CTY2026}. 
	\item Second, even after the objective-wise second-order information is aggregated within the subproblem, each inner iteration often requires evaluating a proximal mapping under a non-Euclidean metric induced by the aggregated Hessian approximation. Such metric proximal mappings are generally unavailable in closed form when the metric matrix is non-diagonal.
\end{itemize}
 
 The first difficulty is not unique to composite multiobjective optimization. It also appears in second-order methods for unconstrained multiobjective optimization, where constructing objective-wise curvature approximations may lead to a high computational burden \cite{CTY2026}. To address this issue, Chen et al. \cite{CTY2026} proposed a ``preconditioning + preconditioning'' strategy, which provides a useful balance between approximation accuracy and computational efficiency. The main idea is to use a common preconditioning matrix to capture the shared curvature information of the problem, while employing objective-wise scaling parameters to alleviate the imbalance among different objectives. This idea can be naturally incorporated into multiobjective composite optimization, where both curvature information and objective-wise scaling are crucial for designing efficient descent directions \cite{CTY2024}.
 
 The second difficulty is also encountered in scalar proximal Newton-type methods. A common remedy is to employ diagonal preconditioners so that the associated proximal mappings remain simple \cite{PDB2020}, or to solve the proximal Newton subproblem inexactly by an inner iterative procedure \cite{LPQ2025}. However, diagonal preconditioning may provide only limited curvature approximation, especially for high-dimensional and ill-conditioned problems. On the other hand, solving the proximal Newton subproblem inexactly by an additional inner routine can be computationally prohibitive in multiobjective algorithms, since the direction-finding procedure already involves a dual inner loop. Adding another layer of inner iterations would significantly increase the total computational cost.
 
 Recently, motivated by the seminal work of Yuan and Stoer \cite{YS1995}, Chen et al. \cite{CY2026} developed a single-objective subspace proximal Newton-type method that avoids evaluation of complicated metric proximal mappings while still exploiting curvature information. In their subspace framework, the search direction is refined in a two-dimensional subspace generated by the current proximal gradient direction and a projected historical direction. By introducing a conjugate basis with respect to the preconditioning metric, the resulting subspace model can be decomposed into tractable one-dimensional subproblems. In this way, the method reduces the complexity of the proximal Newton subproblem and remains effective for high-dimensional ill-conditioned problems. It is also worth noting that subspace second-order method has already shown promising performance in high-dimensional ill-conditioned unconstrained multiobjective optimization \cite{CTY2025b}. These observations suggest that an appropriately designed subspace proximal Newton-type framework may provide a key tool for solving ill-conditioned multiobjective composite optimization problems efficiently.
 
 In addition to the difficulty caused by metric proximal mappings, another important challenge arises from composite terms involving linear operators. In single-objective composite optimization, when the nonsmooth term takes the form $g(Ax)$ and the linear operator $A$ is non-diagonal, the proximal mapping of $g\circ A$ is generally not available in closed form. A standard strategy is to reformulate the problem as a saddle-point problem and then solve it by primal-dual algorithms, such as the Chambolle--Pock method \cite{CP2011} and the PDHG method \cite{HY2012}. The convergence analysis of multiobjective gradient-type methods fundamentally relies on establishing a suitable descent property along the generated sequence. Similarly, in the scalar setting, the convergence analysis of primal-dual algorithms often depends on an underlying descent structure associated with the dual problem \cite{RY2022}. More precisely, the primal-dual iteration for the original problem can often be interpreted as a descent-type method, such as a gradient or proximal point method, applied to the corresponding Fenchel dual problem. In the context of multiobjective optimization, Fenchel duality has been extensively investigated by Bo\c{t} et al. \cite{BGW2009}. However, existing multiobjective Fenchel duality frameworks do not provide a dual descent mechanism analogous to the scalar case. This creates a substantial obstacle to the convergence analysis of primal-dual algorithms for multiobjective composite optimization problems with linear operators. 

The main purpose of this paper is to develop second-order inspired proximal methods for multiobjective composite optimization that strike a balance between curvature exploitation and computational efficiency. The main contributions of this paper are summarized as follows.

\begin{itemize}
	\item[(i)] We propose a preconditioned proximal Barzilai-Borwein method for multiobjective composite optimization. The proposed method uses a common positive definite matrix to approximate the overall second-order information of the problem, while objective-wise Barzilai-Borwein scaling parameters are employed to alleviate imbalance among different objectives. Under standard assumptions, we establish the local superlinear convergence of the proposed method.
	\item[(ii)] To avoid computing complicated metric proximal operators, we develop a subspace preconditioned proximal Barzilai-Borwein framework. More precisely, the preconditioned proximal Barzilai-Borwein subproblem is approximately solved in a two-dimensional subspace generated by a current proximal-gradient-type direction and a projected historical direction. By constructing a conjugate basis with respect to the preconditioning metric, the resulting two-dimensional nonsmooth subproblem can be efficiently approximated by solving two one-dimensional subproblems. It is worth emphasizing that this decomposition relies essentially on the use of a single common second-order approximation. This feature makes the proposed second-order model crucial for the decomposition of the subspace subproblem. We further explain that proximal Newton-type models with objective-wise second-order matrices generally do not admit such a decomposition (see Subsection \ref{sec4.4}).
	
	\item[(iii)] We further extend the proposed subspace second-order framework to multiobjective composite optimization problems involving a linear operator, namely problems of the form $f(x)+g(Ax)$. Since the subspace model can effectively reduce the computational difficulty caused by the linear operator, we design a linear-operator-aware preconditioned subproblem for constructing the current descent direction when $A$ has full row rank. With this preconditioner, the metric proximal mapping appearing in each dual iteration can be computed explicitly. In this way, the difficulty of solving problems with the structure $f(x)+g(Ax)$ is decomposed into two parts: the linear-operator-aware preconditioning removes the computational obstacle caused by $A$ in the construction of the subspace, while the subspace strategy improves the exploitation of local curvature information. To the best of our knowledge, this is the first descent method for multiobjective composite problems with the structure $f(x)+g(Ax)$.
	
	\item[(iv)] We introduce an inexact solution strategy for both the full-space and subspace subproblems. Specifically, approximate solutions are defined through relaxed descent-type optimality conditions (\ref{app1}) and (\ref{app2}). Based on these inexact conditions, we prove that the generated subspace search direction satisfies sufficient descent conditions (see Proposition \ref{suco}). These conditions play a central role in establishing the global convergence of the proposed method. Moreover, under an appropriate error-bound condition, we further obtain a linear convergence rate for the inexact subspace method.
	
	\item[(v)] From the computational perspective, we derive the gradients of the dual subproblems by using Danskin's theorem and solve the resulting simplex-constrained dual problems by a spectral projected gradient method with warm start. Numerical experiments on high-dimensional ill-conditioned quadratic problems with $\ell_1$-regularization, structured $\ell_1$-regularization, and linear constraints demonstrate the effectiveness of the proposed inexact subspace method. In particular, owing to the inexact solution strategy and the warm-start technique, each dual subproblem requires fewer than two inner iterations on average to satisfy the inexact descent condition. In many cases, no additional inner iteration is needed, since the warm-started approximate dual solution from the previous iteration already satisfies the current inexact descent condition.
\end{itemize}

The rest of this paper is organized as follows. Section \ref{sec2} introduces some basic concepts and algorithms for (\ref{MCOP}). Section \ref{sec3} presents the preconditioned proximal Barzilai-Borwein method and analyzes its local superlinear convergence. Section \ref{sec4} develops the subspace framework and discusses the extension to composite problems with linear operators. Section \ref{sec5} introduces the inexact subspace preconditioned method and establishes its global and linear convergence properties. Section \ref{sec6} reports numerical experiments on several classes of ill-conditioned test problems. Finally, some conclusions are drawn at the end of the paper.

\section{Preliminaries and Related Algorithms}\label{sec2}
\subsection{Notations and Auxiliary Results}
Throughout the paper, we equip the Euclidean space $\mathbb{R}^n$ with the standard inner product $\langle\cdot,\cdot\rangle$ and its induced norm $\|\cdot\|$. Denote by $\mathbb{S}^{n}_{++}$ ($\mathbb{S}^{n}_{+}$) the set of symmetric positive (semi-)definite matrices and by $\mathbb{O}^{n}$ the set of orthogonal matrices in $\mathbb{R}^{n\times n}$. The rank of a matrix is denoted by $\mathcal{R}(\cdot)$. For a positive definite matrix $H$, the notation $\|x\|_{H}=\sqrt{\langle x,Hx \rangle}$ is used to represent the norm induced by $H$ on vector $x$. For a mapping $f:\mathbb{R}^n\to\mathbb{R}^m$, we denote by $Jf(x)\in\mathbb{R}^{m\times n}$ its Jacobian at $x$, and by $\nabla f_i(x)\in\mathbb{R}^n$ the gradient of its $i$th component $f_i$ at $x$. We write $F_i'(x;d)$ for the directional derivative of $F_i$ at $x$ along the direction $d$. For simplicity, we denote $[m]:=\{1,2,...,m\}$, and $$\Delta_{m}:=\left\{\lambda:\sum\limits_{i\in[m]}\lambda_{i}=1,\lambda_{i}\geq0,\ i\in[m]\right\}$$ the $m$-dimensional unit simplex. {For $a,b\in\mathbb{R}^{m}$, if $b_{i}\neq0$ for $i\in[m]$, we denote 
$$\frac{a}{b}:=\left(\frac{a_{1}}{b_{1}},\cdots,\frac{a_{m}}{b_{m}}\right).$$} To prevent any ambiguity, we establish the  {partial} order $\preceq(\prec)$ in $\mathbb{R}^{m}$ as follows: $$u\preceq(\prec)v~\Leftrightarrow~v-u\in\mathbb{R}^{m}_{+}(\mathbb{R}^{m}_{++}),$$
and in $\mathbb{S}^{n}$ as follows:
$$U\preceq(\prec)V~\Leftrightarrow~V-U\in\mathbb{S}^{n}_{+}(\mathbb{S}^{n}_{++}).$$
\par Next, we introduce optimality concepts for (\ref{MCOP}) in the Pareto sense. 

\begin{definition}\label{def1}
	A vector $x^{\ast}\in\mathbb{R}^{n}$ is said to be a Pareto solution to (\ref{MCOP}), if there exists no $x\in\mathbb{R}^{n}$ such that $F(x)\preceq F(x^{\ast})$ and $F(x)\neq F(x^{\ast})$.
\end{definition}

\begin{definition}\label{def2}
	A vector $x^{\ast}\in\mathbb{R}^{n}$ is said to be a weakly Pareto solution to (\ref{MCOP}), if there exists no $x\in\mathbb{R}^{n}$ such that $F(x)\prec F(x^{\ast})$.
\end{definition}

\begin{definition}\label{def3}
	A vector $x^{\ast}\in\mathbb{R}^{n}$ is said to be Pareto critical point of (\ref{MCOP}), if
	$$\max\limits_{i\in[m]}F_{i}^{\prime}(x^{*};d)\geq0,~ \forall d\in\mathbb{R}^{n}.$$
\end{definition}

\par From Definitions \ref{def1} and \ref{def2}, it is evident that Pareto solutions are always weakly Pareto solutions. The following lemma shows the relationships among the three concepts of Pareto optimality.

\begin{lemma}[See Theorem 3.1 of \cite{FD2009}] The following statements hold.
	\begin{itemize}
		\item[$\mathrm{(i)}$]  If $x\in\mathbb{R}^{n}$ is a weakly Pareto solution to (\ref{MCOP}), then $x$ is a Pareto critical point.
		\item[$\mathrm{(ii)}$] Let every component $F_{i}$ of $F$ be convex. If $x\in\mathbb{R}^{n}$ is a Pareto critical point of (\ref{MCOP}), then $x$ is a weakly Pareto solution.
		\item[$\mathrm{(iii)}$] Let every component $F_{i}$ of $F$ be strictly convex. If $x\in\mathbb{R}^{n}$ is a Pareto critical point of (\ref{MCOP}), then $x$ is a Pareto solution.
	\end{itemize}
\end{lemma}

\par Next, we introduce a relaxation of Pareto criticality, called preconditioned $\varepsilon$-Pareto criticality.
\begin{definition}\label{pc}
	Let $\varepsilon>0$, $P\succ0$. A vector $x\in\mathbb{R}^{n}$ said to be a preconditioned $\varepsilon$-Pareto critical point of (\ref{MCOP}), if $\nm{d(x)}\leq \varepsilon$, where $d(x)$ is the minimizer of the following subproblem:
	$$\min\limits_{d\in\mathbb{R}^{n}} \max\limits_{i\in[m]}\ \left\langle\nabla f_{i}(x),d\right\rangle + g_{i}(x+d)-g_{i}(x)+\frac{1}{2}\|d\|_{P}^{2}.$$	
\end{definition}

\begin{definition}
	A differentiable function $h:\mathbb{R}^{n}\rightarrow\mathbb{R}$ is $L$-smooth if $$h(y)\leq h(x) + \dual{\nabla h(x),y-x}+\frac{L}{2}\|y-x\|^{2}$$  holds for
	all $x,y\in\mathbb{R}^{n}$. And $h$ is $\mu$-strongly convex if $$h(y)\geq h(x) + \dual{\nabla h(x),y-x}+\frac{\mu}{2}\|y-x\|^{2}$$ holds for all $x,y\in\mathbb{R}^{n}$. When the Euclidean distance is replaced by $\|\cdot\|_{B}$, where $B$ is a positive definite matrix, we say $h$ is $L$-smooth and $\mu$-strongly convex relative to $\|\cdot\|_{B}$.
\end{definition}

The proximal operator associated with $h$ is defined by
\[
\mathrm{prox}_h(x)
=\arg\min_{u\in\mathbb{R}^n}
\left\{
h(u)+\frac12\|u-x\|^2
\right\}.
\]
The corresponding Moreau envelope is defined by
\[
\mathcal{M}_h(x)
=\min_{u\in\mathbb{R}^n}
\left\{
h(u)+\frac{1}{2}\|u-x\|^2
\right\}.
\]

Let $H \in \mathbb{R}^{n\times n}$ be a symmetric positive definite matrix. 
The proximal operator of $h$ with respect to the metric induced by $H$ is defined as
\[
\mathrm{prox}_h^H(x)
=\arg\min_{u\in\mathbb{R}^n}
\left\{
h(u)+\frac{1}{2}\|u-x\|_H^2
\right\},
\]
where $\|u-x\|_H^2 := (u-x)^\top H (u-x)$.

The corresponding Moreau envelope under the metric $H$ is defined by
\[
\mathcal{M}_h^H(x)
=\min_{u\in\mathbb{R}^n}
\left\{
h(u)+\frac{1}{2}\|u-x\|_H^2
\right\}.
\]

To simplify the notation in our analysis, we denote by $$h_{\lambda}(x):=\sum\limits_{i\in[m]}\lambda_{i}h_{i}(x),$$
$$\nabla h_{\lambda}(x):=\sum\limits_{i\in[m]}\lambda_{i}\nabla h_{i}(x),$$
$$\nabla^{2}h_{\lambda}(x):=\sum\limits_{i\in[m]}\lambda_{i}\nabla^{2}h_{i}(x).$$

Next, we introduce some auxiliary results, which will be used in computing gradients of direction-finding subproblems and convergence analysis.
\begin{lemma}[Danskin's theorem: differentiability case; see {\cite[Proposition~B.22]{B2016}}]
	\label{danskin}
	Let $Y\subset \mathbb{R}^{m}$ be a compact set, and let
	$\phi:\mathbb{R}^{n}\times Y\to \mathbb{R}$ be continuous and such that $\phi(\cdot,y):\mathbb{R}^{n}\to \mathbb{R}$ is convex for each $y\in Y$. Define
	\[
	\psi(x):=\max_{y\in Y}\phi(x,y),
	\qquad
	Y(x):=\arg\max_{y\in Y}\phi(x,y).
	\]
	If, for a given $x\in\mathbb{R}^{n}$, the maximization problem has a unique
	optimal solution, i.e.,
	\[
	Y(x)=\{\bar y(x)\},
	\]
	and $\phi(\cdot,\bar y(x))$ is differentiable at $x$, then $\psi$ is differentiable at $x$, and
	\[
	\nabla \psi(x)=\nabla_x\phi(x,\bar y(x)).
	\]
\end{lemma}

\begin{lemma}\cite[Lemma 1]{CY2026}\label{ll2}
	Let $h:\mathbb{R}^{n}\rightarrow\mathbb{R}\cup\{+\infty\}$ be a proper convex and lower semicontinuous function, which is not necessarily differentiable. Assume that $x^{*}$ is the minimizer of
	\begin{equation}\label{op}
		\min\limits_{x\in\mathbb{R}^{n}}h(x)+\frac{1}{2}\nm{x}^{2}_{P},
	\end{equation}
	where $P\succ0$. Then 
	\begin{equation}
		h(x^*)-h(0)\leq-\nm{x^{*}}^{2}_{P}
	\end{equation}
\end{lemma}

\par In the following, we briefly review some descent methods for MCOPs.
\subsection{Proximal gradient method}
\par For $x\in\mathbb{R}^{n}$, the proximal gradient descent direction \cite{TFY2019} is defined as the optimal solution of the following subproblem:
\begin{equation}\label{prox}
	\min\limits_{d\in\mathbb{R}^{n}}\max\limits_{i\in[m]}\left\{{
		\left\langle\nabla f_{i}(x^{k}),d\right\rangle + g_{i}(x^{k}+d)-g_{i}(x^{k})}+\frac{\ell}{2}\|d\|^{2}\right\},
\end{equation}
where $\ell>0$. The subproblem can be equivalently expressed as
$$\min\limits_{d\in\mathbb{R}^{n}}\max\limits_{\lambda\in\Delta_{m}}\left\{{
	\left\langle\nabla f_{\lambda}(x^{k}),d\right\rangle + g_{\lambda}(x^{k}+d)-g_{\lambda}(x^{k})}+\frac{\ell}{2}\|d\|^{2}\right\}.$$
There exists $\lambda_{PG}^{k}\in\Delta_{m}$ such that
$$d^{k}_{PG}=\mathrm{prox}_{\frac{1}{\ell}{ g_{\lambda_{PG}^{k}}}}\left(x^{k}-\frac{1}{\ell}\nabla f_{\lambda_{PG}^{k}}(x^{k})\right)  - x^{k},$$ and $\lambda_{PG}^{k}$ is the optimal solution of the following dual problem \cite[Section 6]{TFY2023a}:
\begin{align*}
	-&\min\limits_{\lambda}\omega_{PG}^{k}(\lambda)\\
	&\mathrm{ s.t.} \ \lambda\in\Delta_{m},
\end{align*}
where 

	$$\omega_{PG}^{k}(\lambda):=\frac{1}{2\ell} \left\|\nabla f_{\lambda}(x^{k})\right\|^{2}+g_{\lambda}(x^{k})-\mathcal{M}_{\frac{1}{\ell}{ g_{\lambda}}}\left(x^{k}-\frac{1}{\ell}\nabla f_{\lambda}(x^{k})\right).$$
As described in \cite[Theorem 13]{TFY2023a}, the dual subproblem is differentiable and its gradient can be written as follows:
\begin{align*}
	\nabla\omega_{PG}^{k}(\lambda)&={g(x^{k})}-{Jf(x^{k})}\left(\mathrm{prox}_{\frac{1}{\ell}{ g_{\lambda}}}\left(x^{k}-\frac{1}{\ell}\nabla f_{\lambda}(x^{k})\right)-x^{k}\right)\\
	&-g\left(\mathrm{prox}_{\frac{1}{\ell}{ g_{\lambda}}}\left(x^{k}-\frac{1}{\ell}\nabla f_{\lambda}(x^{k})\right)\right).
\end{align*}

\begin{remark}
	In \cite[Theorem~13]{TFY2023a}, the gradient of the dual objective is derived from the differentiability of the Moreau envelope. It is worth noting that the same gradient formula can also be obtained by applying Danskin's theorem. Specifically, by Sion's minimax theorem, the direction-finding subproblem can be equivalently written as
	\[
	-\min_{\lambda\in\Delta_m}\max_{d\in\mathbb{R}^n}
	\left\{
	-\left\langle \nabla f_{\lambda}(x^k),d\right\rangle
	-g_{\lambda}(x^k+d)+g_{\lambda}(x^k)
	-\frac{\ell}{2}\|d\|^2
	\right\}.
	\]
	For this purpose, define
	\[
	\phi(\lambda,d):=
	-\left\langle \nabla f_{\lambda}(x^k),d\right\rangle
	-g_{\lambda}(x^k+d)+g_{\lambda}(x^k)
	-\frac{\ell}{2}\|d\|^2 .
	\]
	Then
	\[
	\omega_{PG}^k(\lambda)=\max_{d\in\mathbb{R}^n}\phi(\lambda,d).
	\]
	Since the maximization problem is strongly concave in $d$, its upper level sets are compact and it admits a unique maximizer. Moreover, since $\phi(\cdot,d)$ is affine for every fixed $d\in\mathbb{R}^n$, Danskin's theorem yields
	\[
	\nabla \omega_{PG}^k(\lambda)
	=
	\nabla_{\lambda}\phi(\lambda,d(\lambda))
	=
	g(x^k)-Jf(x^k)d(\lambda)
	-g\left(x^k+d(\lambda)\right),
	\]
	where
	\[
	d(\lambda):=
	\operatorname{prox}_{\frac{1}{\ell}g_{\lambda}}
	\left(
	x^k-\frac{1}{\ell}\nabla f_{\lambda}(x^k)
	\right)-x^k
	\]
	is the unique solution of $\max_{d\in\mathbb{R}^n}\phi(\lambda,d)$.
\end{remark}

\subsection{Proximal Newton-type methods}
Similar to its counterpart for SOPs, PGMO is sensitive to problem's conditioning. In response to this challenge, Ansary \cite{A2023} proposed a proximal Newton method for MCOPs. The proximal Newton direction is the optimal solution to the following subproblem:
\begin{equation}\label{nt}
	\mathop{\min}\limits_{d\in\mathbb{R}^{n}}\max\limits_{i\in[m]}\left\{\dual{\nabla f_{i}(x^{k}),d}+g_{i}(x^{k}+d)-g_{i}(x^{k})+\frac{1}{2}\dual{d,\nabla^{2}f_{i}(x^{k})d}\right\},
\end{equation}
where the Hessian matrices $\nabla^{2}f_{i}(x^{k})$ are positive definite for all $i\in[m]$ and $x\in\mathbb{R}^{n}$, the subproblem can be equivalently expressed as
$$\mathop{\min}\limits_{d\in\mathbb{R}^{n}}\max\limits_{\lambda\in\Delta_{m}}\left\{\langle\nabla f_{\lambda}(x^{k}),d\rangle+g_{\lambda}(x^{k}+d)-g_{\lambda}(x^{k})+\frac{1}{2}\dual{d,\nabla^{2}f_{\lambda}(x^{k})d}\right\}.$$
By Sion's minimax theorem, there exists $\lambda_{N}^{k}\in\Delta_{m}$ such that
$$d^{k}_{N}=\mathrm{prox}^{\nabla^{2}f_{\lambda_{N}^{k}}(x^{k})}_{{ g_{\lambda_{N}^{k}}}}\left(x^{k}-[\nabla^{2}f_{\lambda^{k}_{N}}(x^{k})]^{-1}\nabla f_{\lambda_{N}^{k}}(x^{k})\right)  - x^{k},$$ and $\lambda^{k}_{N}$ is the optimal solution of the following dual problem
\begin{align*}
	-&\min\limits_{\lambda}\omega_{N}^{k}(\lambda)\\
	&\mathrm{ s.t.} \ \lambda\in\Delta_{m},
\end{align*}
where 
$$\omega_{N}^{k}(\lambda):=\frac{1}{2} \left\|\nabla f_{\lambda}(x^{k})\right\|_{[\nabla^{2}f_{\lambda}(x^{k})]^{-1}}^{2}+g_{\lambda}(x^{k})-\mathcal{M}^{\nabla^{2}f_{\lambda}(x^{k})}_{{ g_{\lambda}}}\left(x^{k}-[\nabla^{2}f_{\lambda}(x^{k})]^{-1}\nabla f_{\lambda}(x^{k})\right).$$
By Danskin's theorem, we obtain the gradient of $\omega_{N}^{k}$ as
\begin{align*}
	\nabla \omega_{N}^{k}(\lambda)
	= g(x^{k}) - Jf(x^{k})d_{N}(\lambda) -g\left(x^k+d(\lambda)\right)
	-\frac{1}{2}Df(x^{k})[d_{N}(\lambda)]^{2},
\end{align*}
where
\[
d_{N}(\lambda)
:=
\operatorname{prox}^{\nabla^{2}f_{\lambda}(x^{k})}_{g_{\lambda}}
\left(
x^{k}
-\left[\nabla^{2}f_{\lambda}(x^{k})\right]^{-1}
\nabla f_{\lambda}(x^{k})
\right)
- x^{k},
\]
and
\[
Df(x^{k})[d_{N}(\lambda)]^{2}
:=
\left(
\left\langle d_{N}(\lambda),\nabla^{2}f_{1}(x^{k})d_{N}(\lambda)\right\rangle,
\ldots,
\left\langle d_{N}(\lambda),\nabla^{2}f_{m}(x^{k})d_{N}(\lambda)\right\rangle
\right).
\]
However, since the matrix $\nabla^{2} f_{\lambda}(x^{k})$ is generally non-diagonal, the evaluation of the above metric proximal mapping, and hence of the associated gradient, can be computationally challenging. Furthermore, exact Hessian matrices are often difficult to obtain and may not be positive definite. To address these issues, Peng et al. \cite{PRY2024} employed quasi-Newton-type methods, which construct positive definite approximations of the Hessian so as to incorporate curvature information.

\subsection{Scaled proximal gradient method}
 Chen et al. \cite{CTY2024} devised the scaled proximal gradient descent direction, which is the optimal solution of the following subproblem:
\begin{equation}\label{bb}
	\min\limits_{d\in\mathbb{R}^{n}}\max\limits_{i\in[m]}\left\{\frac{
		\left\langle\nabla f_{i}(x^{k}),d\right\rangle + g_{i}(x^{k}+d)-g_{i}(x^{k})}{\alpha_{i}(x^{k})}+\frac{1}{2}\|d\|^{2}\right\}.
\end{equation}
where $\alpha(x^{k})\in\mathbb{R}^{m}_{++}$ is given by Barzilai-Borwein method:
\begin{equation}\label{bbalpha_k}
	\alpha_{i}(x^{k})=\left\{
	\begin{aligned}
		&\max\left\{\alpha_{\min},\min\left\{\frac{\langle s_{k-1},y^{k-1}_{i}\rangle}{\nm{s_{k-1}}^{2}}, \alpha_{\max}\right\}\right\}, & \langle s_{k-1},y^{k-1}_{i}\rangle&>0, \\
		&\max\left\{\alpha_{\min},\min\left\{\frac{\nm{y^{k-1}_{i}}}{\nm{s_{k-1}}}, \alpha_{\max}\right\}\right\}, & \langle s_{k-1},y^{k-1}_{i}\rangle&<0, \\
		& \alpha_{\min}, &  \langle s_{k-1},y^{k-1}_{i}\rangle&=0,
	\end{aligned}
	\right.
\end{equation}
for all $i\in[m]$, where $\alpha_{\max}$ is a sufficiently large positive constant and $\alpha_{\min}$ is a sufficiently small positive constant, $s_{k-1}=x^{k}-x^{k-1},\ y^{k-1}_{i}=\nabla f_{i}({x^{k}})-\nabla f_{i}(x^{k-1}),\ i\in[m].$ By Sion's minimax theorem, there exists $\lambda_{BB}^{k}\in\Delta_{m}$ such that
\begin{equation}\label{bbk}
d^{k}_{BB}=\mathrm{prox}_{{ g_{{\lambda_{BB}^{k}}/{\alpha(x^{k})}}}}\left(x^{k}-\nabla f_{{\lambda_{BB}^{k}}/{\alpha(x^{k})}}(x^{k})\right)  - x^{k},	
\end{equation}
 and $\lambda_{BB}^{k}$ is the optimal solution of the following dual problem
\begin{align*}
	-&\min\limits_{\lambda}\omega_{BB}^{k}(\lambda)\\
	&\mathrm{ s.t.} \ \lambda\in\Delta_{m},
\end{align*}
where 

$$\omega_{BB}^{k}(\lambda):=\frac{1}{2} \left\|\nabla f_{{\lambda}/{\alpha(x^{k})}}(x^{k})\right\|^{2}+g_{{\lambda}/{\alpha(x^{k})}}(x^{k})-\mathcal{M}_{{ g_{{\lambda}/{\alpha(x^{k})}}}}\left(x^{k}-\nabla f_{{\lambda}/{\alpha(x^{k})}}(x^{k})\right).$$
By Danskin's theorem, its gradient can be written as follows:
\begin{align*}
	\nabla\omega_{BB}^{k}(\lambda)&=\frac{g(x^{k})}{\alpha(x^k)}-\frac{{Jf(x^{k})}\left(\mathrm{prox}_{{ g_{{\lambda}/{\alpha(x^{k})}}}}\left(x^{k}-\nabla f_{{\lambda}/{\alpha(x^{k})}}(x^{k})\right)-x^{k}\right)}{\alpha(x^k)}\\
	&-\frac{g\left(\mathrm{prox}_{{ g_{{\lambda}/{\alpha(x^{k})}}}}\left(x^{k}-\nabla f_{{\lambda}/{\alpha(x^{k})}}(x^{k})\right)\right)}{\alpha(x^k)}.
\end{align*}
\begin{remark}
	{Chen et al. \cite{CTY2024} demonstrated that scaled proximal gradient method for MCOPs (SPGMO) can alleviate the impacts of interference and imbalances among objectives, resulting in improved convergence rates compared to PGMO. However, it is essential to note that BBDMO still exhibits sensitivity to conditioning, as observed from a theoretical perspective \cite{CTY2024}.}
\end{remark}
\section{Preconditioned proximal Barzilai-Borwein method}\label{sec3}
For unconstrained multiobjective optimization, Chen et al. \cite{CTY2026} proposed a preconditioned Barzilai-Borwein method to balance per-iteration cost and curvature exploration. Similarly, we can adapt the idea in (\ref{MCOP}).
%
Naturally, to mitigate the impact of conditioning and imbalances among objectives, we aim to leverage the strengths of both NPGMO and SPGMO in developing the descent direction:

\begin{equation}\label{d}
	\min\limits_{d\in\mathbb{R}^{n}}\max\limits_{i\in[m]}\left\{\frac{
		\left\langle\nabla f_{i}(x^{k}),d\right\rangle + g_{i}(x^{k}+d)-g_{i}(x^{k})}{\alpha_{i}^{k}}+\frac{1}{2}\|d\|_{B_{k}}^{2}\right\},
\end{equation}
where $\alpha^{k}\succ 0$ mitigates the impact of objective imbalances, and $B_{k}\succ0$ is applied to better capture the local geometry of the problem. We denote by $d^{k}$
the optimal solution of the minimization problem in (\ref{d}).
By Sion's minimax theorem, there exists $\lambda^{k}\in\Delta_{m}$ such that
\begin{equation}\label{dk}
	d^{k}=\mathrm{prox}^{B_{k}}_{ g_{{\lambda^{k}}/{\alpha^{k}}}}\left(x^{k}-B_k^{-1}\nabla f_{{\lambda^{k}}/{\alpha^{k}}}(x^{k})\right)  - x^{k},
\end{equation}
and
\begin{equation}\label{subgrad}
	-\nabla f_{{\lambda^{k}}/{\alpha^{k}}}(x^{k}) - B_{k}d^{k}\in \partial g_{{\lambda^{k}}/{\alpha^{k}}}(x^k+d^k),
\end{equation} 
where $\lambda^{k}$ is the optimal solution of the following dual problem
\begin{equation}\label{pdp}
	\begin{aligned}
		-&\min\limits_{\lambda}\omega^{k}(\lambda)\\
		&\mathrm{ s.t.} \ \lambda\in\Delta_{m},
	\end{aligned}
\end{equation}

where 

$$\omega^{k}(\lambda):=\frac{1}{2} \left\|\nabla f_{{\lambda}/{\alpha^{k}}}(x^{k})\right\|_{B_{k}^{-1}}^{2}+g_{{\lambda}/{\alpha^{k}}}(x^{k})-\mathcal{M}^{B_k}_{ g_{{\lambda}/{\alpha^{k}}}}\left(x^{k}-B_{k}^{-1}\nabla f_{{\lambda}/{\alpha^{k}}}(x^{k})\right).$$
By Danskin's theorem, the gradient of $\omega^{k}$ can be written as follows:
\begin{align*}
	\nabla\omega^{k}(\lambda)&=\frac{g(x^{k})}{\alpha^k}-\frac{{Jf(x^{k})}\left(\mathrm{prox}^{B_{k}}_{ g_{{\lambda}/{\alpha^{k}}}}\left(x^{k}-B_k^{-1}\nabla f_{{\lambda}/{\alpha^{k}}}(x^{k})\right)-x^{k}\right)}{\alpha^k}\\
	&-\frac{g\left(\mathrm{prox}^{B_{k}}_{ g_{{\lambda}/{\alpha^{k}}}}\left(x^{k}-B_k^{-1}\nabla f_{{\lambda}/{\alpha^{k}}}(x^{k})\right)\right)}{\alpha^k}.
\end{align*}

Next, we will present several properties of $d^{k}$.
\begin{proposition}\label{pop}
	Assume that $0< \alpha_{\min}\leq \alpha_{i}^{k}\leq\alpha_{\max}$, $a\bm I_{n}\preceq B_{k}\preceq b\bm I_{n}(a>0)$ for all $k\geq0,~i\in[m]$. Let $d^{k}$
	be defined as (\ref{d}), then the following statements hold.
	\begin{itemize}
		\item[$\mathrm{(i)}$] the following assertions are equivalent:
		\subitem$\mathrm{(a)}$ The point $x^{k}$ is non-critical;
		\subitem$\mathrm{(b)}$ $d^{k}\neq0$;
		\subitem$\mathrm{(c)}$ $d^{k}$ is a descent direction.
		\item[$\mathrm{(ii)}$] If there exists a convergent subsequence $x^{k}\stackrel{\mathcal{K}}{\longrightarrow} x^{*}$ such that $d^{k}\stackrel{\mathcal{K}}{\longrightarrow}0$, then $x^{*}$ is Pareto critical.
	\end{itemize}
\end{proposition}
\begin{proof}
	The assertions can be obtained by using the similar arguments as in the proof of \cite[Lemma 3.2]{CTY2023b}.
\end{proof}

\subsection{``Preconditioning'' $+$ ``Preconditioning''}
The remaining question is how to choose $\alpha^{k}$ and $B_{k}$ to preserve the
advantages of NPGMO and SPGMO. For similar issues in unconstrained MOPs, Chen et al. \cite{CTY2026} developed a ``Preconditioning'' $+$ ``Preconditioning'' strategy. Specifically, $B_{k}\approx \nabla^{2} f_{{\lambda^{k}}/{\alpha^{k}}}(x^{k})$
is a preconditioner that captures the overall geometry of the problem, while $\alpha^{k}_{i},~i\in[m]$, serve as additional diagonal preconditioners, namely $\alpha^{k}_{i}\bm I_{n},~i\in[m]$, that adapt to the local geometry of each objective function in the transformed space equipped with norm $\|\cdot\|_{B_{k}}$. Accordingly, we set $\alpha^{k}\in\mathbb{R}^{m}_{++}$ as follows:
\begin{equation}\label{alpha_k}
	\alpha^{k}_{i}=\left\{
	\begin{aligned}
		&\max\left\{\alpha_{\min},\min\left\{\frac{\langle s_{k-1},y^{k-1}_{i}\rangle}{\nm{s_{k-1}}^{2}_{B_{k}}}, \alpha_{\max}\right\}\right\}, & \langle s_{k-1},y^{k-1}_{i}\rangle&>0, \\
		&\max\left\{\alpha_{\min},\min\left\{\frac{\nm{y^{k-1}_{i}}}{\nm{B_{k}s_{k-1}}}, \alpha_{\max}\right\}\right\}, & \langle s_{k-1},y^{k-1}_{i}\rangle&<0, \\
		& \alpha_{\min}, &  \langle s_{k-1},y^{k-1}_{i}\rangle&=0.
	\end{aligned}
	\right.
\end{equation}

The preconditioned proximal Barzilai-Borwein method for MCOPs is described as follows.

\begin{algorithm} 
	\caption{Preconditioned proximal Barzilai-Borwein method for MCOPs}\label{pbb}
	\LinesNumbered  
	\KwData{$x^{0}\in\mathbb{R}^{n},~B_{0}\succ0,~\sigma,\gamma\in(0,1)$}
	{Choose $x^{-1}$ in a small neighborhood of $x^{0}$}\\
	\For{$k=0,...$}{Update $\alpha^{k}_{i}$ as (\ref{alpha_k}),\ $i\in[m]$\\
		Compute $\lambda^{k}$ a solution of (\ref{pdp})\\
	Update $d^{k}$ as (\ref{dk})\\
		\eIf{$d^{k}=0$}{ {\bf{return}} Pareto critical point $x^{k}$ }{
			Compute the stepsize $t_{k}\in(0,1]$ in the following way:
			\begin{align*}
				t_{k}:=\max\big\{\gamma^{j}:j\in\mathbb{N},~&F_{i}\left(x^{k}+\gamma^{j}d_{k}\right)-F_{i}(x^{k})\\
				&\leq \sigma\gamma^{j}(\dual{\nabla f_{i}(x^{k}),d_{k}} + g_{i}(x^{k}+ d_{k})-g_{i}(x^{k})).\big\}	
			\end{align*}
			Update $x^{k+1}:= x^{k}+t_{k}d^{k}$\\
			Update $B_{k+1}\succ0$ 
		}
	}  
\end{algorithm}

\subsection{Local superlinear convergence}
For a generic case, we further explain the choice of $B_{k}$ by the following asymptotic convergence result.
\begin{theorem}\label{sup}
	Let $x^{k+1}=x^{k}+t_{k}d^{k}$, $d^{k}$ be denoted as the minimizer of (\ref{d}),
	suppose the following assumptions hold:
	\begin{itemize}
		\item[$\mathrm{(a)}$] $\{x^{k}\}$ converges to some Pareto solution $x^{*}$ and $F(x^{*})\preceq F(x^{k})$ for all $k$,
		\item[$\mathrm{(b)}$] $t_{k}=1$ for sufficiently large $k$,
		\item[$\mathrm{(c)}$] $\left\{\sum_{i\in[m]}\frac{\lambda^{k}_{i}}{\alpha^{k}_{i}}\right\}$ is bounded,
		\item[$\mathrm{(d)}$]  $aI\preceq\lim\limits_{k\rightarrow\infty}\nabla^{2}f_{{\lambda^{k}}/{\alpha^{k}}}(x^{*})\preceq bI,~(a>0)$,
		\item[$\mathrm{(e)}$] $\lim\limits_{k\rightarrow\infty}\nabla^{2}f_{i}(x^{k}) = \nabla^{2}f_{i}(x^{*})$ for all $i\in[m]$,
		\item[$\mathrm{(f)}$]$\lim\limits_{k\rightarrow\infty}\frac{\nm{\left(B_{k}-\nabla^{2}f_{{\lambda^{k}}/{\alpha^{k}}}(x^{k})\right)s_{k}}}{\nm{s_{k}}}=0.$
	\end{itemize}
	Then, $\{x^{k}\}$ converges to $x^{*}$ superlinearly.
\end{theorem}
\begin{proof}
		Given the twice continuity of $f_{i}$, we use Newton-Leibniz formula to get
		\begin{equation}\label{eqn}
			f_{i}(b)-f_{i}(a)=\dual{\int_{0}^{1}\nabla f_{i}(a+t(b-a))dt,b-a}.
		\end{equation}
		Again using the Newton-Leibniz formula for the average gradient, we have
		$$\int_{0}^{1}(\nabla f_{i}(a+t(b-a))-\nabla f_{i}(a))dt=\int_{0}^{1}\int_{0}^{1}\nabla^{2}f_{i}(a+st(b-a))ds(t(b-a))dt.$$
		Plugging this into (\ref{eqn}) gives 
			\begin{equation}\label{eqnn}
				\begin{aligned}
					f_{i}(b)-f_{i}(a)&=\dual{\nabla f_{i}(a),b-a}+\dual{b-a,\int_{0}^{1}\int_{0}^{1}\nabla^{2}f_{i}(a+st(b-a))ds(t(b-a))dt}.
				\end{aligned}
			\end{equation}
		
		By substituting $b=x^{k+1},~a=x^{k}$ and $b=x^{*},~a=x^{k}$ into (\ref{eqnn}), respectively, we have
	\begin{align*}
		0&\leq F_{i}(x^{k+1}) - F_{i}(x^{*})\\
		&= (f_{i}(x^{k+1})-f_{i}(x^{k})) - (f_{i}(x^*)-f_{i}(x^{k}))+g_{i}(x^{k+1})-g_{i}(x^*)\\
		&=\dual{\nabla f_{i}(x^{k}),x^{k+1}-x^{k}}+\dual{x^{k+1}-x^{k},\int_{0}^{1}\int_{0}^{1}\nabla^{2}f_{i}(x^{k}+st(x^{k+1}-x^{k}))ds(t(x^{k+1}-x^{k}))dt}\\
		&~~~~+\dual{\nabla f_{i}(x^{k}),x^{k}-x^*}-\dual{x^*-x^{k},\int_{0}^{1}\int_{0}^{1}\nabla^{2}f_{i}(x^{k}+st(x^*-x^{k}))ds(t(x^*-x^{k}))dt}\\
		&~~~~+g_{i}(x^{k+1})-g_{i}(x^*)\\
		&=\dual{\nabla f_{i}(x^{k}),x^{k+1}-x^*}+\dual{x^{k+1}-x^{k},\int_{0}^{1}\int_{0}^{1}\nabla^{2}f_{i}(x^{k}+st(x^{k+1}-x^{k}))ds(t(x^{k+1}-x^{k}))dt}\\
		&~~~~-\dual{x^*-x^{k},\int_{0}^{1}\int_{0}^{1}\nabla^{2}f_{i}(x^{k}+st(x^*-x^{k}))ds(t(x^*-x^{k}))dt}+g_{i}(x^{k+1})-g_{i}(x^*).
	\end{align*}
On the other hand, from (\ref{subgrad}), we have $$-\nabla f_{{\lambda^{k}}/{\alpha^{k}}}(x^{k}) - B_{k}d^{k}\in \partial g_{{\lambda^{k}}/{\alpha^{k}}}(x^k+d^k).$$ This, together with the fact that $t_k=1$, implies
\begin{align*}
	g_{{\lambda^{k}}/{\alpha^{k}}}(x^{k+1})-g_{{\lambda^{k}}/{\alpha^{k}}}(x^*)&\leq\dual{-\nabla f_{{\lambda^{k}}/{\alpha^{k}}}(x^{k})-B_{k}d^{k},x^{k+1}-x^*}\\
	&=\dual{-\nabla f_{{\lambda^{k}}/{\alpha^{k}}}(x^{k})-B_k(x^{k+1}-x^{k}),x^{k+1}-x^*}.
\end{align*}
By substituting the preceding relation, we have
	\begin{align*}
	0&\leq\dual{-B_{k}(x^{k+1}-x^{k}),x^{k+1}-x^{*}}\\
		&~~~~+\dual{x^{k+1}-x^{k},\int_{0}^{1}\int_{0}^{1}\nabla^{2}f_{{\lambda^{k}}/{\alpha^{k}}}(x^{k}+st(x^{k+1}-x^{k}))ds(t(x^{k+1}-x^{k}))dt}\\
		&~~~~-\dual{x^*-x^{k},\int_{0}^{1}\int_{0}^{1}\nabla^{2}f_{{\lambda^{k}}/{\alpha^{k}}}(x^{k}+st(x^*-x^{k}))ds(t(x^*-x^{k}))dt}.
	\end{align*}	
Then there exist $\bar{x}^{k}_{1}\in[x^{k},x^{k+1}]$ (line segment between $x^{k}$ and $x^{k+1}$) and $\bar{x}^{k}_{2}\in[x^{k},x^{*}]$ such that 

		\begin{align*}
			0&\leq\dual{B_{k}(x^{k}-x^{k+1}),x^{k+1}-x^{*}}+\frac{1}{2}\nm{x^{k+1}-x^{k}}^{2}_{\nabla^{2}f_{{\lambda^{k}}/{\alpha^{k}}}(\bar{x}^{k}_{1})}-\frac{1}{2}\nm{x^{k}-x^{*}}^{2}_{\nabla^{2}f_{{\lambda^{k}}/{\alpha^{k}}}(\bar{x}^{k}_{2})}\\
			&=\dual{B_{k}(x^{k}-x^{k+1}),x^{k+1}-x^{*}}+\frac{1}{2}\nm{x^{k+1}-x^{k}}^{2}_{\nabla^{2}f_{{\lambda^{k}}/{\alpha^{k}}}(\bar{x}^{k}_{1})}-\frac{1}{2}\nm{x^{k}-x^{k+1}+x^{k+1}-x^{*}}^{2}_{\nabla^{2}f_{{\lambda^{k}}/{\alpha^{k}}}(\bar{x}^{k}_{2})}\\
			&=\dual{\left(B_{k}-\nabla^{2}f_{{\lambda^{k}}/{\alpha^{k}}}(\bar{x}^{k}_{2})\right)(x^{k}-x^{k+1}),x^{k+1}-x^{*}}+\frac{1}{2}\nm{x^{k+1}-x^{k}}^{2}_{\nabla^{2}f_{{\lambda^{k}}/{\alpha^{k}}}(\bar{x}^{k}_{1})}\\
			&~~~~-\frac{1}{2}\nm{x^{k+1}-x^{k}}^{2}_{\nabla^{2}f_{{\lambda^{k}}/{\alpha^{k}}}(\bar{x}^{k}_{2})}-\frac{1}{2}\nm{x^{k+1}-x^{*}}^{2}_{\nabla^{2}f_{{\lambda^{k}}/{\alpha^{k}}}(\bar{x}^{k}_{2})},
		\end{align*}
	 Without loss of generality, for any $\epsilon>0$, there exists $k_{\epsilon}$ such that, for all $k\geq k_{\epsilon}$ and $j\in\{1,2\}$,
	\begin{equation}\label{e39}
		\nm{\left(\nabla^{2}f_{{\lambda^{k}}/{\alpha^{k}}}(\bar{x}^{k}_{j})-B_{k}\right)s_{k}}\overset{}{\leq}\epsilon\|s_{k}\|,
	\end{equation}
	and
	\begin{equation}\label{e40}
		\nm{\left(\nabla^{2}f_{{\lambda^{k}}/{\alpha^{k}}}(\bar{x}^{k}_{1})-\nabla^{2}f_{{\lambda^{k}}/{\alpha^{k}}}(\bar{x}^{k}_{2})\right)s_{k}}\leq\epsilon\|s_{k}\|,
	\end{equation}
	where (\ref{e39}) is given by assumptions (c), (e) and (f), (\ref{e40}) follows by (c) and (e).
	Then, we use relations (\ref{e39}) and (\ref{e40}) to get
	\begin{equation}\label{e41}
		\begin{aligned}
			&~~~~\nm{x^{k+1}-x^{*}}^{2}_{\nabla^{2}f_{{\lambda^{k}}/{\alpha^{k}}}(\bar{x}^{k}_{2})}\\
			&\leq2\dual{\left(B_{k}-\nabla^{2}f_{{\lambda^{k}}/{\alpha^{k}}}(\bar{x}^{k}_{2})\right)(x^{k}-x^{k+1}),x^{k+1}-x^{*}}\\
			&~~~~+\dual{\left(\nabla^{2}f_{{\lambda^{k}}/{\alpha^{k}}}(\bar{x}^{k}_{1})-\nabla^{2}f_{{\lambda^{k}}/{\alpha^{k}}}(\bar{x}^{k}_{2})\right)(x^{k+1}-x^{k}),x^{k+1}-x^{k}}\\
			&\leq2\epsilon\nm{s_{k}}\nm{x^{k+1}-x^{*}} + \epsilon\nm{s_{k}}^{2}.
		\end{aligned}
	\end{equation}
	On the other hand, by assumptions (c), (d) and (e), we have
	$$\nm{x^{k+1}-x^{*}}^{2}_{\nabla^{2}f_{{\lambda^{k}}/{\alpha^{k}}}(\bar{x}^{k}_{2})}\geq a\nm{x^{k+1}-x^{*}}^{2}.$$
	Rearranging and substituting the above relation into (\ref{e41}), we obtain
	$$a\nm{x^{k+1}-x^{*}}^{2}-2\epsilon\nm{x^{k+1}-x^{*}}\nm{s_{k}}-\epsilon\nm{s_{k}}^{2}\leq0.$$
	Dividing by $\nm{s_{k}}^{2}$, it is easy to get $$\frac{\nm{x^{k+1}-x^{*}}}{\nm{s_{k}}}\in\left[\frac{\epsilon-\sqrt{\epsilon^{2}+a\epsilon}}{a},\frac{\epsilon+\sqrt{\epsilon^{2}+a\epsilon}}{a}\right].$$
	{Being $\epsilon>0$ arbitrary}, it follows that
	$$\lim\limits_{k\rightarrow\infty}\frac{\nm{x^{k+1}-x^{*}}}{\nm{s_{k}}}=0.$$
	Notice that $\|s_{k}\|\leq\nm{x^{k+1}-x^{*}}+\nm{x^{k}-x^{*}}$, then
	$$0\leq\lim\limits_{k\rightarrow\infty}\frac{\nm{x^{k+1}-x^{*}}}{\nm{x^{k+1}-x^{*}}+\nm{x^{k}-x^{*}}}\leq\lim\limits_{k\rightarrow\infty}\frac{\nm{x^{k+1}-x^{*}}}{\nm{s_{k}}}=0.$$
	{It follows that
		$$\lim\limits_{k\rightarrow\infty}\frac{\nm{x^{k+1}-x^{*}}}{\nm{x^{k+1}-x^{*}}+\nm{x^{k}-x^{*}}}=0.$$
	Dividing by $\nm{x^{k}-x^{*}}$, we have
$$\lim\limits_{k\rightarrow\infty}\frac{\frac{\nm{x^{k+1}-x^{*}}}{\nm{x^{k}-x^{*}}}}{\frac{\nm{x^{k+1}-x^{*}}}{\nm{x^{k}-x^{*}}}+1}=0.$$}
Therefore,	
	$$\lim\limits_{k\rightarrow\infty}\frac{\nm{x^{k+1}-x^{*}}}{\nm{x^{k}-x^{*}}}=0,$$
	and hence the rate of convergence is superlinear.
\end{proof}
\begin{remark}
	The proposed method enjoys fast asymptotic convergence comparable to that of proximal quasi-Newton methods \cite{PRY2024}. However, it has two notable advantages. First, it only requires constructing a single quasi-Newton matrix at each iteration, while proximal quasi-Newton methods generally construct multiple quasi-Newton matrices, one for each objective. Second, the fast local convergence of proximal quasi-Newton methods usually requires the Hessian matrices of all objective functions to be positive definite at the optimal solution. By contrast, the fast asymptotic convergence of the proposed method is guaranteed under the weaker assumption that the aggregated Hessian matrix is positive definite; see assumption (d) in Theorem \ref{sup}.
\end{remark}

\section{Subspace framework of preconditioned proximal Barzilai-Borwein method}\label{sec4}
When $B_{k}$ is non-diagonal, the proximal operator $$\mathrm{prox}^{B_{k}}_{ g_{{\lambda^{k}}/{\alpha^{k}}}}\left(x^{k}-B_k^{-1}\nabla f_{{\lambda^{k}}/{\alpha^{k}}}(x^{k})\right)$$
is generally not available in closed form.

To further enhance the performance of the algorithm, we incorporate a subspace acceleration mechanism. Within this framework, two key questions naturally arise. The first concerns how to construct a subspace that effectively captures useful curvature and descent information from past iterates. The second concerns how to design an efficient subspace model so that the resulting subproblem can be solved rapidly while maintaining good approximation quality. These issues will be addressed in the following subsections.
\subsection{Selection of subspace}
To exploit historical information while keeping the computational cost low, we construct a low-dimensional subspace that captures useful search directions from recent iterations. In particular, for $k>1$ we define the two-dimensional subspace

$$\mathcal{L}_{k}=\mathtt{span}\{v_{k},u_k\},$$
where $v_k$ represents the scaled proximal gradient direction (\ref{bbk}) and $u_k$ incorporates information from the previous step. Specifically, the direction $u_k$ is defined as 

\begin{equation}\label{uk}
	u_{k}:=\mathrm{\Pi}_{\cap_{i\in[m]}\mathrm{dom}(g_{i})}(x^{k}+s_{k-1})-x^{k}.
\end{equation}

\begin{remark}
	If $\mathrm{dom}(g_i)=\mathbb{R}^{n}$, then $u_{k}=s_{k-1}$. When $\mathrm{dom}(g_i)\neq\mathbb{R}^{n}$, as in constrained optimization problems, the direction $s_{k-1}$ may become infeasible at the point $x^{k}$. In this case, we set
	$$\mathrm{\Pi}_{\cap_{i\in[m]}\mathrm{dom}(g_{i})}(x^{k}+s_{k-1})-x^{k}.$$ 
\end{remark}
\subsection{Selection of approximate model}
Having constructed the subspace $\mathcal{L}_k$, we next define the corresponding subspace model used to refine the search direction. 
Restricting the step to $\mathcal{L}_k$, we consider the following subspace preconditioned proximal Barzilai-Borwein subproblem:
\begin{equation}\label{sn}
	\begin{aligned}
		\min_{d\in \mathcal{L}_{k}}\max\limits_{i\in[m]}\left\{\frac{
			\left\langle\nabla f_{i}(x^{k}),d\right\rangle + g_{i}(x^{k}+d)-g_{i}(x^{k})}{\alpha_{i}^{k}}+\frac{1}{2}\|d\|_{B_{k}}^{2}\right\},     
	\end{aligned}
\end{equation}
where $\alpha^{k}$ is defined as (\ref{alpha_k}) and $B_{k}\approx \nabla^{2} f_{{\lambda^{k}}/{\alpha^{k}}}(x^{k})$.
Since $\mathcal{L}_k=\mathrm{span}\{v_k,u_k\}$ is two-dimensional, any $d\in\mathcal{L}_k$ can be written as $d=G_kz$, where $G_k=[v_k,u_k]$ and $z\in\mathbb{R}^2$. Substituting this representation into \eqref{sn} yields the equivalent two-dimensional optimization  problem
\begin{equation}\label{subsp}
	\begin{aligned}
		\min_{z\in\mathbb{R}^{2}}\max_{i\in[m]}\frac{\left\langle\nabla f_{i}(x^{k}),G_{k}z\right\rangle+g_{i}(x^{k}+G_{k}z)-g_{i}(x^{k})}{\alpha^{k}_{i}}+\frac{1}{2}\|z\|^{2}_{Q_{k}}   ,
	\end{aligned}
\end{equation}
where $Q_{k}=\begin{bmatrix}
	\dual{v_{k},B_{k}v_{k}}   & &~ & \dual{v_{k},B_{k}u_{k}} \\
	
	\dual{v_{k},B_{k}u_{k}}& &~  &\dual{u_{k},B_{k}u_{k}}
\end{bmatrix}$.
\subsection{Decomposition of subspace subproblem}
Although problem \eqref{subsp} is only two-dimensional, obtaining its exact solution may still be nontrivial due to the presence of the nonsmooth terms $g_{i}(x^k+G_k\alpha),~i\in[m]$. In many practical situations, computing the exact minimizer is unnecessary and may introduce additional computational overhead. Therefore, instead of solving \eqref{subsp} exactly, we aim to construct an efficient approximation of the minimizer.

To this end, we exploit the structure of the subspace model and perform optimization along carefully chosen directions. In particular, by transforming the basis of the subspace into a conjugate basis with respect to the $B_k$-inner product, the quadratic term becomes diagonal, which enables efficient alternating one-dimensional optimization.

Specifically, we orthogonalize $u_k$ with respect to $v_k$ under the $B_k$-inner product and define
\begin{equation}
	\tilde{u}_k = u_{k} - \frac{\dual{u_{k},B_kv_k}}{\dual{v_{k},B_kv_k}}  v_k.
\end{equation}

With this construction we have $v_k^{\top}B_k\tilde{u}_k=0$. Consequently, problem \eqref{subsp} can be rewritten in the equivalent form
\begin{equation}\label{subss}
	\begin{aligned}
		\min_{z\in\mathbb{R}^{2}}\max_{i\in[m]}\frac{\left\langle\nabla f_{i}(x^{k}),\tilde{G}_{k}z\right\rangle+g_{i}(x^{k}+\tilde{G}_{k}z)-g_{i}(x^{k})}{\alpha^{k}_{i}}+\frac{1}{2}\|z\|^{2}_{\tilde{Q}_{k}}   ,
	\end{aligned}
\end{equation}
where $\tilde{G}_{k}=[v_{k},\tilde{u}_{k}]$ and $\tilde{Q}_{k}=\begin{bmatrix}
	\dual{v_{k},B_{k}v_{k}}   & &~ & 0 \\
	
	0& &~  &\dual{\tilde{u}_{k},B_{k}\tilde{u}_{k}}
\end{bmatrix}$.
Instead of solving \eqref{subss} directly, 
we consider the following subproblem 
\begin{equation}\label{subsss}
	\begin{aligned}
		\min_{z\in\mathbb{R}^{2}}\max_{i\in[m]}\frac{\left\langle\nabla f_{i}(x^{k}),\tilde{G}_{k}z\right\rangle+(g_{i}(x^{k}+z_{1}v_{k})-g_{i}(x^{k}))+(g_{i}(x^{k}+z_{2}\tilde{u}_{k})-g_{i}(x^{k}))}{\alpha^{k}_{i}}+\frac{1}{2}\|z\|^{2}_{\tilde{Q}_{k}}   ,
	\end{aligned}
\end{equation}

The remaining question is how to solve (\ref{subsss}). We first write the duality of (\ref{subsss}) as follows
\begin{align*}\tag{D}\label{D}
	-&\min\limits_{\lambda}\omega_{S}^{k}(\lambda)\\
	&\mathrm{ s.t.} \ \lambda\in\Delta_{m},
\end{align*}
where 
\begin{equation}\label{ome}
\omega_{S}^{k}(\lambda):=-\min_{z\in\mathbb{R}^{2}}\dual{\frac{\lambda}{\alpha^{k}},Jf(x^{k})\tilde{G}_{k}z+g(x^{k}+z_{1}v_{k})+g(x^{k}+z_{2}\tilde{u}_{k})-2g(x^{k})}+\frac{1}{2}\|z\|^{2}_{\tilde{Q}_{k}}.	
\end{equation}
By Danskin's theorem, the gradient of $\omega_{S}^{k}$ can be written as follows:
\begin{align*}
	\nabla\omega_{S}^{k}(\lambda)&=\frac{2g(x^{k})-Jf(x^{k})\tilde{G}_{k}z(\lambda)-g(x^{k}+z_{1}(\lambda)v_{k})-g(x^{k}+z_{2}(\lambda)\tilde{u}_{k})}{\alpha^k},
\end{align*}
where $z_{1}(\lambda)$ and $z_{2}(\lambda)$ are the optimal solutions to the following one-dimensional subproblems, respectively,
\begin{equation}\label{subss1}
	\begin{aligned}
		\min_{z_{1}\in\mathbb{R}}\dual{\frac{\lambda}{\alpha^{k}},z_{1}Jf(x^{k})v_{k}+g(x^{k}+z_{1}v_{k})}+\frac{1}{2}\dual{v_{k},B_{k}v_{k}}(z_{1})^{2}.
	\end{aligned}
\end{equation}
and
\begin{equation}\label{subss2}
	\begin{aligned}
		\min_{z_{2}\in\mathbb{R}}\dual{\frac{\lambda}{\alpha^{k}},z_{2}Jf(x^{k})\tilde{u}_{k}+g(x^{k}+z_{2}\tilde{u}_{k})}+\frac{1}{2}\dual{\tilde{u}_{k},B_{k}\tilde{u}_{k}}(z_{2})^{2}.
	\end{aligned}
\end{equation}
It remains to compute the Hessian--vector products $B_{k}v_{k}$, $B_{k}\tilde{u}_{k}$ and $B_{k}s_{k-1}$ in (\ref{alpha_k}). 
To avoid explicitly forming the Hessian matrix, we approximate the Hessian--vector products using finite differences of gradients. Specifically, we use
\begin{equation}\label{fd}
	B_{k}v\approx \frac{1}{\epsilon}(\nabla f_{\lambda^{k-1}/\alpha^{k-1}}(x^{k}+\epsilon v)-\nabla f_{\lambda^{k-1}/\alpha^{k-1}}(x^{k})).
\end{equation}
Based on this approximation, the one-dimensional subproblems \eqref{subss1} and \eqref{subss2} can be reformulated as
\begin{equation}\label{subs1}
	\begin{aligned}
		\min_{z_{1}\in\mathbb{R}}\dual{\frac{\lambda}{\alpha^{k}},z_{1}Jf(x^{k})v_{k}+g(x^{k}+z_{1}v_{k})}+\frac{q_{k}(v_{k})}{2}\nm{v_{k}}^{2}(z_{1})^{2}.
	\end{aligned}
\end{equation}
and 
\begin{equation}\label{subs2}
	\begin{aligned}
		\min_{z_{2}\in\mathbb{R}}\dual{\frac{\lambda}{\alpha^{k}},z_{2}Jf(x^{k})\tilde{u}_{k}+g(x^{k}+z_{2}\tilde{u}_{k})}+\frac{q_{k}(\tilde{u}_{k})}{2}\nm{\tilde{u}_{k}}^{2}(z_{2})^{2}.
	\end{aligned}
\end{equation}
where $\alpha^{k}\in\mathbb{R}^{m}_{++}$ is rewritten as follows:
\begin{equation}\label{alpha_k1}
	\alpha^{k}_{i}=\left\{
	\begin{aligned}
		&\max\left\{\alpha_{\min},\min\left\{\frac{\langle s_{k-1},y^{k-1}_{i}\rangle}{q_{k}(s_{k-1})\nm{s_{k-1}}^{2}}, \alpha_{\max}\right\}\right\}, & {\langle s_{k-1},y^{k-1}_{i}\rangle}&>0, \\
		&\max\left\{\alpha_{\min},\min\left\{\frac{\nm{y^{k-1}_{i}}}{q_{k}(s_{k-1})\nm{s_{k-1}}}, \alpha_{\max}\right\}\right\}, &{\langle s_{k-1},y^{k-1}_{i}\rangle}&<0, \\
		& \alpha_{\min}, &  \langle s_{k-1},y^{k-1}_{i}\rangle&=0,
	\end{aligned}
	\right.
\end{equation}
$q_{k}(v)\approx \dual{v,B_kv}/\nm{v}^{2},~v\in\{v_{k},\tilde{u}_{k},s_{k-1}\}$.
\par Denote $z_{1}(\lambda^{k})$ and $z_{2}(\lambda^{k})$ are the minimizers of (\ref{subs1}) and (\ref{subs2}) with $\lambda = \lambda^{k}$, respectively. Therefore, the subspace preconditioned proximal Barzilai-Borwein direction can be expressed as $$d^{k}_{S}=z_{1}(\lambda^{k})v_{k}+z_{2}(\lambda^{k})\tilde{u}_{k},$$
where $\lambda^{k}$ is a minimizer of (\ref{D}) with $\alpha^{k}$ defined as (\ref{alpha_k1}) and 
$$\tilde{Q}_{k}=\begin{bmatrix}
	q_{k}(v_{k})\nm{v_{k}}^{2}   & &~ & 0 \\
	
	0& &~  &q_{k}(\tilde{u}_{k})\nm{\tilde{u}_{k}}^{2}
\end{bmatrix}.$$
\subsection{Further discussion on approximate model}\label{sec4.4}
Consider the subspace proximal Newton subproblem:
\begin{equation}\label{subn}
	\begin{aligned}
		\min_{d\in \mathcal{L}_{k}}\max\limits_{i\in[m]}\left\{{
			\left\langle\nabla f_{i}(x^{k}),d\right\rangle + g_{i}(x^{k}+d)-g_{i}(x^{k})}+\frac{1}{2}\|d\|_{\nabla^2f_{i}(x^k)}^{2}\right\}.   
	\end{aligned}
\end{equation}
We can write the equivalent two-dimensional optimization  problem
\begin{equation}
	\begin{aligned}
		\min_{z\in \mathbb{R}^{2}}\max\limits_{\lambda\in\Delta_{m}}\left\{{
			\left\langle\nabla f_{\lambda}(x^{k}),G_{k}z\right\rangle + g_{\lambda}(x^{k}+G_{k}z)-g_{\lambda}(x^{k})}+\frac{1}{2}\|z\|_{G_k^{\top}\nabla^2f_{\lambda}(x^k)G_{k}}^{2}\right\}.   
	\end{aligned}
\end{equation}
It is worth noting that $G_k^{\top}\nabla^2 f_{\lambda}(x^k)G_k$
varies with $\lambda$. Hence, it cannot be used to construct a fixed
conjugate basis.
\begin{remark}
	The conjugate transformation of basis plays a crucial role in decomposing and simplifying
	the subproblem. In our setting, this transformation is determined by the matrix
	$B_k$, which highlights the role of the ``preconditioning'' $+$ ``preconditioning'' mechanism in multiobjective composite optimization.
	Specifically, a single common preconditioning matrix $B_k$ is used to capture
	the curvature information shared by all objectives. In contrast, the proximal
	Newton method employs objective-specific preconditioning matrices, typically
	$\nabla^2 f_i(x^k)$, to capture the curvature information of each objective
	separately. Consequently, there is generally no single known matrix
	that can be used to obtain an analogous decomposition. 
\end{remark}
\subsection{Extension to multiobjective composite optimization with linear operator}
Consider the multiobjective composite optimization problem involving a linear operator:
\begin{align*}
	\min_{x\in\mathbb{R}^{n}} F^{A}(x):=
	\big(f_{1}(x)+g_{1}(Ax),\ldots,f_{m}(x)+g_{m}(Ax)\big)^{\top},\tag{MCOP$_A$}\label{MCOPA}
\end{align*}
where $A\in\mathbb{R}^{p\times n}$. When $A=I$, problem (\ref{MCOPA}) reduces to (\ref{MCOP}).
\subsubsection{The subspace subproblems}
Let $g=\tilde{g}:=g\circ A$. Then, after replacing
$g$ by $\tilde{g}$, the analysis developed for the subspace method remains
valid. Accordingly, the subproblems~(\ref{subs1}) and~(\ref{subs2}) become
\begin{equation}\label{subsss1}
	\begin{aligned}
		\min_{z_{1}\in\mathbb{R}}\quad
		&\dual{\frac{\lambda}{\alpha^{k}},
			z_{1}Jf(x^{k})v_{k}
			+g(Ax^{k}+z_{1}Av_{k})}
		+\frac{q_{k}(v_{k})}{2}\nm{v_{k}}^{2}z_{1}^{2},
	\end{aligned}
\end{equation}
and
\begin{equation}\label{subsss2}
	\begin{aligned}
		\min_{z_{2}\in\mathbb{R}}\quad
		&\dual{\frac{\lambda}{\alpha^{k}},
			z_{2}Jf(x^{k})\tilde{u}_{k}
			+g(Ax^{k}+z_{2}A\tilde{u}_{k})}
		+\frac{q_{k}(\tilde{u}_{k})}{2}\nm{\tilde{u}_{k}}^{2}z_{2}^{2}.
	\end{aligned}
\end{equation}
The additional computational burden introduced by the linear operator $A$ is
therefore reduced to the matrix-vector products $Av_{k}$ and
$A\tilde{u}_{k}$ in~(\ref{subsss1}) and~(\ref{subsss2}), respectively.
\subsubsection{Linear-operator-aware preconditioning}
 It remains to determine how to compute a descent direction for constructing
 the low-dimensional subspace. Motivated by Chen and Yang~\cite{CY2026}, we introduce
 the following linear-operator-aware preconditioned proximal Barzilai-Borwein
 subproblem:
 \begin{equation}\label{ppd}
 	\min\limits_{x\in\mathbb{R}^{n}}\max\limits_{i\in[m]}\left\{\frac{
 		\left\langle\nabla f_{i}(x^{k}),x-x^k\right\rangle + g_{i}(Ax)-g_{i}(Ax^{k})}{\tilde{\alpha}_{i}^{k}}+\frac{1}{2}\|x-x^k\|_{P}^{2}\right\},
 \end{equation}
 where $\tilde{\alpha}^{k}\in\mathbb{R}^{m}_{++}$ is as follows:
 \begin{equation}\label{talpha}
 	\tilde{\alpha}^{k}_{i}=\left\{
 	\begin{aligned}
 		&\max\left\{\alpha_{\min},\min\left\{\frac{\langle s_{k-1},y^{k-1}_{i}\rangle}{\nm{s_{k-1}}^{2}_{P}}, \alpha_{\max}\right\}\right\}, & \langle s_{k-1},y^{k-1}_{i}\rangle&>0, \\
 		&\max\left\{\alpha_{\min},\min\left\{\frac{\nm{y^{k-1}_{i}}}{\nm{Ps_{k-1}}}, \alpha_{\max}\right\}\right\}, & \langle s_{k-1},y^{k-1}_{i}\rangle&<0, \\
 		& \alpha_{\min}, &  \langle s_{k-1},y^{k-1}_{i}\rangle&=0.
 	\end{aligned}
 	\right.
 \end{equation}
 The subproblem can be rewritten as 
 \begin{equation}\label{ppbb}
 	\min\limits_{x\in\mathbb{R}^{n}}\max\limits_{\lambda\in\Delta_{m}}\left\{{
 		\left\langle\nabla f_{\lambda/\tilde{\alpha}^{k}}(x^{k}),x-x^k\right\rangle + g_{\lambda/\tilde{\alpha}^{k}}(Ax)-g_{\lambda/\tilde{\alpha}^{k}}(Ax^{k})}+\frac{1}{2}\|x-x^k\|_{P}^{2}\right\}.
 \end{equation}
 By Sion's minimax theorem, the dual problem can be expressed as follows:
 \begin{align*}\tag{DP}\label{DP}
 	-&\min\limits_{\lambda}\omega_{P}^{k}(\lambda)\\
 	&\mathrm{ s.t.} \ \lambda\in\Delta_{m},
 \end{align*}
 where 
 \begin{equation}\label{e37}
 	\omega_{P}^{k}(\lambda):=-\min\limits_{x\in\mathbb{R}^{n}}\left\{{
 		\left\langle\nabla f_{\lambda/\tilde{\alpha}^{k}}(x^{k}),x-x^k\right\rangle + g_{\lambda/\tilde{\alpha}^{k}}(Ax)-g_{\lambda/\tilde{\alpha}^{k}}(Ax^{k})}+\frac{1}{2}\|x-x^k\|_{P}^{2}\right\}.	
 \end{equation}
 By Danskin's theorem, the gradient of $\omega_{P}^{k}$ can be written as follows:
 \begin{align*}
 	\nabla\omega_{P}^{k}(\lambda)&=\frac{g(Ax^{k})-g(Ax(\lambda))-Jf(x^{k})(x(\lambda)-x^{k})}{\tilde{\alpha}^k},
 \end{align*}
 where 
 \begin{equation}\label{xlam}
 	x(\lambda):=\mathrm{prox}^{P}_{{ g\circ A_{{\lambda}/{\tilde\alpha^{k}}}}}\left(x^{k}-P^{-1}\nabla f_{{\lambda}/{\tilde\alpha^{k}}}(x^{k})\right),
 \end{equation}
 is the optimal solution of (\ref{e37}). 
 \par Next, we revisit the construction of $P$ such that (\ref{e37}) admits a closed form solution. To analyze this problem, we introduce its saddle-point formulation
 $$\min_{x\in\mathbb{R}^{n}}\max_{y\in\mathbb{R}^{p}}\dual{\nabla f_{\lambda/\tilde{\alpha}^{k}}(x^{k}),x-x^{k}}+\frac{1}{2}\|x-x^{k}\|^{2}_{P}+\dual{y, Ax}-g_{\lambda/\tilde{\alpha}^{k}}^{*}(y),$$
 where $g^*$ is the convex conjugate of $g$. By the minimax theorem, the problem can be equivalently written as
 \begin{align*}
 	\max_{y\in\mathbb{R}^{p}}\min_{x\in\mathbb{R}^{n}}\dual{\nabla f_{\lambda/\tilde{\alpha}^{k}}(x^{k}),x-x^{k}}+\frac{1}{2}\|x-x^{k}\|^{2}_{P}+\dual{y, Ax}-g_{\lambda/\tilde{\alpha}^{k}}^{*}(y).
 \end{align*}
Recall that $x(\lambda)$ is the optimal solution to (\ref{e37}). 
 The optimality condition with respect to $x$ yields
 \begin{equation}\label{x}
 	x(\lambda)=x^{k}-P^{-1}(\nabla f_{\lambda/\tilde{\alpha}^{k}}(x^{k})+A^{\top}y^{k}),
 \end{equation}
 where $y^{k}$ is the optimal solution of the following dual problem:
 \begin{equation}\label{dual}
 	\min_{y\in\mathbb{R}^{p}}\frac{1}{2}\nm{y}^{2}_{AP^{-1}A}+g_{\lambda/\tilde{\alpha}^{k}}^*(y) -\dual{a^{k},y},
 \end{equation}
 with $$a^{k}:=Ax^{k}-AP^{-1}\nabla f_{\lambda/\tilde{\alpha}^{k}}(x^{k}).$$
 \par The main motivation behind the preconditioned proximal gradient method is to simplify the dual subproblem through a proper choice of the preconditioner $P$. In particular, if $P$ is chosen such that 
 \begin{equation}\label{inv}
 	AP^{-1}A^{\top}=\bm I_{p},
 \end{equation}
 then the dual problem \eqref{dual} reduces to
 \[
 \min_{y\in\mathbb{R}^{p}}
 \frac12\|y\|^2 + g_{\lambda/\tilde{\alpha}^{k}}^*(y) - a^{k\top}y,
 \]
 whose solution is simply
 $$y^{k}=\mathrm{prox}_{g_{\lambda/\tilde{\alpha}^{k}}^*}(a^{k}).$$
 Therefore, the key question becomes how to construct a suitable preconditioner $P$ that satisfies condition \eqref{inv}.
  For full row rank matrix $A$, let the SVD of $A$ be
$$A = U\Lambda V^{\top},$$
where $U\in\mathbb{O}^{p}$ and $V\in\mathbb{O}^{n}$. Then
$$A^{\top}A=V\Lambda^{\top}\Lambda V^{\top}.$$
Based on this decomposition, we select the preconditioner $P$ as
\begin{equation}\label{p}
	P=\left\{
	\begin{aligned}
		&A^{\top}A, & \mathcal{R}(A)&= n, \\
		&A^{\top}A+ V\begin{bmatrix}
			\bm0_{p\times p}   & &  \\
			
			&  &\tilde{P}
		\end{bmatrix}V^{\top}, &  \mathcal{R}(A)&=p<n,
	\end{aligned}
	\right.
\end{equation}
where $\tilde{P}\in\mathbb{S}^{n-p}_{++}$. 
%
\begin{remark}
Several closed form preconditioned proximal operators have been studied in \cite{CY2026}, including those for ellipsoidally constrained problems, structured $\ell_{1}$ regularization problems, and linearly constrained optimization problems. We refer the reader to \cite[Section~4]{CY2026} for further details.
\end{remark}
\section{Inexact subspace preconditioned Barzilai-Borwein proximal gradient method}\label{sec5}
The subproblems (\ref{D}) with $g=\tilde{g}$ and (\ref{DP}) with $P=P_{k}\succ0$ are both simplex constrained problems, which can be solved by the projected gradient method. However, finding the exact global minimizers of these subproblems at each iteration is computationally expensive and generally unnecessary for the overall convergence. Instead, we adopt an inexact strategy.
Denote
\begin{equation}
	\mathcal{D}_{PBB}^{k}(\lambda):=\max\limits_{i\in[m]}\left\{\frac{
		\left\langle\nabla f_{i}(x^{k}),x(\lambda)-x^{k}\right\rangle + g_{i}(Ax(\lambda))-g_{i}(Ax^{k})}{\tilde{\alpha}_{i}^{k}}\right\},
\end{equation}
where $x(\lambda)$ is defined as (\ref{xlam}) with $P=P_{k}$. Denote
\begin{equation}\label{theta}
	\theta(x^{k}):=\min\limits_{v\in\mathbb{R}^{n}}\max\limits_{\lambda\in\Delta_{m}}\dual{\nabla f_{\lambda/\tilde\alpha^{k}}(x^{k}),v}+g_{\lambda/\tilde\alpha^{k}}(Ax^{k}+Av)-g_{\lambda/\tilde\alpha^{k}}(Ax^{k})+\frac{1}{2}\| v\|_{P_{k}}^{2},
\end{equation}
the optimal value of (\ref{ppbb}) with $P=P_{k}$.
Denote $\lambda^{k}_{PBB}$ an optimal solution of (\ref{DP}). By Lemma \ref{ll2}, we have 
$$\mathcal{D}^{k}_{PBB}(\lambda_{PBB}^{k})\leq-\|x(\lambda^{k}_{PBB})-x^{k}\|_{P_{k}}^{2}.$$
\subsection{Approximate descent directions}
Now, we define the following $\epsilon$-approximate descent direction to construct subspace. 
\begin{definition}\label{defa1}
	Let $\epsilon\in[0,1)$. A vector $v_{k,\epsilon}\in\mathbb{R}^{n}$ is called $\epsilon$-approximate linear-operator-aware preconditioned proximal Barzilai-Borwein descent direction of (\ref{MCOPA}) at $x^{k}$, if there exists $\lambda_{PBB}^{k,\epsilon}\in\Delta_{m}$ such that
	\begin{equation}\label{app1}
		\begin{aligned}
			\mathcal{D}_{PBB}^{k}(\lambda_{PBB}^{k,\epsilon})\leq(1-\epsilon)\left(\dual{\nabla f_{\lambda_{PBB}^{k,\epsilon}/\tilde\alpha^{k}}(x^k),v_{k,\epsilon}}+g_{\lambda_{PBB}^{k,\epsilon}/\tilde\alpha^{k}}(Ax^{k}+Av_{k,\epsilon})-g_{\lambda_{PBB}^{k,\epsilon}/\tilde\alpha^{k}}(Ax^{k})\right),
		\end{aligned}
	\end{equation}
	where \begin{equation}\label{vke}
		v_{k,\epsilon}:=x(\lambda_{PBB}^{k,\epsilon})-x^{k}.
	\end{equation}
\end{definition}
\par Denote
\begin{equation}\label{thetae}
	\theta_{\epsilon}(x^{k}):=\dual{\nabla f_{\lambda_{PBB}^{k,\epsilon}/\tilde\alpha^{k}}(x^k),v_{k,\epsilon}}+g_{\lambda_{PBB}^{k,\epsilon}/\tilde\alpha^{k}}(Ax^{k}+Av_{k,\epsilon})-g_{\lambda_{PBB}^{k,\epsilon}/\tilde\alpha^{k}}(Ax^{k})+\frac{1}{2}\nm{v_{k,\epsilon}}_{P_{k}}^{2},
\end{equation}
the $\epsilon$-approximate optimal value of (\ref{ppbb}) with $P=P_{k}$.
\par After obtaining the $\epsilon$-approximate linear-operator-aware preconditioned proximal Barzilai-Borwein descent direction, we use it to construct the subspace $\mathtt{span}\{v_{k,\epsilon},\tilde{u}_{k,\epsilon}\}$, where 
$$\tilde{u}_{k,\epsilon} = u_{k} - \frac{\dual{u_{k},B_kv_{k,\epsilon}}}{\dual{v_{k,\epsilon},B_kv_{k,\epsilon}}}  v_{k,\epsilon}.$$
To get a closed form projection in constrained cases, the $u_{k}$ defined in (\ref{uk}) changes to
\begin{equation}\label{puk}
	u_{k}:=\mathrm{\Pi}^{P}_{\cap_{i\in[m]}\mathrm{dom}(\tilde{g}_{i})}(x^{k}+s_{k-1})-x^{k}.
\end{equation}
The corresponding subspace dual subproblem (\ref{D}) changes to
\begin{align*}\tag{D$_{\epsilon}$}\label{De}
	-&\min\limits_{\lambda}\omega_{S}^{k,\epsilon}(\lambda)\\
	&\mathrm{ s.t.} \ \lambda\in\Delta_{m},
\end{align*}
where the objective function $\omega_{S}^{k,\epsilon}(\lambda)$ is defined as
\begin{equation}\label{dualobj}
	\begin{aligned}
		-\min_{z\in\mathbb{R}^{2}}\dual{\frac{\lambda}{\alpha^{k}},Jf(x^{k})\tilde{G}_{k,\epsilon}z+g(Ax^{k}+z_{1}Av_{k,\epsilon})+g(Ax^{k}+z_{2}A\tilde{u}_{k,\epsilon})-2g(Ax^{k})}+\frac{1}{2}\|z\|^{2}_{\tilde{Q}_{k,\epsilon}},	
	\end{aligned}
\end{equation}
where $\tilde{G}_{k,\epsilon}=[v_{k,\epsilon},\tilde{u}_{k,\epsilon}]$ and $\tilde{Q}_{k}=\begin{bmatrix}
	q_{k}(v_{k,\epsilon})\nm{v_{k,\epsilon}}^{2}   & &~ & 0 \\
	
	0& &~  &q_{k}(\tilde{u}_{k,\epsilon})\nm{\tilde{u}_{k,\epsilon}}^{2}
\end{bmatrix}$.
By Danskin's theorem, the gradient of $\omega_{S}^{k,\epsilon}$ can be written as follows:
\begin{align*}
	\nabla\omega_{S}^{k,\epsilon}(\lambda)&=\frac{2g(Ax^{k})-Jf(x^{k})\tilde{G}_{k,\epsilon}z(\lambda)-g(Ax^{k}+z_{1}(\lambda)Av_{k,\epsilon})-g(Ax^{k}+z_{2}(\lambda)A\tilde{u}_{k,\epsilon})}{\alpha^k},
\end{align*}
where $z_{1}(\lambda)$ and $z_{2}(\lambda)$ are the optimal solutions to the following one-dimensional subproblems, respectively,
\begin{equation}\label{subssss1}
	\begin{aligned}
		\min_{z_{1}\in\mathbb{R}}\quad
		&\dual{\frac{\lambda}{\alpha^{k}},
			z_{1}Jf(x^{k})v_{k,\epsilon}
			+g(Ax^{k}+z_{1}Av_{k,\epsilon})}
		+\frac{q_{k}(v_{k,\epsilon})}{2}\nm{v_{k,\epsilon}}^{2}z_{1}^{2},
	\end{aligned}
\end{equation}
and
\begin{equation}\label{subssss2}
	\begin{aligned}
		\min_{z_{2}\in\mathbb{R}}\quad
		&\dual{\frac{\lambda}{\alpha^{k}},
			z_{2}Jf(x^{k})\tilde{u}_{k,\epsilon}
			+g(Ax^{k}+z_{2}A\tilde{u}_{k,\epsilon})}
		+\frac{q_{k}(\tilde{u}_{k,\epsilon})}{2}\nm{\tilde{u}_{k,\epsilon}}^{2}z_{2}^{2}.
	\end{aligned}
\end{equation}
Denote
\begin{equation}
	\mathcal{D}^{k}(\lambda):=\max_{i\in[m]}\frac{\dual{\nabla f_{i}(x^{k}),\tilde{G}_{k,\epsilon}z(\lambda)}+g_{i}(Ax^{k}+z_{1}(\lambda)Av_{k,\epsilon})+g_{i}(Ax^{k}+z_{2}(\lambda)A\tilde{u}_{k,\epsilon})-2g_{i}(Ax^{k})}{\alpha^{k}_{i}},
\end{equation}
where $z(\lambda)=(z_{1}(\lambda),z_{2}(\lambda))$, $z_{1}(\lambda)$ and $z_{2}(\lambda)$ are the minimizers of (\ref{subssss1}) and (\ref{subssss2}), respectively. 
Denote $\lambda^{k}$ an optimal solution of (\ref{De}). By Lemma \ref{ll2}, we have 
$$\mathcal{D}^{k}(\lambda^{k})\leq-\|z(\lambda^{k})\|^{2}_{\tilde{Q}_{k,\epsilon}}.$$
Next, we define the following $(\epsilon,\delta)$-approximate subspace descent direction. 
\begin{definition}\label{defa2}
	Let $\epsilon,\delta\in[0,1)$. A vector $d_{S}^{k,\epsilon,\delta}\in\mathbb{R}^{n}$ is called $(\epsilon,\delta)$-approximate subspace descent direction of (\ref{MCOPA}) at $x^{k}$, if there exists $\lambda^{k,\epsilon,\delta}\in\Delta_{m}$ such that
	\begin{equation}\label{app2}
		\begin{aligned}
			\mathcal{D}^{k}(\lambda^{k,\epsilon,\delta})\leq(1-\delta)\big(&\dual{\nabla f_{\lambda^{k,\epsilon,\delta}/\alpha^{k}}(x^{k}),d_{S}^{k,\epsilon,\delta}}+ g_{\lambda^{k,\epsilon,\delta}/\alpha^{k}}\left(Ax^{k}+z_{1}(\lambda^{k,\epsilon,\delta})Av_{k,\epsilon}\right)\\
			&+g_{\lambda^{k,\epsilon,\delta}/\alpha^{k}}\left(Ax^{k}+z_{2}(\lambda^{k,\epsilon,\delta})A\tilde{u}_{k,\epsilon}\right)-2g_{\lambda^{k,\epsilon,\delta}/\alpha^{k}}(Ax^k)\big),
		\end{aligned}
	\end{equation}
	where
	\begin{equation}\label{ds}
		d^{k,\epsilon,\delta}_{S}:=z_{1}(\lambda^{k,\epsilon,\delta})v_{k,\epsilon}+z_{2}(\lambda^{k,\epsilon,\delta})\tilde{u}_{k,\epsilon}.
	\end{equation}
\end{definition}

%
Some properties of the approximate descent directions are given as follows.

\begin{proposition}\label{p4.1}
	Let $\epsilon,\delta\in[0,1)$ and $P=P_{k}\succeq c_{3}\bm I$, the following statements hold.
	\begin{itemize}
		\item[(i)] $v_{k,\epsilon}$ satisfies
		\begin{equation}\label{e57}
			\dual{\nabla f_{\lambda_{PBB}^{k,\epsilon}/\tilde\alpha^{k}},v_{k,\epsilon}}+g_{\lambda_{PBB}^{k,\epsilon}/\tilde\alpha^{k}}(Ax^{k}+Av_{k,\epsilon})-g_{\lambda_{PBB}^{k,\epsilon}/\tilde\alpha^{k}}(Ax^{k})\leq - \nm{v_{k,\epsilon}}_{P_{k}}^{2};
		\end{equation}
	\item[(ii)] If $\epsilon=0$, then $\lambda^{k,\epsilon}_{PBB}$ is a solution of (\ref{DP}), and hence approximate descent direction defined in Definitions~\ref{defa1} coincides with the corresponding exact descent direction;
	\item[(iii)] $\theta_{\epsilon}(x^{k})\leq\theta(x^{k})\leq(1-\epsilon)^{2}\theta_{\epsilon}(x^{k})$;
	\item[(iv)] $v_{k,\epsilon}\rightarrow0$ if and only if $v^{P}_{k}\rightarrow 0$, where $v^{P}_{k}$ is the optimal solution of the minimization problem in (\ref{theta});
		\item[(v)] $z_{1}(\lambda^{k,\epsilon,\delta})$ satisfies
		\begin{equation}\label{e58}
			\begin{aligned}
				\dual{\nabla f_{\lambda^{k,\epsilon,\delta}/\alpha^{k}}(x^{k}),z_{1}(\lambda^{k,\epsilon,\delta})v_{k,\epsilon}}&+ g_{\lambda^{k,\epsilon,\delta}/\alpha^{k}}\left(Ax^{k}+z_{1}(\lambda^{k,\epsilon,\delta})Av_{k,\epsilon}\right)-g_{\lambda^{k,\epsilon,\delta}/\alpha^{k}}(Ax^k)\\
				&\leq - {q_{k}(v_{k,\epsilon})}\nm{v_{k,\epsilon}}^{2}(z_{1}(\lambda^{k,\epsilon,\delta}))^{2};
			\end{aligned}
		\end{equation}
	\item[(vi)] $z_{2}(\lambda^{k,\epsilon,\delta})$ satisfiies
	\begin{equation}\label{e55}
		\begin{aligned}
			\dual{\nabla f_{\lambda^{k,\epsilon,\delta}/\alpha^{k}}(x^{k}),z_{2}(\lambda^{k,\epsilon,\delta})\tilde{u}_{k,\epsilon}}&+ g_{\lambda^{k,\epsilon,\delta}/\alpha^{k}}\left(Ax^{k}+z_{2}(\lambda^{k,\epsilon,\delta})A\tilde{u}_{k,\epsilon}\right)-g_{\lambda^{k,\epsilon,\delta}/\alpha^{k}}(Ax^k)\\
			&\leq - {q_{k}(\tilde{u}_{k,\epsilon})}\nm{\tilde{u}_{k,\epsilon}}^{2}(z_{2}(\lambda^{k,\epsilon,\delta}))^{2};
		\end{aligned}
	\end{equation}
\item[(vii)] If $\delta=0$, then $\lambda^{k,\epsilon,\delta}$ is a solution of (\ref{De}), and hence the approximate descent direction defined in Definitions~\ref{defa2} coincides with the corresponding exact descent direction.
	\end{itemize}
\end{proposition}
\begin{proof}
	(i) The assertion is a consequence of Lemma \ref{ll2}.
	\par (ii) For $\epsilon = 0$, Definition \ref{defa1} gives
	\begin{align*}
		&\max\limits_{i\in[m]}\left\{\frac{
			\left\langle\nabla f_{i}(x^{k}),x(\lambda_{PBB}^{k,\epsilon})-x^{k}\right\rangle + g_{i}(Ax(\lambda_{PBB}^{k,\epsilon}))-g_{i}(Ax^{k})}{\tilde{\alpha}_{i}^{k}}\right\}\\
		&\leq \dual{\nabla f_{\lambda_{PBB}^{k,\epsilon}/\tilde\alpha^{k}}(x^k),x(\lambda_{PBB}^{k,\epsilon})-x^{k}}+g_{\lambda_{PBB}^{k,\epsilon}/\tilde\alpha^{k}}(Ax(\lambda_{PBB}^{k,\epsilon}))-g_{\lambda_{PBB}^{k,\epsilon}/\tilde\alpha^{k}}(Ax^{k}),
	\end{align*}
and hence the equality holds. Adding $\frac{1}{2}\nm{x(\lambda_{PBB}^{k,\epsilon})-x^{k}}_{P_{k}}^{2}$ on both sides, we have 
 \begin{align*}
 	&\min\limits_{x\in\mathbb{R}^{n}}\max_{\lambda\in\Delta_{m}}\dual{\nabla f_{\lambda/\tilde\alpha^{k}}(x^k),x-x^{k}}+g_{\lambda/\tilde\alpha^{k}}(Ax)-g_{\lambda/\tilde\alpha^{k}}(Ax^{k})+\frac{1}{2}\nm{x-x^{k}}_{P_{k}}^{2}\\
 	&\leq\max\limits_{i\in[m]}\left\{\frac{
 		\left\langle\nabla f_{i}(x^{k}),x(\lambda_{PBB}^{k,\epsilon})-x^{k}\right\rangle + g_{i}(Ax(\lambda_{PBB}^{k,\epsilon}))-g_{i}(Ax^{k})}{\tilde{\alpha}_{i}^{k}}\right\}+\frac{1}{2}\nm{x(\lambda_{PBB}^{k,\epsilon})-x^{k}}_{P_{k}}^{2}\\
 	&= \dual{\nabla f_{\lambda_{PBB}^{k,\epsilon}/\tilde\alpha^{k}}(x^k),x(\lambda_{PBB}^{k,\epsilon})-x^{k}}+g_{\lambda_{PBB}^{k,\epsilon}/\tilde\alpha^{k}}(Ax(\lambda_{PBB}^{k,\epsilon}))-g_{\lambda_{PBB}^{k,\epsilon}/\tilde\alpha^{k}}(Ax^{k})+\frac{1}{2}\nm{x(\lambda_{PBB}^{k,\epsilon})-x^{k}}_{P_{k}}^{2}\\
 	&=\min\limits_{x\in\mathbb{R}^{n}}\dual{\nabla f_{\lambda_{PBB}^{k,\epsilon}/\tilde\alpha^{k}}(x^k),x-x^{k}}+g_{\lambda_{PBB}^{k,\epsilon}/\tilde\alpha^{k}}(Ax)-g_{\lambda_{PBB}^{k,\epsilon}/\tilde\alpha^{k}}(Ax^{k})+\frac{1}{2}\nm{x-x^{k}}_{P_{k}}^{2}\\
 	&\leq\max_{\lambda\in\Delta_{m}}\min\limits_{x\in\mathbb{R}^{n}}\dual{\nabla f_{\lambda/\tilde\alpha^{k}}(x^k),x-x^{k}}+g_{\lambda/\tilde\alpha^{k}}(Ax)-g_{\lambda/\tilde\alpha^{k}}(Ax^{k})+\frac{1}{2}\nm{x-x^{k}}_{P_{k}}^{2},
 \end{align*}
where the second equality follows by the definition of $x(\lambda_{PBB}^{k,\epsilon})$. By Sion's minimax theorem, the above equality holds. Therefore, assertion (ii) holds.
\par (iii) By the definitions of $\theta(x^{k})$ and $\theta_{\epsilon}(x^{k})$, we have
\begin{align*}
	\theta_{\epsilon}(x^{k})&=\dual{\nabla f_{\lambda_{PBB}^{k,\epsilon}/\tilde\alpha^{k}}(x^k),v_{k,\epsilon}}+g_{\lambda_{PBB}^{k,\epsilon}/\tilde\alpha^{k}}(Ax^{k}+Av_{k,\epsilon})-g_{\lambda_{PBB}^{k,\epsilon}/\tilde\alpha^{k}}(Ax^{k})+\frac{1}{2}\nm{v_{k,\epsilon}}_{P_{k}}^{2}\\
	&=\min\limits_{v\in\mathbb{R}^{n}}\dual{\nabla f_{\lambda_{PBB}^{k,\epsilon}/\tilde\alpha^{k}}(x^k),v}+g_{\lambda_{PBB}^{k,\epsilon}/\tilde\alpha^{k}}(Ax^{k}+Av)-g_{\lambda_{PBB}^{k,\epsilon}/\tilde\alpha^{k}}(Ax^{k})+\frac{1}{2}\nm{v}_{P_{k}}^{2}\\
	&\leq\min\limits_{v\in\mathbb{R}^{n}}\max\limits_{\lambda\in\Delta_{m}}\dual{\nabla f_{\lambda/\tilde\alpha^{k}}(x^k),v}+g_{\lambda/\tilde\alpha^{k}}(Ax^{k}+Av)-g_{\lambda/\tilde\alpha^{k}}(Ax^{k})+\frac{1}{2}\|v\|_{P_{k}}^{2}\\
	&=\theta(x^{k}),
\end{align*}
where the second equality follows by the definition of $v_{k,\epsilon}$. Next, we prove the right-hand side of the inequality. By the definitions of $\theta(x^{k})$ and $\theta_{\epsilon}(x^{k})$, we have
\begin{align*}
	&\theta(x^{k})=\min\limits_{v\in\mathbb{R}^{n}}\max\limits_{\lambda\in\Delta_{m}}\dual{\nabla f_{\lambda/\tilde\alpha^{k}}(x^k),v}+g_{\lambda/\tilde\alpha^{k}}(Ax^{k}+Av)-g_{\lambda/\tilde\alpha^{k}}(Ax^{k})+\frac{1}{2}\| v\|_{P_{k}}^{2}\\
	&=\min\limits_{v\in\mathbb{R}^{n}}\max\limits_{\lambda\in\Delta_{m}}\dual{\nabla f_{\lambda/\tilde\alpha^{k}}(x^k),(1-\epsilon)v}+g_{\lambda/\tilde\alpha^{k}}(Ax^{k}+(1-\epsilon)Av)-g_{\lambda/\tilde\alpha^{k}}(Ax^{k})+\frac{1}{2}\| (1-\epsilon)v\|_{P_{k}}^{2}\\
	&\leq\min\limits_{v\in\mathbb{R}^{n}}\max\limits_{\lambda\in\Delta_{m}}(1-\epsilon)\left(\dual{\nabla f_{\lambda/\tilde\alpha^{k}}(x^k),v}+g_{\lambda/\tilde\alpha^{k}}(Ax^{k}+Av)-g_{\lambda/\tilde\alpha^{k}}(Ax^{k})\right)+\frac{(1-\epsilon)^{2}}{2}\| v\|_{P_{k}}^{2}\\
	&\leq\max\limits_{\lambda\in\Delta_{m}}(1-\epsilon)\left(\dual{\nabla f_{\lambda/\tilde\alpha^{k}}(x^k),v_{k,\epsilon}}+g_{\lambda/\tilde\alpha^{k}}(Ax^{k}+Av_{k,\epsilon})-g_{\lambda/\tilde\alpha^{k}}(Ax^{k})\right)+\frac{(1-\epsilon)^{2}}{2}\| v_{k,\epsilon}\|_{P_{k}}^{2}\\
	&\leq(1-\epsilon)^{2}\left(\dual{\nabla f_{\lambda_{PBB}^{k,\epsilon}/\tilde\alpha^{k}}(x^k),v_{k,\epsilon}}+g_{\lambda_{PBB}^{k,\epsilon}/\tilde\alpha^{k}}(Ax^{k}+Av_{k,\epsilon})-g_{\lambda_{PBB}^{k,\epsilon}/\tilde\alpha^{k}}(Ax^{k})+\frac{1}{2}\| v_{k,\epsilon}\|_{P_{k}}^{2}\right)\\
	&=(1-\epsilon)^{2}\theta_{\epsilon}(x^{k}),
\end{align*}
where the first inequality follows by $1-\epsilon\in(0,1]$ and the convexity of $g_{i}$, and the third inequality is due to the relation (\ref{app1}).  
\par (iv) Notice that 
$$\theta(x^{k})=\dual{\nabla f_{\lambda_{PBB}^{k}/\tilde\alpha^{k}}(x^k),v_{k}^{P}}+g_{\lambda_{PBB}^{k}/\tilde\alpha^{k}}(Ax^{k}+Av_{k}^{P})-g_{\lambda_{PBB}^{k}/\tilde\alpha^{k}}(Ax^{k})+\frac{1}{2}\| v_{k}^{P}\|_{P_{k}}^{2},$$
and
$$\theta_{\epsilon}(x^{k})=\dual{\nabla f_{\lambda_{PBB}^{k,\epsilon}/\tilde\alpha^{k}}(x^k),v_{k,\epsilon}}+g_{\lambda_{PBB}^{k,\epsilon}/\tilde\alpha^{k}}(Ax^{k}+Av_{k,\epsilon})-g_{\lambda_{PBB}^{k,\epsilon}/\tilde\alpha^{k}}(Ax^{k})+\frac{1}{2}\nm{v_{k,\epsilon}}_{P_{k}}^{2}.$$
Therefore, the assertion is a consequence of assertion (iii) and the continuity of $g$.
 \par(v) and (vi) The assertions can be obtained by Lemma \ref{ll2}.
\par(vii)  The assertions can be obtained by using the similar arguments as in the proof of assertion (ii).                                                                                                                                                                                                                                                                                                                                                                                                                                                                                                                                                                                                                                                                                                                                                 
\end{proof}
\par The remaining question is whether the search direction $d_{S}^{k,\epsilon,\delta}$ satisfies a sufficient descent condition required for the global convergence analysis. We provide such a sufficient condition in the following proposition.
\begin{proposition}[Sufficient descent conditions]\label{suco}
	Let $\epsilon,\delta\in[0,1)$ and assume that there exist constants $0<c_{1}\leq c_{2}$ and $c_{3}>0$ such that
	$c_{1}\leq q_{k}(v_{k,\epsilon})\leq c_{2}$, $q_{k}(\tilde{u}_{k,\epsilon})\geq c_{1}$ and $P=P_{k}\succeq c_{3}\bm{I}_{n}$ in (\ref{subssss1}), (\ref{subssss2}) and (\ref{e37}) for all $k$. Then, the search direction $d_{S}^{k,\epsilon,\delta}$ defined in (\ref{ds}) satisfies the following conditions:  
	\begin{equation}\label{con1}
		\dual{\nabla f_{i}(x^{k}), \frac{1}{2}d_{S}^{k,\epsilon,\delta}}+g_{i}\left(Ax^{k}+\frac{1}{2} Ad_{S}^{k,\epsilon,\delta}\right)-g_{i}(Ax^k)\leq-\frac{c_{1}(1-\delta)\alpha^{k}_{i}}{4}\nm{d_{S}^{k,\epsilon,\delta}}^{2},
	\end{equation}
and
	\begin{equation}\label{con2}
		\begin{aligned}
			&\dual{\nabla f_{i}(x^{k}), \frac{1}{2}d_{S}^{k,\epsilon,\delta}}+g_{i}\left(Ax^{k}+\frac{1}{2} Ad_{S}^{k,\epsilon,\delta}\right)-g_{i}(Ax^k)\\
			&\leq c\alpha^{k}_{i}\left(\dual{\nabla f_{\lambda_{PBB}^{k,\epsilon}/\tilde{\alpha}^{k}}(x^{k}),v_{k,\epsilon}}+ g_{\lambda_{PBB}^{k,\epsilon}/\tilde{\alpha}^{k}}\left(Ax^{k}+Av_{k,\epsilon}\right)-g_{\lambda_{PBB}^{k,\epsilon}/\tilde{\alpha}^{k}}(Ax^k)\right).
		\end{aligned}
	\end{equation}
where
\begin{equation}\label{c}
	c:=\frac{(1-\delta)(1-\epsilon){\alpha_{\min}}}{4{\alpha_{\max}}}\min\left\{\frac{(1-\epsilon)c_{3}\alpha_{\min}}{c_{2}\alpha_{\max}},1\right\}.
\end{equation} 
\end{proposition}
\begin{proof}
	By the definition of $d_{S}^{k,\delta}$, we obtain 
	\begin{equation}\label{econ1}\small
		\begin{alignedat}{2}
			&\quad &&
			\max_{\lambda\in\Delta_{m}}\bigg\{
			\dual{\nabla f_{\lambda/\alpha^{k}}(x^{k}), d_{S}^{k,\epsilon,\delta}}
			+2g_{\lambda/\alpha^{k}}\left(Ax^{k}+\frac{1}{2}Ad_{S}^{k,\epsilon,\delta}\right)
			-2g_{\lambda/\alpha^{k}}(Ax^k)
			\bigg\} \\
			&~~~~~~~= {} &&
			\max_{\lambda\in\Delta_{m}}\bigg\{
			\dual{\nabla f_{\lambda/\alpha^{k}}(x^{k}), d_{S}^{k,\epsilon,\delta}}
			+2g_{\lambda/\alpha^{k}}\left(
			Ax^{k}
			+\frac{1}{2}z_{1}(\lambda^{k,\epsilon,\delta})Av_{k,\epsilon}
			+\frac{1}{2}z_{2}(\lambda^{k,\epsilon,\delta})A\tilde{u}_{k,\epsilon}
			\right) \\
			& &&\qquad
			-2g_{\lambda/\alpha^{k}}(Ax^k)
			\bigg\} \\
			&\overset{\mathrm{convexity~of}~g}{\leq} {} &&
			\max_{\lambda\in\Delta_{m}}\bigg\{
			\dual{\nabla f_{\lambda/\alpha^{k}}(x^{k}),d_{S}^{k,\epsilon,\delta}}
			+g_{\lambda/\alpha^{k}}
			\left(Ax^{k}+z_{1}(\lambda^{k,\epsilon,\delta})Av_{k,\epsilon}\right) \\
			& &&\qquad
			+g_{\lambda/\alpha^{k}}
			\left(Ax^{k}+z_{2}(\lambda^{k,\epsilon,\delta})A\tilde{u}_{k,\epsilon}\right)
			-2g_{\lambda/\alpha^{k}}(Ax^k)
			\bigg\} \\
			&~~~~~~\overset{(\ref{app2})}{\leq} {} &&
			(1-\delta)\bigg(
			\dual{\nabla f_{\lambda^{k,\epsilon,\delta}/\alpha^{k}}(x^{k}),d_{S}^{k,\epsilon,\delta}}
			+g_{\lambda^{k,\epsilon,\delta}/\alpha^{k}}
			\left(Ax^{k}+z_{1}(\lambda^{k,\epsilon,\delta})Av_{k,\epsilon}\right) \\
			& &&\qquad
			+g_{\lambda^{k,\epsilon,\delta}/\alpha^{k}}
			\left(Ax^{k}+z_{2}(\lambda^{k,\epsilon,\delta})A\tilde{u}_{k,\epsilon}\right)
			-2g_{\lambda^{k,\epsilon,\delta}/\alpha^{k}}(Ax^k)
			\bigg) \\
			&~\overset{(\ref{e58})\,\mathrm{and}\,(\ref{e55})}{\leq} {} &&
			-(1-\delta)
			\left(
			q_{k}(v_{k,\epsilon})
			\nm{z_{1}(\lambda^{k,\epsilon,\delta})v_{k,\epsilon}}^{2}
			+q_{k}(\tilde{u}_{k,\epsilon})
			\nm{z_{2}(\lambda^{k,\epsilon,\delta})\tilde{u}_{k,\epsilon}}^{2}
			\right) \\
			&~~~~~~~\leq {} &&
			-\frac{c_{1}(1-\delta)}{2}
			\left(
			\nm{
				z_{1}(\lambda^{k,\epsilon,\delta})v_{k,\epsilon}
				+z_{2}(\lambda^{k,\epsilon,\delta})\tilde{u}_{k,\epsilon}
			}^{2}
			\right) \\
			&~~~~~~~= {} &&
			-\frac{c_{1}(1-\delta)}{2}
			\nm{d_{S}^{k,\epsilon,\delta}}^{2}.
		\end{alignedat}
	\end{equation}
 Dividing both sides by $2$ yields~\eqref{con1}.
	\par Next we prove \eqref{con2}. Since $z_{1}(\lambda^{k,\epsilon,\delta})$ is the minimizer of \eqref{subssss1}, we have
\begin{equation*}\small
	\begin{alignedat}{2}
		&~\quad &&2\left(\max_{\lambda\in\Delta_{m}}\dual{\nabla f_{\lambda/\alpha^{k}}(x^{k}),\frac12d_{S}^{k,\epsilon,\delta}}+g_{\lambda/\alpha^{k}}\left(Ax^{k}+\frac{1}{2}Ad_{S}^{k,\epsilon,\delta}\right)-g_{\lambda/\alpha^{k}}(Ax^k)\right)\\
		&\overset{(\ref{econ1})}{\leq}\quad ~~&& (1-\delta)\big(\dual{\nabla f_{\lambda^{k,\epsilon,\delta}/\alpha^{k}}(x^{k}),d_{S}^{k,\epsilon,\delta}}+ g_{\lambda^{k,\epsilon,\delta}/\alpha^{k}}\left(Ax^{k}+z_{1}(\lambda^{k,\epsilon,\delta})Av_{k,\epsilon}\right)\\
		&~\quad &&+g_{\lambda^{k,\epsilon,\delta}/\alpha^{k}}\left(Ax^{k}+z_{2}(\lambda^{k,\epsilon,\delta})A\tilde{u}_{k,\epsilon}\right)-2g_{\lambda^{k,\epsilon,\delta}/\alpha^{k}}(Ax^k)\big)\\
		&\overset{(\ref{e55})}{\leq}\quad && (1-\delta)\left(\dual{\nabla f_{\lambda^{k,\epsilon,\delta}/\alpha^{k}}(x^{k}),z_{1}(\lambda^{k,\epsilon,\delta})v_{k,\epsilon}}+ g_{\lambda^{k,\epsilon,\delta}/\alpha^{k}}\left(Ax^{k}+z_{1}(\lambda^{k,\epsilon,\delta})Av_{k,\epsilon}\right)-g_{\lambda^{k,\epsilon,\delta}/\alpha^{k}}(Ax^k)\right)\\
		&~~{\leq}\quad && (1-\delta)\big(\dual{\nabla f_{\lambda^{k,\epsilon,\delta}/\alpha^{k}}(x^{k}),z_{1}(\lambda^{k,\epsilon,\delta})v_{k,\epsilon}}+ g_{\lambda^{k,\epsilon,\delta}/\alpha^{k}}\left(Ax^{k}+z_{1}(\lambda^{k,\epsilon,\delta})Av_{k,\epsilon}\right)-g_{\lambda^{k,\epsilon,\delta}/\alpha^{k}}(Ax^k)\\
		&~\quad &&+\frac{q_{k}(v_{k,\epsilon})}{2}\nm{v_{k,\epsilon}}^{2}(z_{1}(\lambda^{k,\epsilon,\delta}))^{2}\big)\\
		&\overset{\mathclap{(\mathrm{any} ~0\leq z_{1}\leq1)}}{\leq}\quad && (1-\delta)\big(\dual{\nabla f_{\lambda^{k,\epsilon,\delta}/\alpha^{k}}(x^{k}),z_{1}v_{k,\epsilon}}+ g_{\lambda^{k,\epsilon,\delta}/\alpha^{k}}\left(Ax^{k}+z_{1}Av_{k,\epsilon}\right)-g_{\lambda^{k,\epsilon,\delta}/\alpha^{k}}(Ax^k)\\
		&~\quad &&+\frac{q_{k}(v_{k,\epsilon})}{2}\nm{v_{k,\epsilon}}^{2}(z_{1})^{2}\big)\\
		&\overset{\mathclap{(\mathrm{convexity~of~}g)}}{\leq}\quad &&(1-\delta)\big(z_{1}\left( \dual{\nabla f_{\lambda^{k,\epsilon,\delta}/\alpha^{k}}(x^{k}),v_{k,\epsilon}}+ g_{\lambda^{k,\epsilon,\delta}/\alpha^{k}}\left(Ax^{k}+Av_{k,\epsilon}\right)-g_{\lambda^{k,\epsilon,\delta}/\alpha^{k}}(Ax^k)\right)\\
		&~\quad &&+\frac{q_{k}(v_{k,\epsilon})}{2}\nm{v_{k,\epsilon}}^{2}(z_{1})^{2}\big)\\
		&{\hspace{1mm}\leq}\quad &&(1-\delta)\big(\max_{\lambda\in\Delta_{m}}z_{1}\left( \dual{\nabla f_{\lambda/\alpha^{k}}(x^{k}),v_{k,\epsilon}}+ g_{\lambda/\alpha^{k}}\left(Ax^{k}+Av_{k,\epsilon}\right)-g_{\lambda/\alpha^{k}}(Ax^k)\right)+\frac{q_{k}(v_{k,\epsilon})}{2}\nm{v_{k,\epsilon}}^{2}(z_{1})^{2}\big)\\
		&{\hspace{1mm}\leq}\quad &&(1-\delta)\big(z_{1}\frac{\alpha_{\min}}{\alpha_{\max}}\max_{\lambda\in\Delta_{m}}\left( \dual{\nabla f_{\lambda/\tilde{\alpha}^{k}}(x^{k}),v_{k,\epsilon}}+ g_{\lambda/\tilde{\alpha}^{k}}\left(Ax^{k}+Av_{k,\epsilon}\right)-g_{\lambda/\tilde{\alpha}^{k}}(Ax^k)\right)\\
		&~\quad &&+\frac{q_{k}(v_{k,\epsilon})}{2}\nm{v_{k,\epsilon}}^{2}(z_{1})^{2}\big)\\
		&\overset{(\ref{app1})}{\leq}\quad &&(1-\delta)\big(z_{1}\frac{\alpha_{\min}}{\alpha_{\max}}(1-\epsilon)\left( \dual{\nabla f_{\lambda_{PBB}^{k,\epsilon}/\tilde{\alpha}^{k}}(x^{k}),v_{k,\epsilon}}+ g_{\lambda_{PBB}^{k,\epsilon}/\tilde{\alpha}^{k}}\left(Ax^{k}+Av_{k,\epsilon}\right)-g_{\lambda_{PBB}^{k,\epsilon}/\tilde{\alpha}^{k}}(Ax^k)\right)\\
		&~\quad &&+\frac{q_{k}(v_{k,\epsilon})}{2c_{3}}\nm{v_{k,\epsilon}}_{P_{k}}^{2}(z_{1})^{2}\big)\\
		&\overset{(\ref{e57})}{\leq}\quad &&(1-\delta)\left((1-\epsilon)\frac{\alpha_{\min}}{\alpha_{\max}}z_{1}-\frac{c_{2}}{2c_{3}}(z_{1})^{2}\right)\big( \dual{\nabla f_{\lambda_{PBB}^{k,\epsilon}/\tilde{\alpha}^{k}}(x^{k}),v_{k,\epsilon}}+ g_{\lambda_{PBB}^{k,\epsilon}/\tilde{\alpha}^{k}}\left(Ax^{k}+Av_{k,\epsilon}\right)\\
		&~\quad &&-g_{\lambda_{PBB}^{k,\epsilon}/\tilde{\alpha}^{k}}(Ax^k)\big)\\
		&\overset{\mathclap{(\mathrm{any} ~0\leq z_{1}\leq1)}}{\leq}\quad &&\frac{(1-\delta)(1-\epsilon){\alpha_{\min}}}{2{\alpha_{\max}}}\min\left\{\frac{(1-\epsilon)c_{3}\alpha_{\min}}{c_{2}\alpha_{\max}},1\right\}\big( \dual{\nabla f_{\lambda_{PBB}^{k,\epsilon}/\tilde{\alpha}^{k}}(x^{k}),v_{k,\epsilon}}+ g_{\lambda_{PBB}^{k,\epsilon}/\tilde{\alpha}^{k}}\left(Ax^{k}+Av_{k,\epsilon}\right)\\
		&~\quad &&-g_{\lambda_{PBB}^{k,\epsilon}/\tilde{\alpha}^{k}}(Ax^k)\big),
	\end{alignedat}
\end{equation*}
where the last inequality follows by the facts that
$$\max\limits_{x\in[0,1]}-ax^{2}+bx\geq\frac{b}{2}\min\left\{\frac{b}{2a},1\right\}$$
with $a=c_{2}/(2c_{3})>0,b=(1-\epsilon){\alpha_{\min}}/{\alpha_{\max}}>0$ and 
$$\dual{\nabla f_{\lambda_{PBB}^{k,\epsilon}/\tilde{\alpha}^{k}}(x^{k}),v_{k,\epsilon}}+ g_{\lambda_{PBB}^{k,\epsilon}/\tilde{\alpha}^{k}}\left(Ax^{k}+Av_{k,\epsilon}\right)-g_{\lambda_{PBB}^{k,\epsilon}/\tilde{\alpha}^{k}}(Ax^k)<0.$$
Dividing both sides by $2$ yields~\eqref{con2}.
\end{proof}

\subsection{Inexact subspace preconditioned Barzilai-Borwein proximal gradient method}
To guarantee the sufficient descent conditions stated in Proposition \ref{suco}, 
for $v\in\{v_{k,\epsilon},\tilde{u}_{k,\epsilon},s_{k-1}\}$ we define
\begin{equation}\label{bb2}
	q_k(v):=\left\{
	\begin{aligned}
		&\max\left\{c_{1},\min\left\{\frac{\dual{v,B_k(v)}}{\nm{v}^{2}}, c_{2}\right\}\right\}, & \dual{v,B_k(v)}&>0, \\
		&\max\left\{c_{1},\min\left\{\frac{\nm{B_{k}(v)}}{\nm{v}}, c_{2}\right\}\right\}, &\dual{v,B_k(v)}&<0, \\
		& c_{1}, &  \dual{v,B_k(v)}&=0,
	\end{aligned}
	\right.
\end{equation}
where 
\begin{equation}\label{fd2}
	 B_k(v):=\frac{1}{\epsilon_{k}}(\nabla f_{\lambda^{k-1,\epsilon,\delta}/\alpha^{k-1}}(x^{k}+\epsilon_{k} v)-\nabla f_{\lambda^{k-1,\epsilon,\delta}/\alpha^{k-1}}(x^{k})),
\end{equation}
 $c_{1}$ and $c_{2}$ are the positive constants introduced in Proposition \ref{suco}.

The complete inexact subspace preconditioned proximal Barzilai-Borwein method is described as follows.

\begin{algorithm}[H]  
	\caption{Inexact subspace preconditioned proximal Barzilai-Borwein method}\label{ppg} 
	\begin{algorithmic}[1]
		\REQUIRE{$Ax^0\in\cap_{i\in[m]}\mathrm{dom}(g_{i})$, $\epsilon,\delta\in[0,1)$, $0<c_{1}\leq c_{2}$, $0<c_{3}\leq c_{4}$, $\sigma,\gamma\in(0,1)$}
		\FOR{$k=0,\cdots$}
		\STATE{Update $c_{3}\bm{I}_{n}\preceq P_{k}\preceq c_{4}\bm{I}_{n}$} 
		\STATE{Update $\tilde{\alpha}^{k}_{i}$ as in (\ref{talpha}) with $P=P_{k}$, $i\in[m]$} 
		\STATE{Compute $\lambda^{k,\epsilon}_{PBB}$ such that (\ref{app1}) holds and update $v_{k,\epsilon}$ as in (\ref{vke})} 
		\IF{$v_{k,\epsilon}=0$}
		\RETURN{$x^{k}$  }
		\ELSE{
			\IF{$k=0$}
			\STATE{Set $d^{k}=v_{k,\epsilon}$, $\lambda^{k,\epsilon,\delta}=\lambda^{k,\epsilon}_{PBB}$, $\alpha^{k}=\tilde{\alpha}^{k}$}
			\ELSE{
				\STATE{Compute $u_{k}$ as in (\ref{puk})}
				\STATE{Update $B_{k}(v_{k,\epsilon})$ and $q_{k}(v_{k,\epsilon})$ as in (\ref{fd2}) and (\ref{bb2}), respectively}
				\STATE{Update $$\tilde{u}_{k,\epsilon}:=u_{k}-\frac{\dual{u_{k},B_{k}(v_{k,\epsilon})}}{q_{k}(v_{k,\epsilon})\nm{v_{k,\epsilon}}^{2}}v_{k,\epsilon}$$}
				\STATE{Update $B_{k}(\tilde{u}_{k,\epsilon})$ and $q_{k}(\tilde{u}_{k,\epsilon})$ as in (\ref{fd2}) and (\ref{bb2}), respectively}
				\STATE{Update $q_{k}(s_{k-1})$ as in (\ref{bb2}) and then update $\alpha^{k}_{i}$ as in (\ref{alpha_k1}), $i\in[m]$}
				\STATE{Compute $\lambda^{k,\epsilon,\delta}$ such that (\ref{app2}) holds and update $d^{k,\epsilon,\delta}_{S}$ as in (\ref{ds})} 
				\STATE{Update $d_{k}=\frac{1}{2}d^{k,\epsilon,\delta}_{S}$}
			}
			\ENDIF
			\STATE{Compute the stepsize $t_{k}\in(0,1]$ in the following way:
				\begin{align*}
					t_{k}:=\max\big\{\gamma^{j}:j\in\mathbb{N}\cup\{0\},F^{A}_{i}\left(x^{k}+\gamma^{j}d_{k}\right)&-F^{A}_{i}(x^{k})
					\leq \sigma\gamma^{j}\big(\dual{\nabla f_{i}(x^{k}),d_{k}} \\
					&+ g_{i}(Ax^{k}+ Ad_{k})-g_{i}(Ax^{k})\big),~i\in[m].\big\}	
				\end{align*}
			}
			\STATE{Update $x^{k+1}:=x^{k}+t_{k}d_{k}$}
		}
		\ENDIF
		\ENDFOR
	\end{algorithmic}
\end{algorithm}
\begin{remark}
	Lines 4 and 16 constitute the main computational cost of Algorithm~\ref{ppg}, since each of these steps requires solving a subproblem. Fortunately, owing to the specific choice of $P_k$, the subproblem~\eqref{DP} in Line 4 can be solved by a projected-gradient method with a closed form gradient. The subproblem~\eqref{De} in Line 16 can also be solved by a projected-gradient method. In computing its gradient, the remaining challenge lies in solving the two one-dimensional subproblems~\eqref{subssss1} and~\eqref{subssss2}. Efficient algorithms for one-dimensional constrained and $\ell_1$-regularized subproblems have been developed in~\cite{CY2026}; we refer the reader to that work for further details.
\end{remark}
\subsection{Convergence analysis}
In this section we analyze the convergence properties of the proposed algorithm.
We first show that the stepsize produced by the line-search procedure admits a uniform lower bound.
\vspace{1mm}
\begin{lemma}
	Suppose that $f_{i}$ is $L_{i}$-smooth for $i\in[m]$. The stepsize generated by Algorithm \ref{ppg} has a lower bound:
	\begin{equation}\label{lbt}
		t_{\min}:=\min\left\{{2\gamma(1-\sigma)(1-\delta)c_{1}\alpha_{\min}}/{L_{\max}},1\right\},
	\end{equation} 
where $L_{\max}:=\max\{L_{i},i\in[m]\}$.
\end{lemma}
\vspace{1mm}
\begin{proof} It suffices to consider the case $t_k<1$, in which the backtracking procedure is activated. 
	In this situation the Armijo condition is violated for the trial stepsize $t_k/\gamma$, then there exists $i_{0}\in[m]$ such that
	\begin{equation}\label{E4.4}
		F^{A}_{i_{0}}\left(x^{k}+\frac{t_{k}}{\gamma}d_{k}\right)-F^{A}_{i_{0}}(x^{k})>\sigma \frac{t_{k}}{\gamma}\big(\dual{\nabla f_{i_{0}}(x^{k}),d_{k}} + g_{i_{0}}(Ax^{k}+ Ad_{k})-g_{i_{0}}(Ax^{k})\big).
	\end{equation}
	Since $f_{i}$ is $L_{i}$-smooth, for any $i\in[m]$ we have
	\begin{equation}\label{up}
		\begin{aligned}
			&~~~F^{A}_{i}\left(x^{k}+\frac{t_{k}}{\gamma}d_{k}\right)-F^{A}_{i}(x^{k})\\
			&\leq \frac{t_{k}}{\gamma}\dual{\nabla f_{i}(x^{k}),d_{k}} + g_{i}(Ax^{k}+\frac{t_{k}}{\gamma}Ad_{k}) - g_{i}(Ax^{k})+ \frac{L_{i}}{2}\left\|\frac{t_{k}}{\gamma}d_{k}\right\|^{2}\\
			&\leq\frac{t_{k}}{\gamma}\left(\dual{\nabla f_{i}(x^{k}),d_{k}} + g_{i}(Ax^{k}+ Ad_{k})-g_{i}(Ax^{k})\right)+\frac{L_{i}}{2}\left\|\frac{t_{k}}{\gamma}d_{k}\right\|^{2},
		\end{aligned}
	\end{equation}
	where the second inequality follows from the convexity of $g$ and the fact that ${t_{k}}/{\gamma}\in(0,1]$. Combining this inequality with \eqref{E4.4} gives
	$$(\sigma - 1)\left(\dual{\nabla f_{i_{0}}(x^{k}),d_{k}} + g_{i_{0}}(Ax^{k}+ Ad_{k})-g_{i_{0}}(Ax^{k})\right)\leq{\frac{L_{i_{0}}t_{k}}{2\gamma}}\left\|d_{k}\right\|^{2}.$$
	Using condition \eqref{con1} and the fact $d_{k}=\frac{1}{2}d^{k,\epsilon,\delta}_{S}$, we obtain
	\begin{equation}\label{et}
		t_{k}\geq\frac{2\gamma(1-\sigma)(1-\delta)c_{1}\alpha^{k}_{i_{0}}}{L_{i_0}}.
	\end{equation}
	Therefore $t_k \ge t_{\min}$, which completes the proof. \qed
\end{proof}

To establish global convergence, we impose the following standard assumption on the objective function.

\begin{assumption}\label{a2}
	For any $x^{0}\mathbb\in\mathrm{dom}F^{A}$, the level set $\mathcal{L}_{F^{A}}(x^0):=\{x:F^{A}(x)\preceq F^{A}(x^0)\}$ is compact.
\end{assumption}

Under this assumption we can prove the global convergence of the proposed algorithm.

\begin{theorem}
	Suppose that Assumption \ref{a2} holds and $f_{i}$ is $L_{i}$-smooth for $i\in[m]$. Let $\{x^{k}\}$ be the sequence of points generated by Algorithm \ref{ppg}. Then $\{x^k\}$ has at least one accumulation point, and any accumulation point $x^*$ is a Pareto critical point.
\end{theorem}
\begin{proof}
	By Armijo line search, we deduce that $\{F^{A}(x^{k})\}$ is monotone decreasing and 
	\begin{equation}\label{des}
		\begin{aligned}
			&F_{i}^{A}(x^{k+1}) - F_{i}^{A}(x^{k})\\
			&\leq \sigma t_{k} \left(\dual{\nabla f_{i}(x^{k}),d_{k}}+g_{i}(Ax^{k}+ Ad_{k})-g_{i}(Ax^{k})\right)\\
			&\leq \sigma t_{k}c\alpha^{k}_{i}\left(\dual{\nabla f_{\lambda_{PBB}^{k,\epsilon}/\tilde{\alpha}^{k}}(x^{k}),v_{k,\epsilon}}+ g_{\lambda_{PBB}^{k,\epsilon}/\tilde{\alpha}^{k}}\left(Ax^{k}+Av_{k,\epsilon}\right)-g_{\lambda_{PBB}^{k,\epsilon}/\tilde{\alpha}^{k}}(Ax^k)\right)\\
			&\leq - \sigma t_{k}c\alpha_{\min}c_{3}\nm{v_{k,\epsilon}}^{2},
		\end{aligned}	
	\end{equation}
	where the second inequality follows by relation (\ref{con2}). Therefore $x^{k}\in\mathcal{L}_{F^{A}}(x^0)$ for all $k$, and hence $\{x^{k}\}$ has at least one accumulation point $x^*$ due to the compactness of $\mathcal{L}_{F^{A}}(x^0)$. 
	In particular, there exists an infinite index set $\mathcal{K}$ such that
	\[
	\lim_{k\in\mathcal{K}}x^{k}=x^* .
	\]
	Moreover, since $F^{A}$ is lower semicontinuous and $\mathcal{L}_{F^{A}}(x^0)$ is compact, the sequence $\{F^{A}(x^{k})\}$ is bounded below. 
	Together with the monotonicity of $\{F^{A}(x^{k})\}$, this implies that $\{F^{A}(x^{k})\}$ is a Cauchy sequence. Hence
	$$\lim_{k\rightarrow\infty}F^{A}(x^{k+1})-F^{A}(x^{k})=0.$$
	Combining this limit with \eqref{des} yields
	\begin{equation}\label{lim}
		\lim_{k\rightarrow\infty}t_{k}\nm{v_{k,\epsilon}}^{2}=0.
	\end{equation}
	Together with \eqref{et}, we obtain
	$$\mathop{\lim}\limits_{k\rightarrow\infty}v_{k,\epsilon}=0.$$
	By Proposition \ref{p4.1}, it follows that 
	$$\mathop{\lim}\limits_{k\rightarrow\infty}v_{k}^{P}=0.$$
	Finally, by Proposition \ref{pop}(ii), we conclude that $x^{*}$ is a Pareto critical point. \qed
\end{proof}

Next, we further strengthen the convergence result by establishing a linear convergence rate. Before presenting the convergence result, we introduce the following error bound condition.
\begin{definition}\label{def5.1}
	The vector-valued function $F^{A}$ satisfies a global error bound, if there exists a constant $\kappa$ such that 
	\begin{equation}\label{pl}
		u_{0}(x)\leq-\kappa\theta_{\epsilon}(x),~\forall x\in\mathbb{R}^{n},
	\end{equation}
	where $$u_{0}(x):=\sup\limits_{y\in\mathbb{R}^{n}}\min\limits_{i\in[m]}\{F_{i}^{A}(x)-F_{i}^{A}(y)\}$$
	and
	$\theta_{\epsilon}(x)$ is defined as in (\ref{thetae}) with $x^{k}=x$, $\tilde\alpha^{k}_{i}=\tilde\alpha_{i}(x)\in[\alpha_{\min},\alpha_{\max}]$, $P_{k}=P(x)$ and $c_{3}\bm{I}_{n}\preceq P(x)\preceq c_{4}\bm{I}_{n}$.
\end{definition}
\begin{remark}
	By Proposition \ref{p4.1}(iii), we have
	$$\theta_{\epsilon}(x)\leq\theta(x)\leq(1-\epsilon)^{2}\theta_{\epsilon}(x),$$
	where $\theta(x)$ is defined as in (\ref{theta}) with $x^{k}=x$, $\tilde\alpha^{k}_{i}=\tilde\alpha_{i}(x)\in[\alpha_{\min},\alpha_{\max}]$, $P_{k}=P(x)$ and $c_{3}\bm{I}_{n}\preceq P(x)\preceq c_{4}\bm{I}_{n}$.
Therefore, the error bound condition (\ref{pl}) is equivalent to the multiobjective PL-inequality \cite{TFY2023b}.
	
\end{remark}

In the following, we show that strong convexity of $f$ is a sufficient condition for the Definition \ref{def5.1}.
\begin{proposition}
	If $f_{i}$ is $\mu_{i}$-strongly convex for $i\in[m]$, then for all $x\in\mathbb{R}^{n}$
	$$	u_{0}(x)\leq-\kappa\theta_{\epsilon}(x)$$ 
	holds with $\kappa=\alpha_{\max}/r$ where $r:=\min\{{\mu_{\min}}/{(c_{4}\alpha_{\max})},1\}$ and $\mu_{\min}:=\min\{\mu_{i},i\in[m]\}$.
\end{proposition}
\begin{proof}
	By the strong convexity of $f_{i}$, it follows that
	\begin{align*}
		u_{0}(x)&=\max\limits_{y\in\mathbb{R}^{n}}\min\limits_{i\in[m]}\{F_{i}^{A}(x)-F_{i}^{A}(y)\}\\
		&\leq\max\limits_{y\in\mathbb{R}^{n}}\min\limits_{i\in[m]}\{\dual{\nabla f_{i}(x),x-y}+g_{i}(Ax)-g_{i}(Ay)-\frac{\mu_{i}}{2}\nm{y-x}^{2}\}\\
		&=\max\limits_{v\in\mathbb{R}^{n}}\min\limits_{i\in[m]}\{\dual{\nabla f_{i}(x),-v}+g_{i}(Ax)-g_{i}(Ax+Av)-\frac{\mu_{i}}{2}\nm{v}^{2}\}\\
		&\leq\max\limits_{v\in\mathbb{R}^{n}}\min\limits_{i\in[m]}\{\alpha_{\max}\frac{\dual{\nabla f_{i}(x),-v}+g_{i}(Ax)-g_{i}(Ax+Av)}{\tilde{\alpha}_{i}}-\frac{\mu_{i}}{2}\nm{v}^{2}\}\\
		&\leq\alpha_{\max}\max\limits_{v\in\mathbb{R}^{n}}\min\limits_{i\in[m]}\{\frac{\dual{\nabla f_{i}(x),-v}+g_{i}(Ax)-g_{i}(Ax+Av)}{\tilde{\alpha}_{i}}\}-\frac{\mu_{\min}}{2\alpha_{\max}}\nm{v}^{2}\\
		&\leq\alpha_{\max}\max\limits_{v\in\mathbb{R}^{n}}\dual{\nabla f_{\lambda_{PBB}^{\epsilon}/\tilde{\alpha}}(x),-v}+g_{\lambda_{PBB}^{\epsilon}/\tilde{\alpha}}(Ax)-g_{\lambda_{PBB}^{\epsilon}/\tilde{\alpha}}(Ax+Av)-\frac{\mu_{\min}}{2c_{4}\alpha_{\max}}\nm{v}_{P(x)}^{2}\\
			&\leq\frac{\alpha_{\max}}{r}\max\limits_{v\in\mathbb{R}^{n}}\dual{\nabla f_{\lambda_{PBB}^{\epsilon}/\tilde{\alpha}}(x),-rv}+r\big(g_{\lambda_{PBB}^{\epsilon}/\tilde{\alpha}}(Ax)-g_{\lambda_{PBB}^{\epsilon}/\tilde{\alpha}}(Ax+Av)\big)-\frac{r^{2}}{2}\nm{v}_{P(x)}^{2}\\
			&\leq\frac{\alpha_{\max}}{r}\max\limits_{v\in\mathbb{R}^{n}}\dual{\nabla f_{\lambda_{PBB}^{\epsilon}/\tilde{\alpha}}(x),-rv}+g_{\lambda_{PBB}^{\epsilon}/\tilde{\alpha}}(Ax)-g_{\lambda_{PBB}^{\epsilon}/\tilde{\alpha}}(Ax+rAv)-\frac{r^{2}}{2}\nm{v}_{P(x)}^{2}\\
			&=\frac{\alpha_{\max}}{r}\max\limits_{v\in\mathbb{R}^{n}}\dual{\nabla f_{\lambda_{PBB}^{\epsilon}/\tilde{\alpha}}(x),-v}+g_{\lambda_{PBB}^{\epsilon}/\tilde{\alpha}}(Ax)-g_{\lambda_{PBB}^{\epsilon}/\tilde{\alpha}}(Ax+Av)-\frac{1}{2}\nm{v}_{P(x)}^{2}\\
			&=-\frac{\alpha_{\max}}{r}\left(\dual{\nabla f_{\lambda_{PBB}^{\epsilon}/\tilde{\alpha}}(x),v_{\epsilon}(x)}+g_{\lambda_{PBB}^{\epsilon}/\tilde{\alpha}}(Ax+Av_{\epsilon}(x))-g_{\lambda_{PBB}^{\epsilon}/\tilde{\alpha}}(Ax)+\frac{1}{2}\nm{v_{\epsilon}(x)}_{P(x)}^{2}\right),
	\end{align*}
where $r=\min\{{\mu_{\min}}/{(c_{4}\alpha_{\max})},1\}$ and $\mu_{\min}=\min\{\mu_{i},i\in[m]\}$.
\end{proof}
\begin{theorem}\label{t5.2}
	Suppose that $F^{A}$ satisfies Definition \ref{def5.1} and $f_{i}$ is $L_{i}$-smooth for $i\in[m]$. Let $\{x^{k}\}$ be the sequence generated by Algorithm \ref{ppg}. Then
	$$u_{0}(x^{k+1})\leq \left(1- \frac{\sigma t_{\min}c\alpha_{\min}}{\kappa}\right)u_{0}(x^{k}),$$
	where $c$, $t_{\min}$ and $\kappa$ are defined as (\ref{c}), (\ref{lbt}) and (\ref{pl}), respectively. 
	
\end{theorem}
\begin{proof}
	By direct calculation, we have
	\begin{align*}
		&F^{A}_{i}(x^{k+1}) - F^{A}_{i}(x^{k})\\
		 &\leq \sigma t_{k}c\alpha^{k}_{i}\left(\dual{\nabla f_{\lambda_{PBB}^{k,\epsilon}/\tilde{\alpha}^{k}}(x^{k}),v_{k,\epsilon}}+ g_{\lambda_{PBB}^{k,\epsilon}/\tilde{\alpha}^{k}}\left(Ax^{k}+Av_{k,\epsilon}\right)-g_{\lambda_{PBB}^{k,\epsilon}/\tilde{\alpha}^{k}}(Ax^k)\right)\\
		&\leq \sigma t_{\min}c\alpha_{\min}\left(\dual{\nabla f_{\lambda_{PBB}^{k,\epsilon}/\tilde{\alpha}^{k}}(x^{k}),v_{k,\epsilon}}+ g_{\lambda_{PBB}^{k,\epsilon}/\tilde{\alpha}^{k}}\left(Ax^{k}+Av_{k,\epsilon}\right)-g_{\lambda_{PBB}^{k,\epsilon}/\tilde{\alpha}^{k}}(Ax^k)\right)\\
		&\leq  \sigma t_{\min}c\alpha_{\min}\theta_{\epsilon}(x^{k})\\
		&\leq - \frac{\sigma t_{\min}c\alpha_{\min}}{\kappa}u_{0}(x^{k}).
	\end{align*}
	 Rearranging and taking the minimum and supremum with respect to $i\in[m]$
	and $y\in\mathbb{R}^{n}$ on both sides, respectively, we obtain
	$$\sup\limits_{y\in\mathbb{R}^{n}}\min\limits_{i\in[m]}\{F_{i}^{A}(x^{k+1})-F_{i}^{A}(y)\}\leq\sup\limits_{y\in\mathbb{R}^{n}}\min\limits_{i\in[m]}\{F_{i}^{A}(x^{k})-F_{i}^{A}(y)\} - \frac{\sigma t_{\min}c\alpha_{\min}}{\kappa}u_{0}(x^{k}).$$
	The desired result follows.
	
\end{proof}
\begin{remark}
	Assume that $f_{i}$ is $\mu_{i}$-strongly convex for $i\in[m]$ in Theorem \ref{t5.2}, the linear convergence rate of Algorithm \ref{ppg} is \begin{small}
	\begin{equation*}
	\mathcal{O}\left(\left(1-\frac{\sigma(1-\delta)(1-\epsilon){\alpha_{\min}^{2}}}{4{\alpha_{\max}^{2}}}\min\left\{\frac{(1-\epsilon)c_{3}\alpha_{\min}}{c_{2}\alpha_{\max}},1\right\}\min\left\{\frac{2\gamma(1-\sigma)(1-\delta)c_{1}\alpha_{\min}}{L_{\max}},1\right\}\min\left\{\frac{\mu_{\min}}{c_{4}\alpha_{\max}},1\right\}\right)^{k}\right).
	\end{equation*}
\end{small}
\end{remark}

\section{Numerical experiments}\label{sec6}
In this section, we report numerical results to evaluate the performance of the proposed inexact subspace preconditioned proximal Barzilai-Borwein method. All experiments were implemented in Python 3.11 and performed on a personal computer equipped with an Intel Core i9-14900HX processor and 64 GB of RAM.

We consider three classes of multiobjective composite optimization problems: $\ell_1$ regularization problems, structured $\ell_1$ regularization problems, and linear constrained problems. In all test problems, the smooth components are ill-conditioned quadratic functions of the form
$$f_{i}(x)=\frac{1}{2}\left\langle x,A_{i}x\right\rangle + \left\langle b_{i},x\right\rangle,~i=1,2,$$
where each $A_i$ is symmetric positive definite. More precisely, $A_i$ is generated as $A_i=H_iD_iH_i^\top$, where $H_i$ is a random orthogonal matrix obtained from the QR factorization of a Gaussian random matrix, and $D_i$ is a diagonal matrix whose eigenvalues are linearly distributed so that the prescribed condition number is attained. The vector $b_i$ is generated with entries uniformly distributed in $[-n,n]$. We consider five quadratic test instances, denoted by QPa--QPe, with dimensions $n=10,100,1000$ and condition numbers ranging from $10^3$ to $10^5$, as shown in Table \ref{tab1}. The second and third columns present the dimension of the variables and condition numbers, respectively. While $x_L$ and $x_U$ represent the lower and upper bounds of the variables, respectively. 
\begin{table}[H]
	\begin{center}
		\caption{Description of quadratic problems.}\label{tab1}
	\vspace{-2mm}
	\end{center}
	\small
	\centering
	\begin{tabular}{clrlclrlrl}
		\hline
		\multicolumn{1}{l}{Problem} &  & n   &           & $(\kappa_{1},\kappa_{2})$ &           & \multicolumn{1}{c}{$x_{L}$} &  & \multicolumn{1}{c}{$x_{U}$} &  \\ \hline
		QPa                         &  & 10  &           & $(10^{3},10^{3})$ &           & 10{[}-1,...,-1{]}      &  & 10{[}1,...,1{]}        &  \\
		QPb                         &  & 10  &           & $(10^{4},10^{4})$ &           & 10{[}-1,...,-1{]}      &  & 10{[}1,...,1{]}        &  \\
		QPc                         &  & 100 &           &  $(10^{4},10^{4})$ &           & 100{[}-1,...,-1{]}     &  & 100{[}1,...,1{]}       &  \\
		QPd                         &  & 100 &           & $(10^{5},10^{5})$ &           & 100{[}-1,...,-1{]}     &  & 100{[}1,...,1{]}       &  \\
		QPe                         &  & 1000 &           & $(10^{5},10^{5})$ &           & 1000{[}-1,...,-1{]}     &  & 1000{[}1,...,1{]}       &  \\ \hline
	\end{tabular}
\end{table}

We compare the following two methods:
	\begin{itemize}
		\item IPPBB: the inexact linear-operator-aware-preconditioned proximal Barzilai-Borwein method (Algorithm \ref{ppg} without subspace strategy).
		\item ISPPBB: the inexact subspace preconditioned proximal Barzilai-Borwein method (Algorithm \ref{ppg}).
	\end{itemize}
	For the parameters in tested methods, we set $c_{1}=\alpha_{\min}=10^{-3}$ and $c_{2}=\alpha_{\max}=10^{3}$ to truncate the Barzilai-Borwein's parameter. In line search, we set $\sigma_{1}=10^{-4},~\sigma_{2}=0.1$. For each problem, we used the same initial points for different tested algorithms. The initial points were randomly selected within the specified lower and upper bounds. The dual subproblems of different algorithms were solved by the spectral projected gradient method with warm start, where the dual solution obtained at the previous outer iteration was used as the initial point for the current dual subproblem. To guarantee fair comparison, we decided to let the algorithms run until one of the following stopping conditions was satisfied:
		\begin{itemize}
			\item the current solution is preconditioned $\varepsilon$-Pareto critical with $\varepsilon\leq10^{-3}$;
			\item the number of iterations reaches 2000.
	\end{itemize} 
	All reported results are averaged over 200 runs. We report the number of outer iterations, CPU time, and the average number of inner iterations required to compute $\lambda_{PBB}^{k,\epsilon}$ and $\lambda^{k,\epsilon,\delta}$. Performance profiles \cite{DM2002} based on iterations and CPU time illustrate overall performance, while the purity metric \cite{CMV2011} is employed to evaluate the quality of the obtained Pareto front. 
\subsection{$\ell_{1}$ regularization problems}
For $\ell_{1}$ regularization problems, the nonsmooth parts are described as follows:
$$g_{i}(x):=\frac{1}{n}\|x\|_{1},~i=1,2.$$
In this case, the linear operator is the identity matrix, i.e., $A=\bm I_n$, and the linear-operator-aware preconditioner reduces to $P=\bm I_n$. Thus, this group of experiments mainly tests the ability of the proposed subspace strategy to exploit curvature information for ill-conditioned nonsmooth problems without additional linear operators.
\begin{table}[h] 
	{\begin{center}
			\caption{Average number of iterations (iter), average CPU time (time ($ms$)), Average number of inner iterations required to compute $\lambda_{PBB}^{k,\epsilon}$ (iter$_{\epsilon}$) and $\lambda^{k,\epsilon,\delta}$ (iter$_{\epsilon,\delta}$) of tested algorithms on $\ell_{1}$ regularization problems.}\label{tab2}\vspace{-2mm}
		\end{center}
		\centering
		\resizebox{.99\columnwidth}{!}{
			\begin{tabular}{lrrrrrrrrrrrrrrrrr}
				\hline
				Problem &
				\multicolumn{3}{l}{IPPBB ($\epsilon=0.2$)} &
				&
				\multicolumn{4}{l}{ISPPBB ($\epsilon=\delta=0.2$)} 
				&
				\multicolumn{1}{l}{} &
				\multicolumn{3}{l}{IPPBB ($\epsilon=0.8$)}&
				\multicolumn{1}{l}{} &
				\multicolumn{4}{l}{ISPPBB ($\epsilon=\delta=0.8$)} \\ \cline{2-4} \cline{6-9} \cline{11-13} \cline{15-18}&
				\multicolumn{1}{r}{iter} &
				\multicolumn{1}{r}{time} &
				\multicolumn{1}{r}{iter$_{\epsilon}$} &
				\textbf{} &
				\multicolumn{1}{r}{iter} &
				\multicolumn{1}{r}{time} &
				\multicolumn{1}{r}{iter$_{\epsilon}$} &
				\multicolumn{1}{r}{iter$_{\epsilon,\delta}$} &
				\multicolumn{1}{r}{} &
				\multicolumn{1}{r}{iter} &
				\multicolumn{1}{r}{time} &
				\multicolumn{1}{r}{iter$_{\epsilon}$} &
				\multicolumn{1}{r}{} &
				\multicolumn{1}{r}{iter} &
				\multicolumn{1}{r}{time} &
				\multicolumn{1}{r}{iter$_{\epsilon}$} &
				\multicolumn{1}{r}{iter$_{\epsilon,\delta}$}  \\ \hline
				QPa & 181.71 & 40.16 & 1.29 &  & \textbf{83.49} & 41.15 & 1.65 & 0.91 &  & 190.12 & 45.05 & 1.18 &  & 85.23 & \textbf{37.06} & 1.52 & 0.50  \\
				QPb & 976.31 & 196.13 & 1.24 &  & \textbf{162.53} & 71.09 & 1.42 & 1.13 &  & 982.68 & 193.06 & 1.15 &  & 162.87 & \textbf{68.19} & 1.31 & 0.63  \\
				QPc & 504.12 & 130.03 & 1.71 &  & 206.51 & 144.10 & 2.96 & 1.07 &  & 484.93 & \textbf{105.27} & 1.40 &  & \textbf{198.20} & 119.50 & 2.50 & 0.54  \\
				QPd & 2000.00 & 592.44 & 0.48 &  & \textbf{383.06} & 203.93 & 1.92 & 1.25 &  & 2000.00 & 535.57 & 0.21 &  & 390.09 & \textbf{182.36} & 1.67 & 0.59  \\
				QPe & 2000.00 & 8749.81 & 0.32 &  & 445.58 & 2095.82 & 2.59 & 0.85 &  & 1997.74 & 8812.58 & 0.88 &  & \textbf{412.96} & \textbf{1834.35} & 2.30 & 0.45  \\
				\hline
			\end{tabular}
	}}
\end{table}

The results in Table \ref{tab2} show that ISPPBB substantially reduces the number of outer iterations compared with IPPBB. This improvement becomes especially significant for high-dimensional and severely ill-conditioned problems. For example, on QPd and QPe, IPPBB often reaches the maximum number of 2000 iterations, while ISPPBB terminates within several hundred iterations. When $\varepsilon=\delta=0.2$, ISPPBB requires 383.06 and 445.58 iterations on QPd and QPe, respectively; when $\varepsilon=\delta=0.8$, it requires 390.09 and 412.96 iterations, respectively. These results indicate that the subspace refinement can effectively accelerate convergence even when the nonsmooth term is non-differentiable.

In terms of CPU time, ISPPBB is also clearly advantageous on large-scale ill-conditioned instances. Although the subspace construction may introduce slight overhead on some small or medium-scale problems, the reduction in outer iterations leads to substantial time savings on the most difficult instances. For example, on QPe, the CPU time is reduced from more than 8700 ms for IPPBB to about 2095.82 ms and 1834.35 ms for ISPPBB under the two inexactness settings. The inner iteration counts remain small, showing that the inexact dual solution strategy and warm start are effective. The performance profiles and purity results in Fig. \ref{fig1} further demonstrate the overall superiority of ISPPBB, while Fig. \ref{fig2} shows that the proposed method produces high-quality Pareto front approximations on representative ill-conditioned instances.

\begin{figure}[H]
	\centering
	\subfigure[Iterations]
	{
		\begin{minipage}[H]{.3\linewidth}
			\centering
			\includegraphics[scale=0.21]{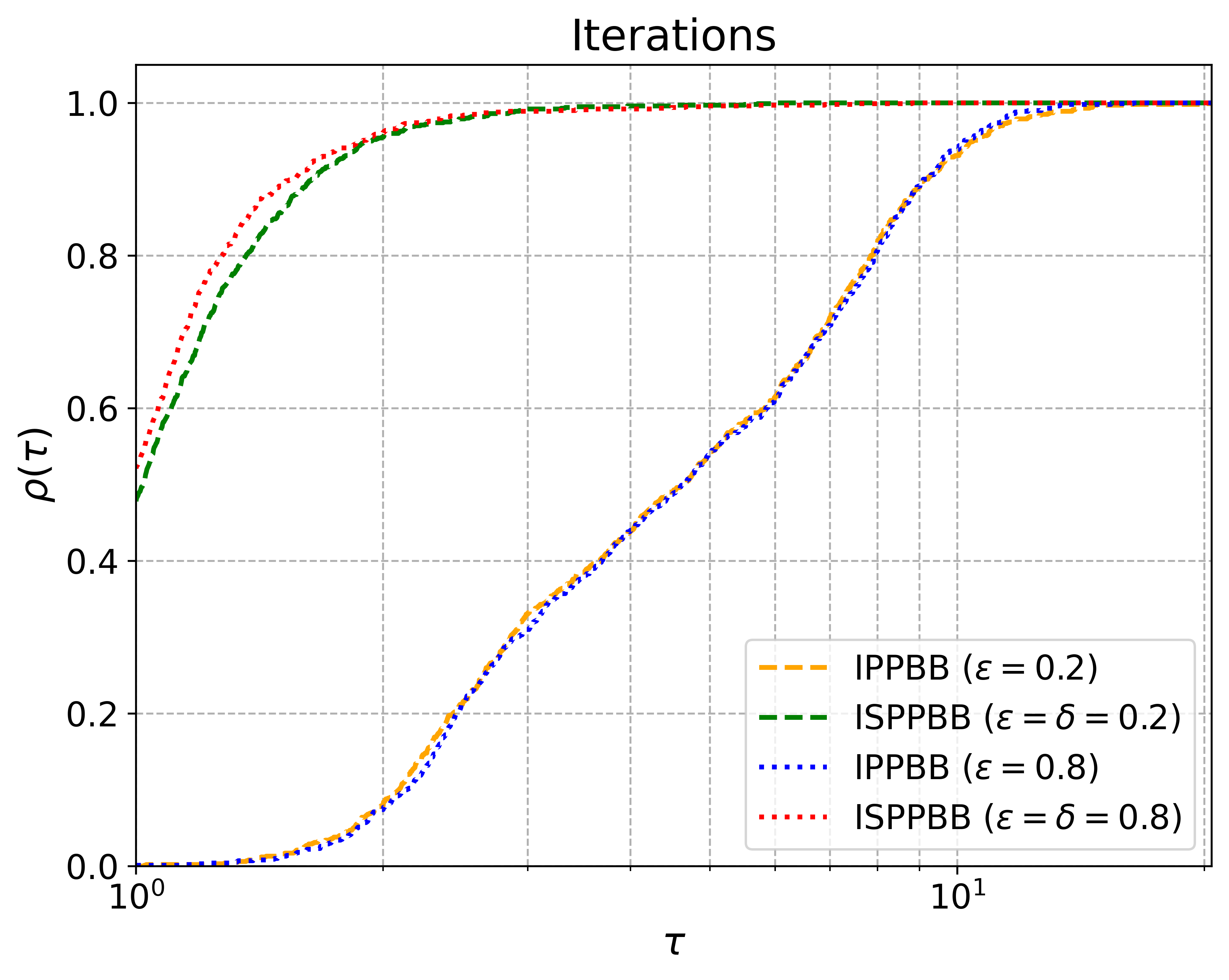} 
		\end{minipage}
	}
	\subfigure[CPU Time]
	{
		\begin{minipage}[H]{.3\linewidth}
			\centering
			\includegraphics[scale=0.21]{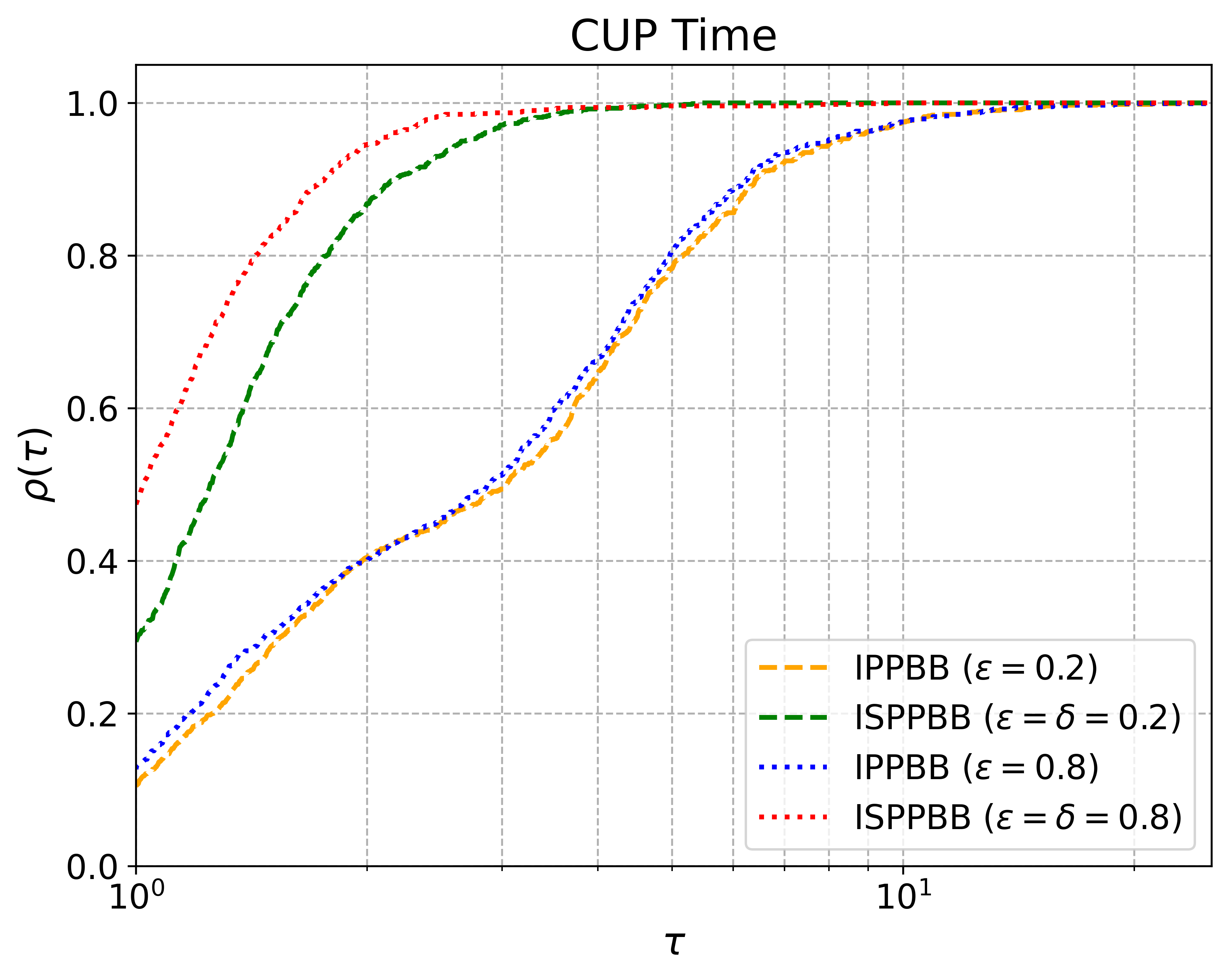} 
		\end{minipage}
	}
	\subfigure[Purity]
	{
		\begin{minipage}[H]{.3\linewidth}
			\centering
			\includegraphics[scale=0.21]{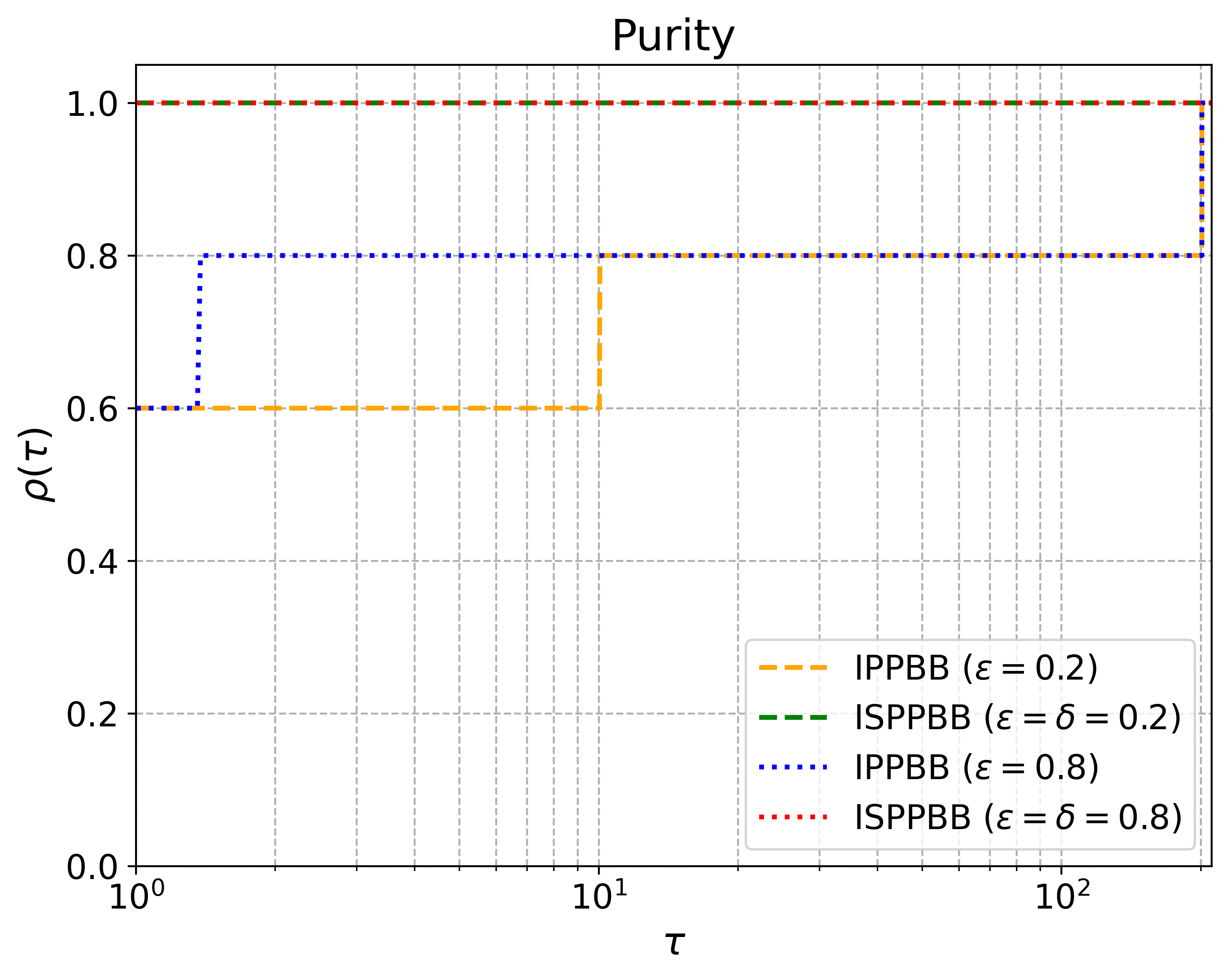} 
		\end{minipage}
	}
	\caption{Performance profiles and purity metric on $\ell_{1}$ regularization problems.}
	\label{fig1}
\end{figure}

\begin{figure}[H]
	\centering
	\subfigure[QPd]
	{
		\begin{minipage}[H]{.45\linewidth}
			\centering
			\includegraphics[scale=0.28]{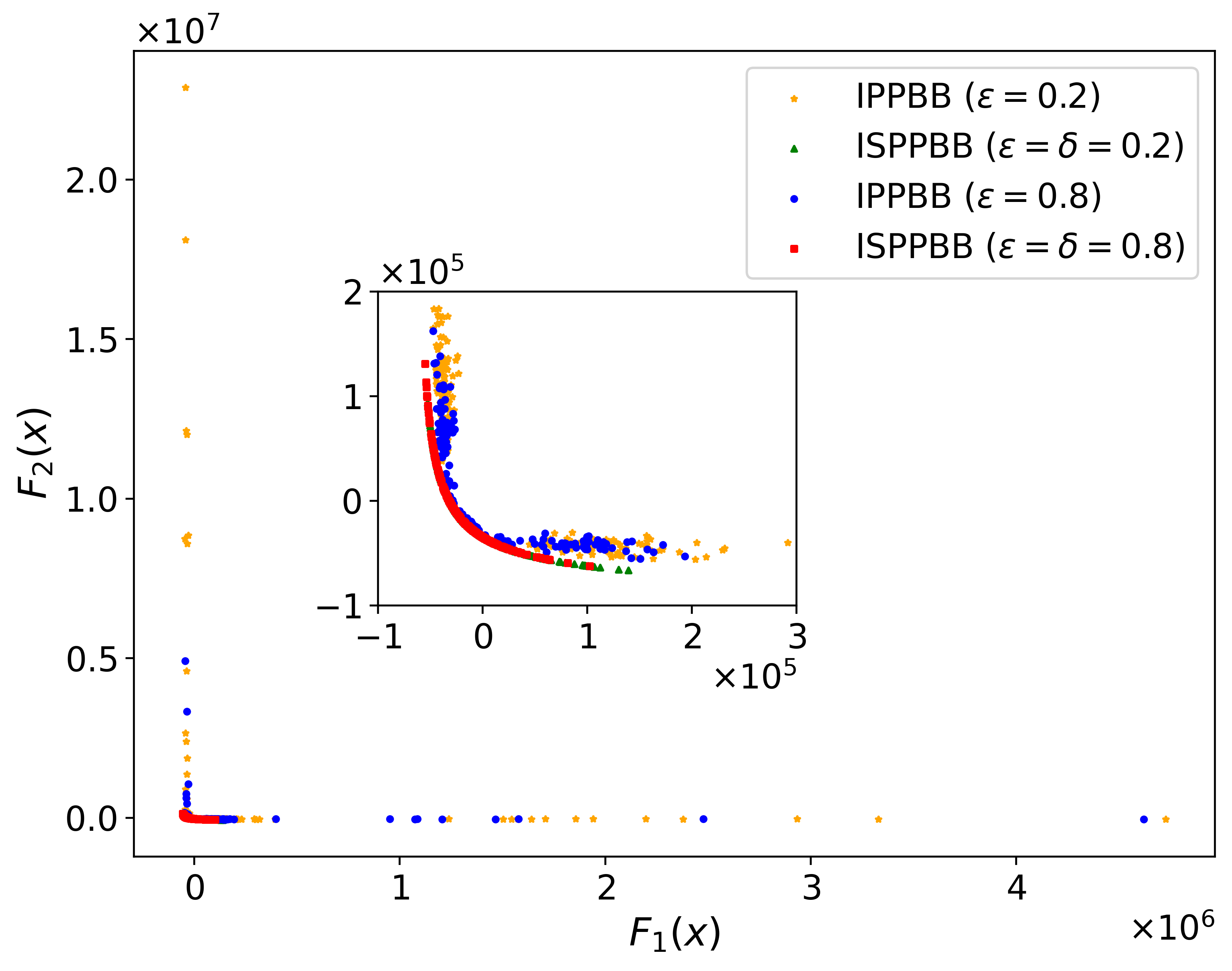} 
		\end{minipage}
	}
	\subfigure[QPe]
	{
		\begin{minipage}[H]{.45\linewidth}
			\centering
			\includegraphics[scale=0.28]{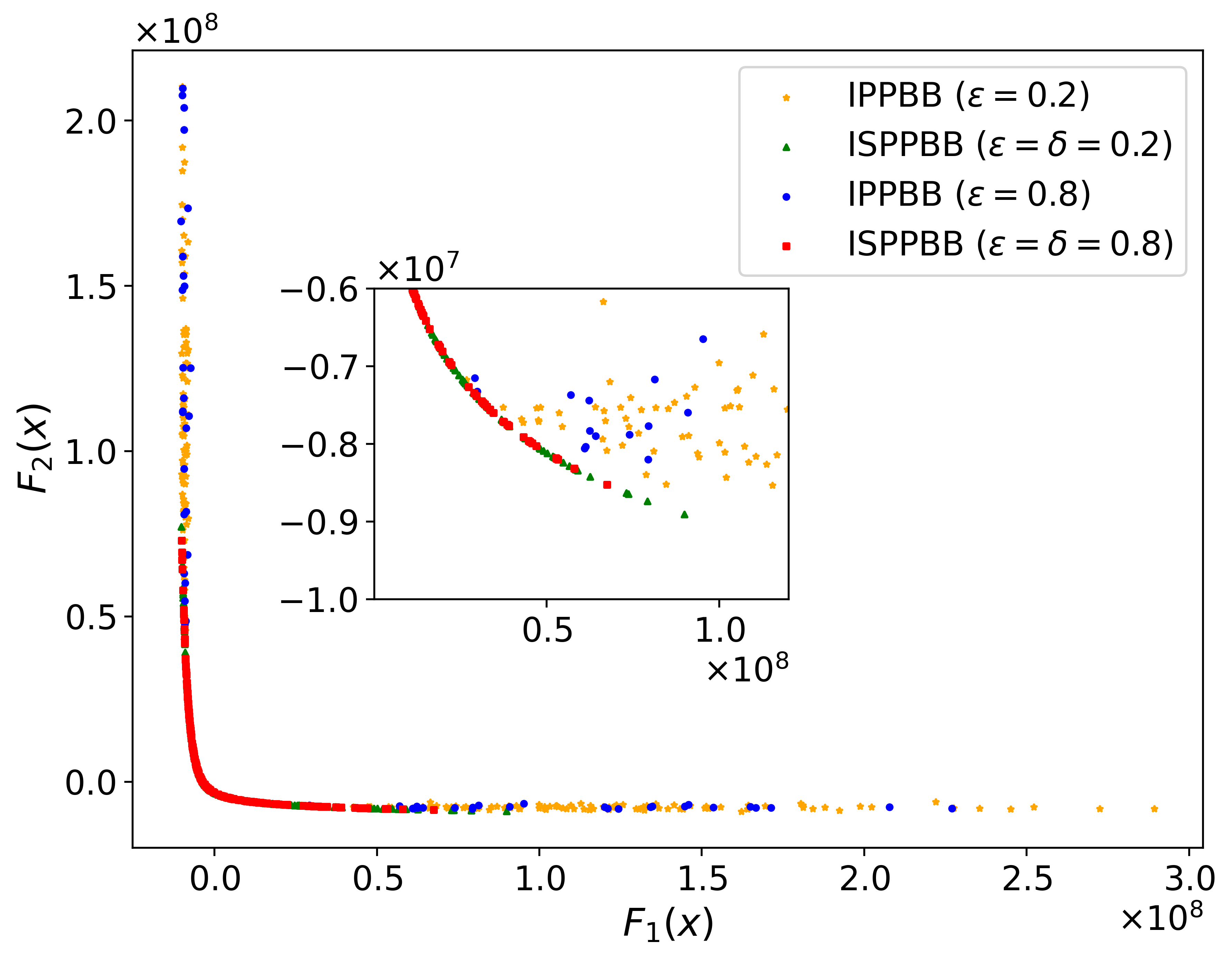} 
		\end{minipage}
	}

	\caption{Numerical results in value space obtained on $\ell_{1}$ regularization problems QPd and QPe.}\label{fig2}
\end{figure}

\subsection{Structured $\ell_{1}$ regularization problems}
For structured $\ell_{1}$ regularization problems, the nonsmooth parts are described as follows:
$$g_{i}(x):=\frac{1}{n}\|Ax\|_{1},~i=1,2,$$
where the matrix $A\in\mathbb{R}^{p\times n}$, $p=\left\lfloor \min\{{n}/{2},100\} \right\rfloor$, is constructed via singular value decomposition
\begin{equation}\label{A}
	A = U \Sigma V^\top,
\end{equation}
where $U$ and $V$ are random orthogonal matrices generated from Gaussian matrices via QR factorization. The singular values of $A$ are logarithmically spaced in $[1,\sigma_A]$ with $\sigma_A=\sqrt{50}$. To exploit the structure of $A$, we construct the linear-operator-aware preconditioner
\begin{equation}\label{P}
	P = V\left(\Lambda^\top\Lambda +
	\begin{bmatrix}
		\bm0_{m\times m}   & &  \\
		
		&  &\tilde{P}
	\end{bmatrix}\right)V^\top ,
\end{equation}
where $\tilde P\in\mathbb{S}_{++}^{n-m}$ is chosen as the identity matrix. The inverse $P^{-1}$ can therefore be computed analytically from the block structure. 
\begin{table}[h] 
	{\begin{center}
			\caption{Average number of iterations (iter), average CPU time (time ($ms$)), Average number of inner iterations required to compute $\lambda_{PBB}^{k,\epsilon}$ (iter$_{\epsilon}$) and $\lambda^{k,\epsilon,\delta}$ (iter$_{\epsilon,\delta}$) of tested algorithms on structured $\ell_{1}$ regularization problems.}\label{tab3}\vspace{-2mm}
		\end{center}
		\centering
		\resizebox{.99\columnwidth}{!}{
			\begin{tabular}{lrrrrrrrrrrrrrrrrr}
				\hline
				Problem &
				\multicolumn{3}{l}{IPPBB ($\epsilon=0.2$)} &
				&
				\multicolumn{4}{l}{ISPPBB ($\epsilon=\delta=0.2$)} 
				&
				\multicolumn{1}{l}{} &
				\multicolumn{3}{l}{IPPBB ($\epsilon=0.8$)}&
				\multicolumn{1}{l}{} &
				\multicolumn{4}{l}{ISPPBB ($\epsilon=\delta=0.8$)} \\ \cline{2-4} \cline{6-9} \cline{11-13} \cline{15-18}&
				\multicolumn{1}{r}{iter} &
				\multicolumn{1}{r}{time} &
				\multicolumn{1}{r}{iter$_{\epsilon}$} &
				\textbf{} &
				\multicolumn{1}{r}{iter} &
				\multicolumn{1}{r}{time} &
				\multicolumn{1}{r}{iter$_{\epsilon}$} &
				\multicolumn{1}{r}{iter$_{\epsilon,\delta}$} &
				\multicolumn{1}{r}{} &
				\multicolumn{1}{r}{iter} &
				\multicolumn{1}{r}{time} &
				\multicolumn{1}{r}{iter$_{\epsilon}$} &
				\multicolumn{1}{r}{} &
				\multicolumn{1}{r}{iter} &
				\multicolumn{1}{r}{time} &
				\multicolumn{1}{r}{iter$_{\epsilon}$} &
				\multicolumn{1}{r}{iter$_{\epsilon,\delta}$}  \\ \hline
				QPa & 585.64 & 109.89 & 1.47 &  & \textbf{237.01} & 97.51 & 1.69 & 0.59 &  & 586.08 & 110.61 & 1.36 &  & 244.56 & \textbf{91.44} & 1.55 & 0.40  \\
				QPb & 1892.48 & 329.35 & 1.21 &  & 684.44 & 269.63 & 1.47 & 0.64 &  & 1955.40 & 332.05 & 1.11 &  & \textbf{655.29} & \textbf{253.49} & 1.35 & 0.45  \\
				QPc & 1876.30 & 588.23 & 1.28 &  & 598.75 & 250.02 & 1.40 & 0.67 &  & 1839.72 & 562.40 & 1.11 &  & \textbf{541.00} & \textbf{208.04} & 1.23 & 0.39  \\
				QPd & 2000.00 & 598.60 & 0.55 &  & 1230.55 & 570.74 & 1.21 & 0.69 &  & 2000.00 & 579.62 & 0.22 &  & \textbf{1119.56} & \textbf{473.06} & 1.10 & 0.42  \\
				QPe & 2000.00 & 14327.80 & 0.33 &  & 1109.59 & 8845.36 & 1.68 & 0.75 &  & 2000.00 & 14535.64 & 0.47 &  & \textbf{965.87} & \textbf{7522.94} & 1.52 & 0.43  \\
				\hline
			\end{tabular}
	}}
\end{table}

This class of problems is more challenging than the standard $\ell_1$-regularized case because the nonsmooth term is composed with a nontrivial linear operator. The results in Table \ref{tab3} show that ISPPBB consistently requires fewer outer iterations than IPPBB on all tested instances. The advantage is particularly pronounced on QPc--QPe. For example, on QPe, IPPBB reaches the maximum number of 2000 iterations for both inexactness settings, while ISPPBB terminates after 1109.59 iterations for $\varepsilon=\delta=0.2$ and 965.87 iterations for $\varepsilon=\delta=0.8$. This confirms that the proposed subspace strategy remains effective even when the nonsmooth term has the structured form $g(Ax)$.

The CPU time results also show clear improvements in most cases. Since the structured $\ell_1$-regularized problems involve the matrix $A$, the linear-operator-aware preconditioner plays an important role in reducing the computational difficulty of the proximal-type subproblem. In addition, the average number of inner iterations is very small. In particular, the number of inner iterations for the subspace dual variable $\lambda^{k,\varepsilon,\delta}$ is usually less than one, indicating that the warm-started point often already satisfies the relaxed inexact descent condition. The performance profiles and purity metric in Fig. \ref{fig3} confirm the robustness of ISPPBB, while the value-space plots in Fig. \ref{fig4} illustrate that the proposed method obtains good Pareto front approximations on the structured $\ell_1$-regularized problems QPd and QPe.

\begin{figure}[H]
	\centering
	\subfigure[Iterations]
	{
		\begin{minipage}[H]{.3\linewidth}
			\centering
			\includegraphics[scale=0.21]{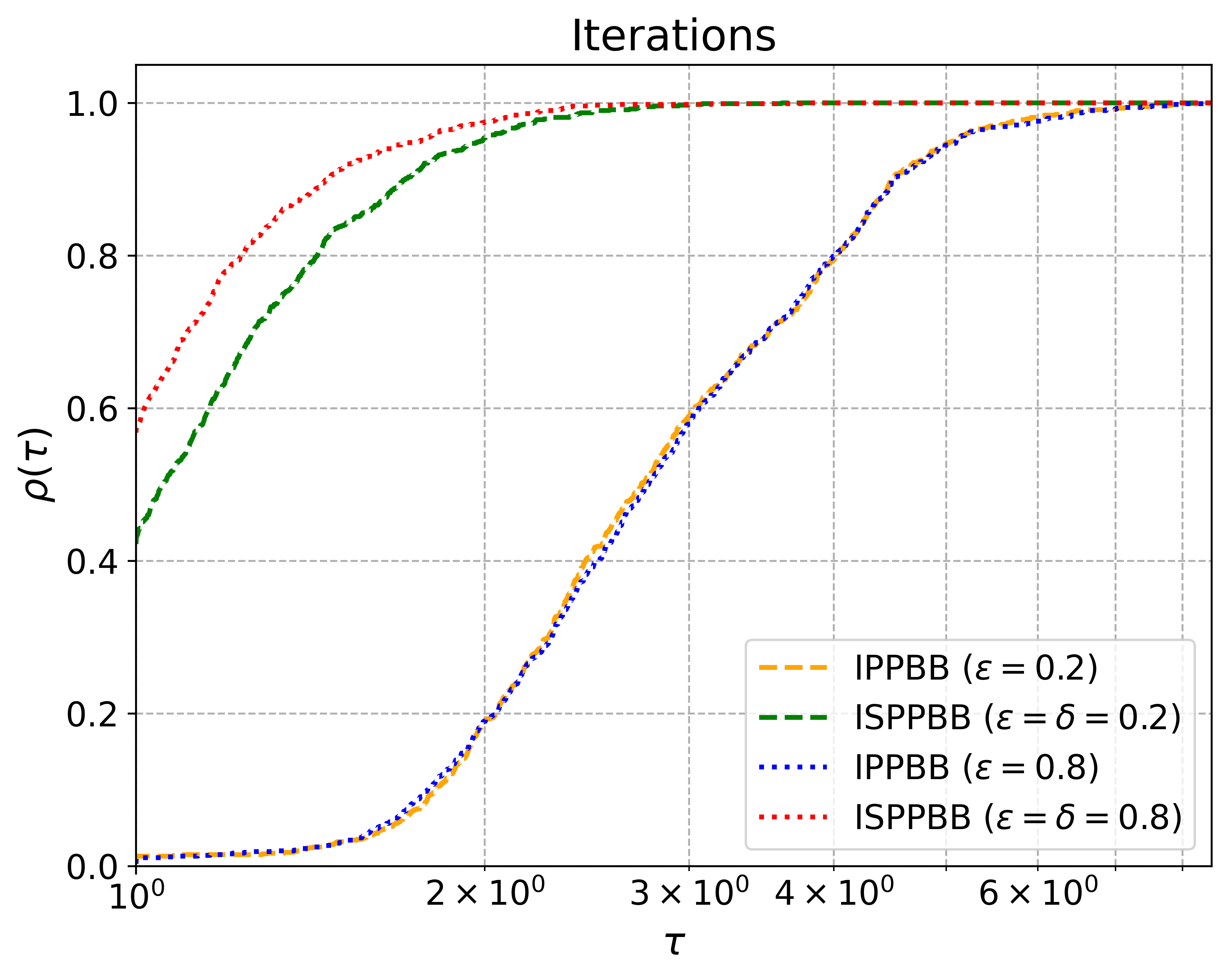} 
		\end{minipage}
	}
	\subfigure[CPU Time]
	{
		\begin{minipage}[H]{.3\linewidth}
			\centering
			\includegraphics[scale=0.21]{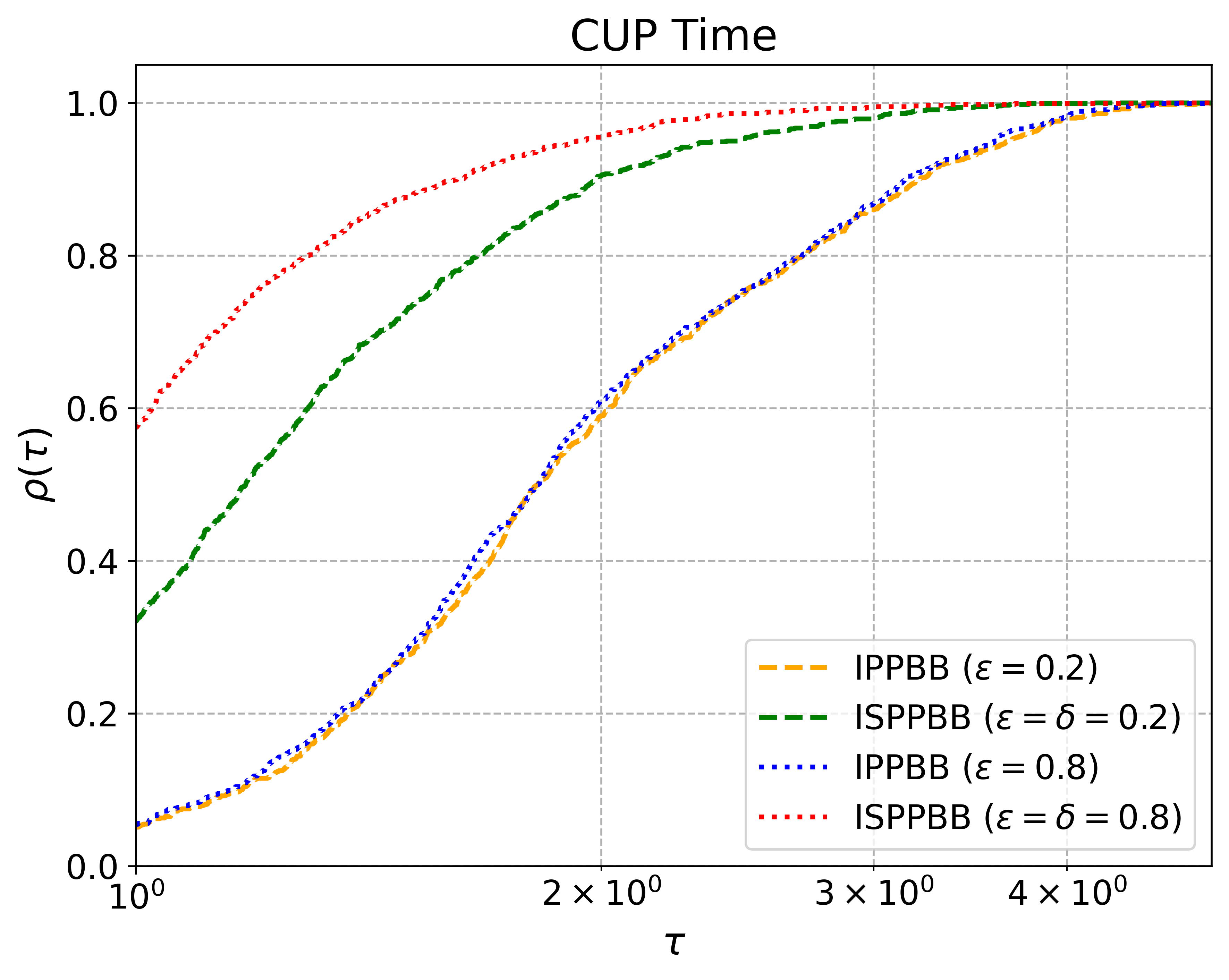} 
		\end{minipage}
	}
	\subfigure[Purity]
	{
		\begin{minipage}[H]{.3\linewidth}
			\centering
			\includegraphics[scale=0.21]{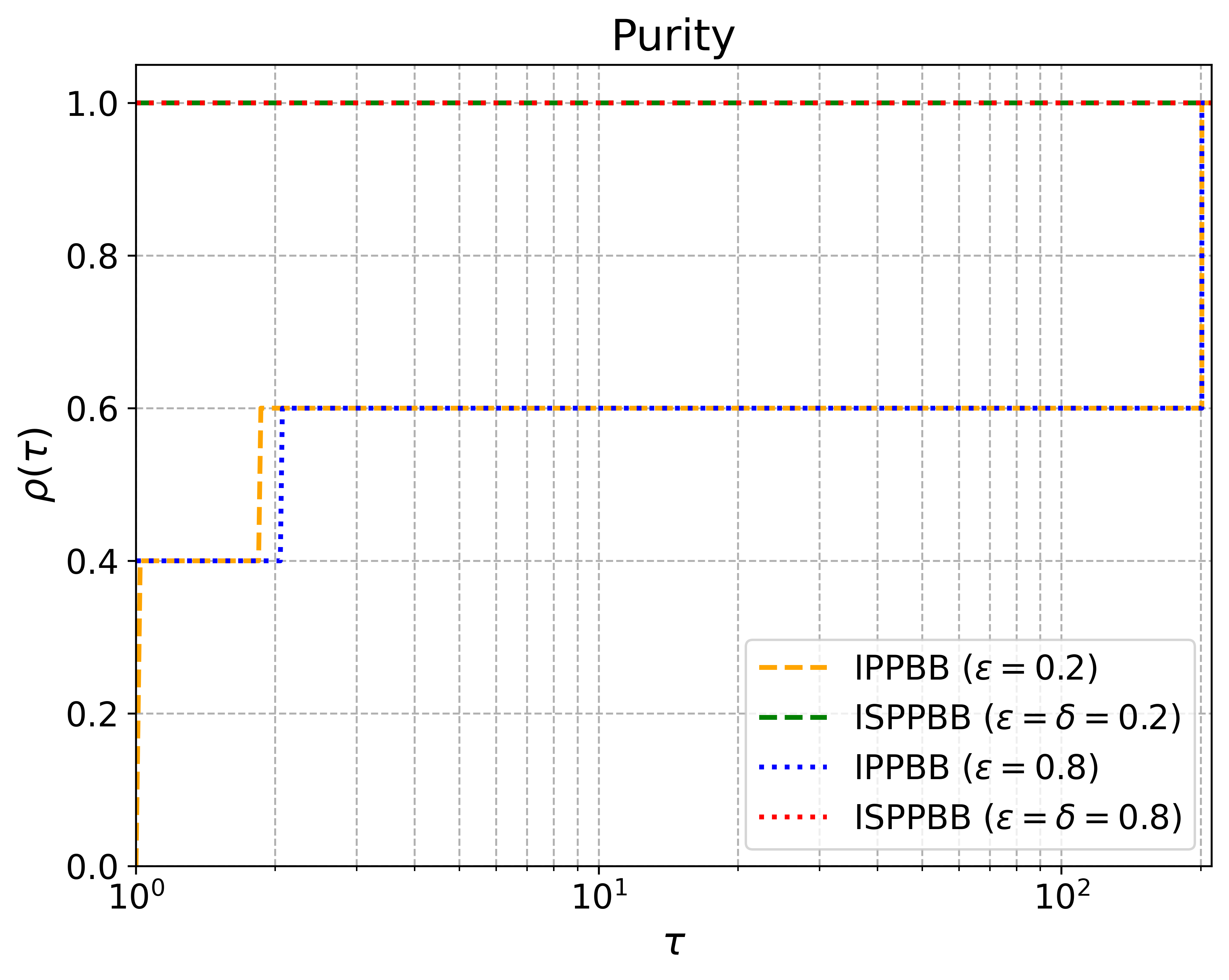} 
		\end{minipage}
	}
	\caption{Performance profiles and purity metric on structured $\ell_{1}$ regularization problems.}
	\label{fig3}
\end{figure}

\begin{figure}[H]
	\centering
	\subfigure[QPd]
	{
		\begin{minipage}[H]{.45\linewidth}
			\centering
			\includegraphics[scale=0.28]{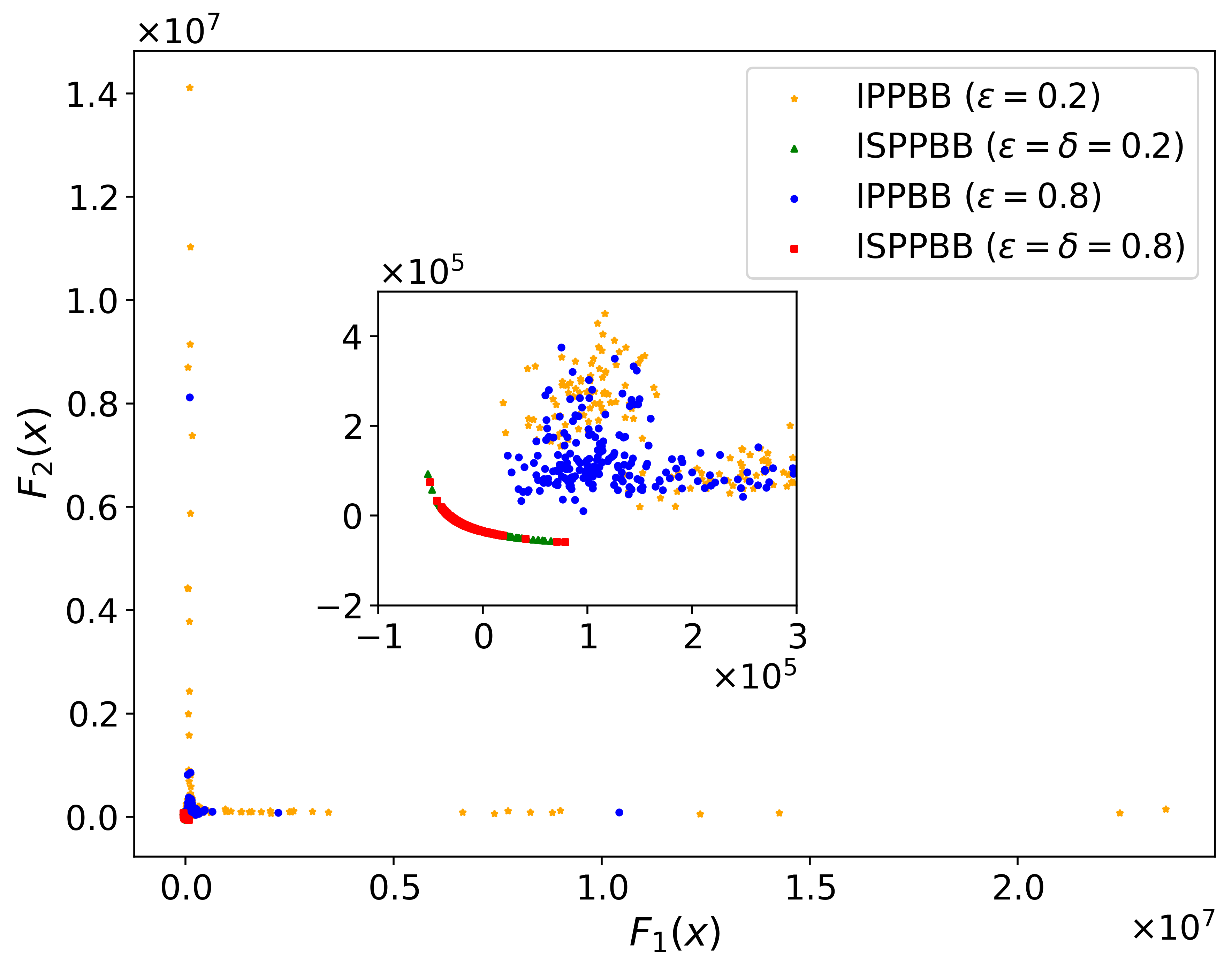} 
		\end{minipage}
	}
	\subfigure[QPe]
	{
		\begin{minipage}[H]{.45\linewidth}
			\centering
			\includegraphics[scale=0.28]{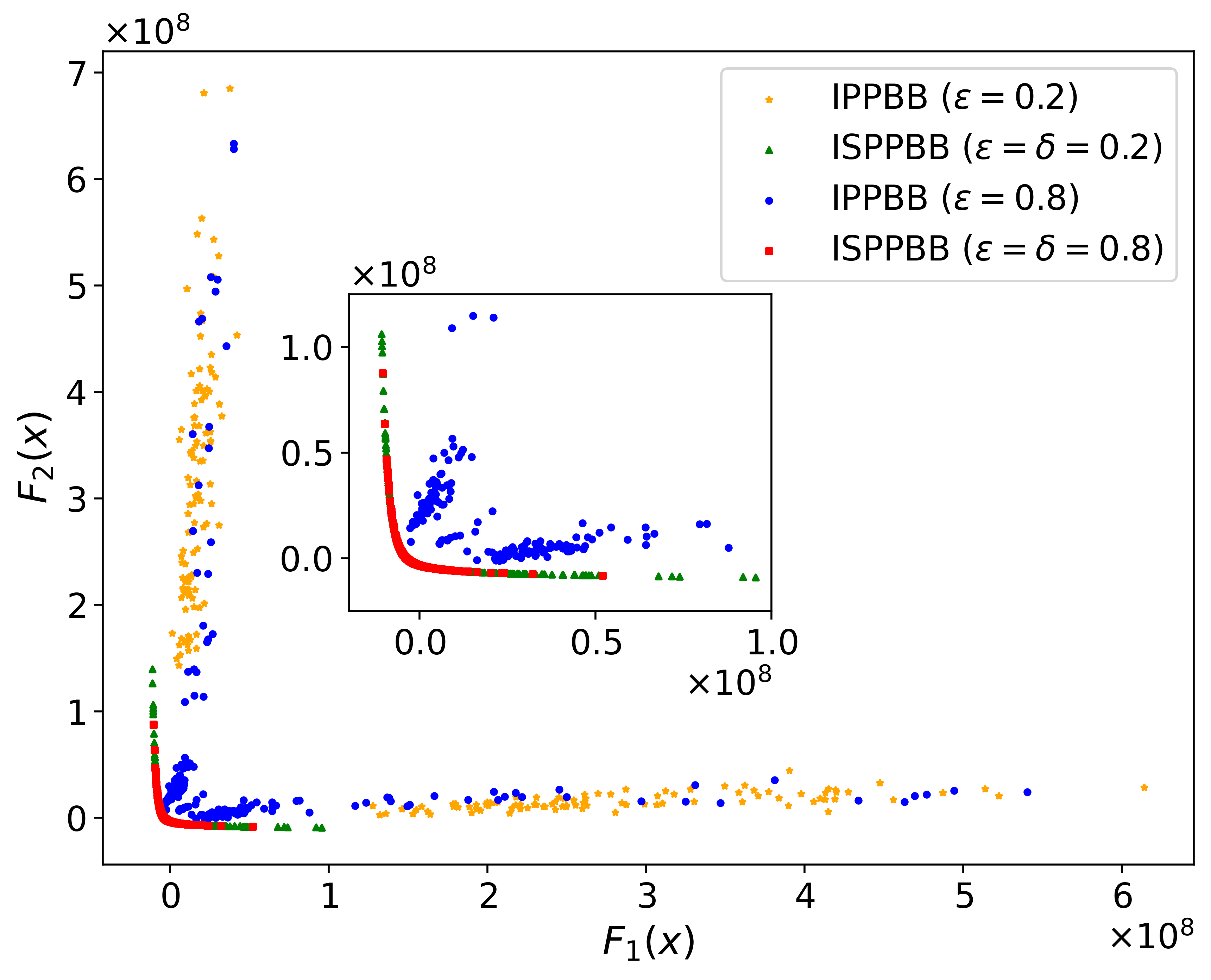} 
		\end{minipage}
	}
	
	\caption{Numerical results in value space obtained on structured $\ell_{1}$ regularization problems QPd and QPe.}\label{fig4}
\end{figure}

\subsection{Linear constrained problems}
The third group of experiments considers linearly constrained multiobjective problems. In this case, the nonsmooth term is the indicator function of a feasible set:
\begin{equation*}
	g_{i}(x)=\left\{
	\begin{aligned}
		&0, & x\in \mathcal{X}, \\
		&+\infty, &  \textrm{otherwise},
	\end{aligned}~~~i=1,2,
	\right.
\end{equation*}
where $\mathcal{X}:=\{x:c_{l}\preceq A_{1}x\preceq c_{u}, A_{2}x = c\}$. The matrix $A$ is generated as in the structured $\ell_1$-regularized case and then partitioned into
\begin{equation}
	A=
	\begin{bmatrix}
		A_1\\
		A_2
	\end{bmatrix},
\end{equation}
where $A_1\in\mathbb{R}^{p_{1}\times n}$, $A_2\in\mathbb{R}^{p_{2}\times n}$,
\begin{equation}
	p_{1}=\left\lfloor \frac{p}{2}\right\rfloor,
	\qquad
	p_{2}=p-p_{1}.
\end{equation}
The lower bound $c_l\in\mathbb{R}^{p_{1}}$ is randomly generated from $[0,\sigma_A]^{p_{1}}$, the upper bound is set as
\begin{equation}
	c_u=c_l+\sigma_A\mathbf{1}_{p_{1}},
\end{equation}
and the equality constraint vector $c\in\mathbb{R}^{p_2}$ is randomly generated from $[0,\sigma_A]^{p_2}$. The same linear-operator-aware preconditioner $P$ defined in \eqref{P} is used.

\begin{table}[h] 
	{\begin{center}
			\caption{Average number of iterations (iter), average CPU time (time ($ms$)), Average number of inner iterations required to compute $\lambda_{PBB}^{k,\epsilon}$ (iter$_{\epsilon}$) and $\lambda^{k,\epsilon,\delta}$ (iter$_{\epsilon,\delta}$) of tested algorithms on linear constrained problems.}\label{tab4}\vspace{-2mm}
		\end{center}
		\centering
		\resizebox{.99\columnwidth}{!}{
			\begin{tabular}{lrrrrrrrrrrrrrrrrr}
				\hline
				Problem &
				\multicolumn{3}{l}{IPPBB ($\epsilon=0.2$)} &
				&
				\multicolumn{4}{l}{ISPPBB ($\epsilon=\delta=0.2$)} 
				&
				\multicolumn{1}{l}{} &
				\multicolumn{3}{l}{IPPBB ($\epsilon=0.8$)}&
				\multicolumn{1}{l}{} &
				\multicolumn{4}{l}{ISPPBB ($\epsilon=\delta=0.8$)} \\ \cline{2-4} \cline{6-9} \cline{11-13} \cline{15-18}&
				\multicolumn{1}{r}{iter} &
				\multicolumn{1}{r}{time} &
				\multicolumn{1}{r}{iter$_{\epsilon}$} &
				\textbf{} &
				\multicolumn{1}{r}{iter} &
				\multicolumn{1}{r}{time} &
				\multicolumn{1}{r}{iter$_{\epsilon}$} &
				\multicolumn{1}{r}{iter$_{\epsilon,\delta}$} &
				\multicolumn{1}{r}{} &
				\multicolumn{1}{r}{iter} &
				\multicolumn{1}{r}{time} &
				\multicolumn{1}{r}{iter$_{\epsilon}$} &
				\multicolumn{1}{r}{} &
				\multicolumn{1}{r}{iter} &
				\multicolumn{1}{r}{time} &
				\multicolumn{1}{r}{iter$_{\epsilon}$} &
				\multicolumn{1}{r}{iter$_{\epsilon,\delta}$}  \\ \hline
				QPa & 73.34 & 12.41 & 1.59 &  & 44.01 & 12.11 & 1.55 & 0.69 &  & 73.30 & 11.76 & 1.47 &  & \textbf{40.82} & \textbf{11.34} & 1.38 & 0.25  \\
				QPb & 328.36 & 54.58 & 1.42 &  & 94.58 & 27.00 & 1.54 & 1.02 &  & 335.22 & 53.85 & 1.38 &  & \textbf{86.83} & \textbf{23.34} & 1.43 & 0.53  \\
				QPc & 271.04 & \textbf{54.49} & 1.86 &  & 171.82 & 80.56 & 1.81 & 0.88 &  & 294.21 & 55.31 & 1.67 &  & \textbf{169.12} & 68.16 & 1.62 & 0.27  \\
				QPd & 2000.00 & 557.95 & 0.80 &  & \textbf{652.62} & 511.88 & 3.37 & 1.24 &  & 2000.00 & 541.29 & 0.57 &  & 667.84 & \textbf{429.50} & 2.85 & 0.48  \\
				QPe & 2000.00 & 13028.71 & 0.31 &  & 706.76 & \textbf{6309.80} & 1.96 & 0.77 &  & 2000.00 & 14312.61 & 1.23 &  & \textbf{691.94} & 12710.20 & 2.05 & 0.53  \\
				\hline
			\end{tabular}
	}}
\end{table}

As shown in Table \ref{tab4}, ISPPBB again significantly reduces the number of outer iterations compared with IPPBB. On the more difficult instances QPd and QPe, IPPBB reaches the maximum iteration limit of 2000, whereas ISPPBB terminates within about 650--710 iterations. This demonstrates that the proposed subspace method is effective not only for regularized problems but also for multiobjective linearly constrained problems.

The CPU time results show that ISPPBB is particularly beneficial for difficult ill-conditioned instances, although the subspace overhead can sometimes affect the CPU time on smaller or moderate-scale problems. For QPe, ISPPBB substantially reduces the CPU time when $\varepsilon=\delta=0.2$, from 13028.71 ms to 6309.80 ms. The average number of inner iterations remains modest, which again supports the efficiency of the inexact stopping conditions and the warm-starting strategy. The performance profiles and purity metric in Fig. \ref{fig5} show the overall advantage of ISPPBB, and the value-space results in Fig. \ref{fig6} indicate that the proposed method can generate competitive Pareto front approximations under linear constraints.

\begin{figure}[H]
	\centering
	\subfigure[Iterations]
	{
		\begin{minipage}[H]{.3\linewidth}
			\centering
			\includegraphics[scale=0.21]{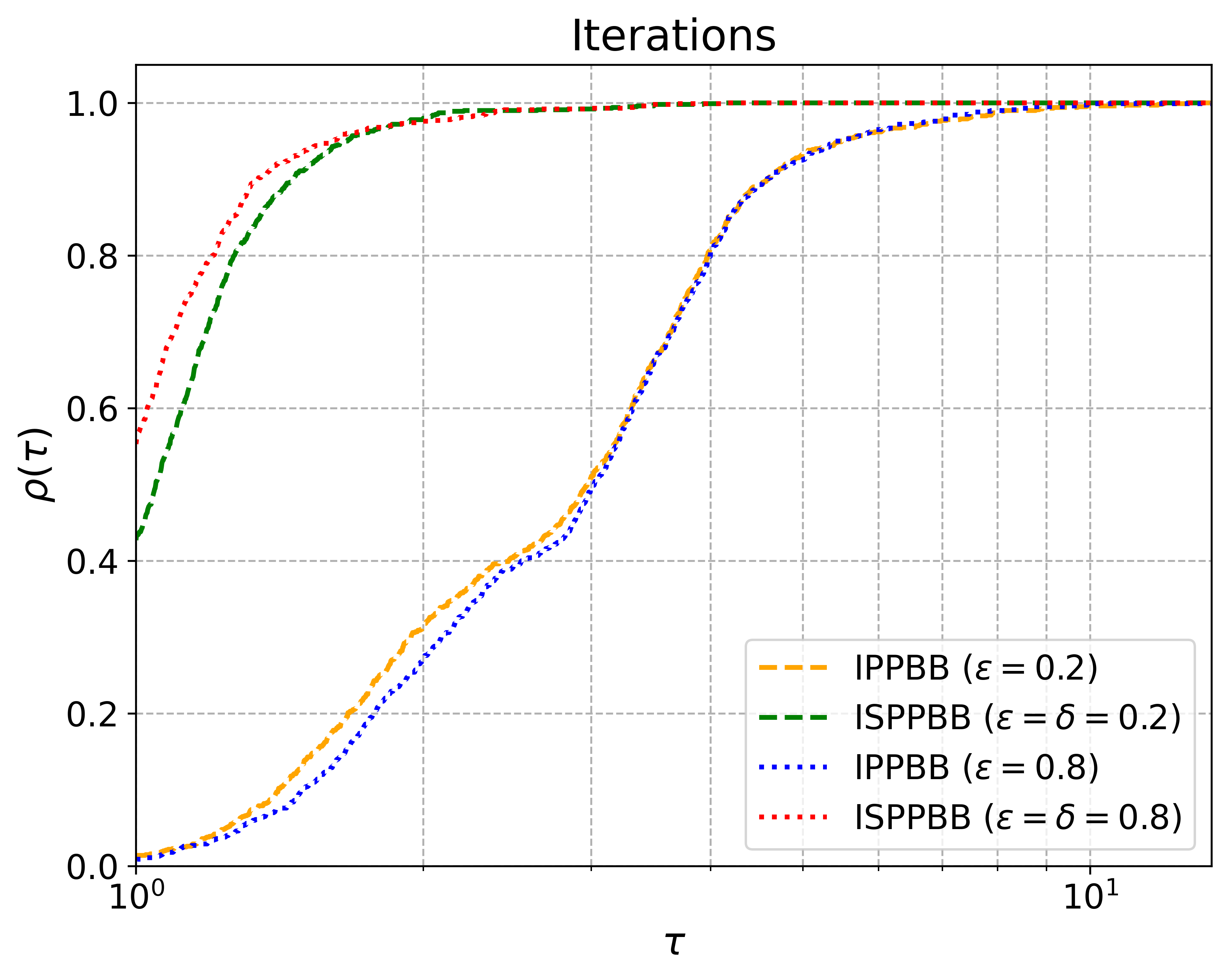} 
		\end{minipage}
	}
	\subfigure[CPU Time]
	{
		\begin{minipage}[H]{.3\linewidth}
			\centering
			\includegraphics[scale=0.21]{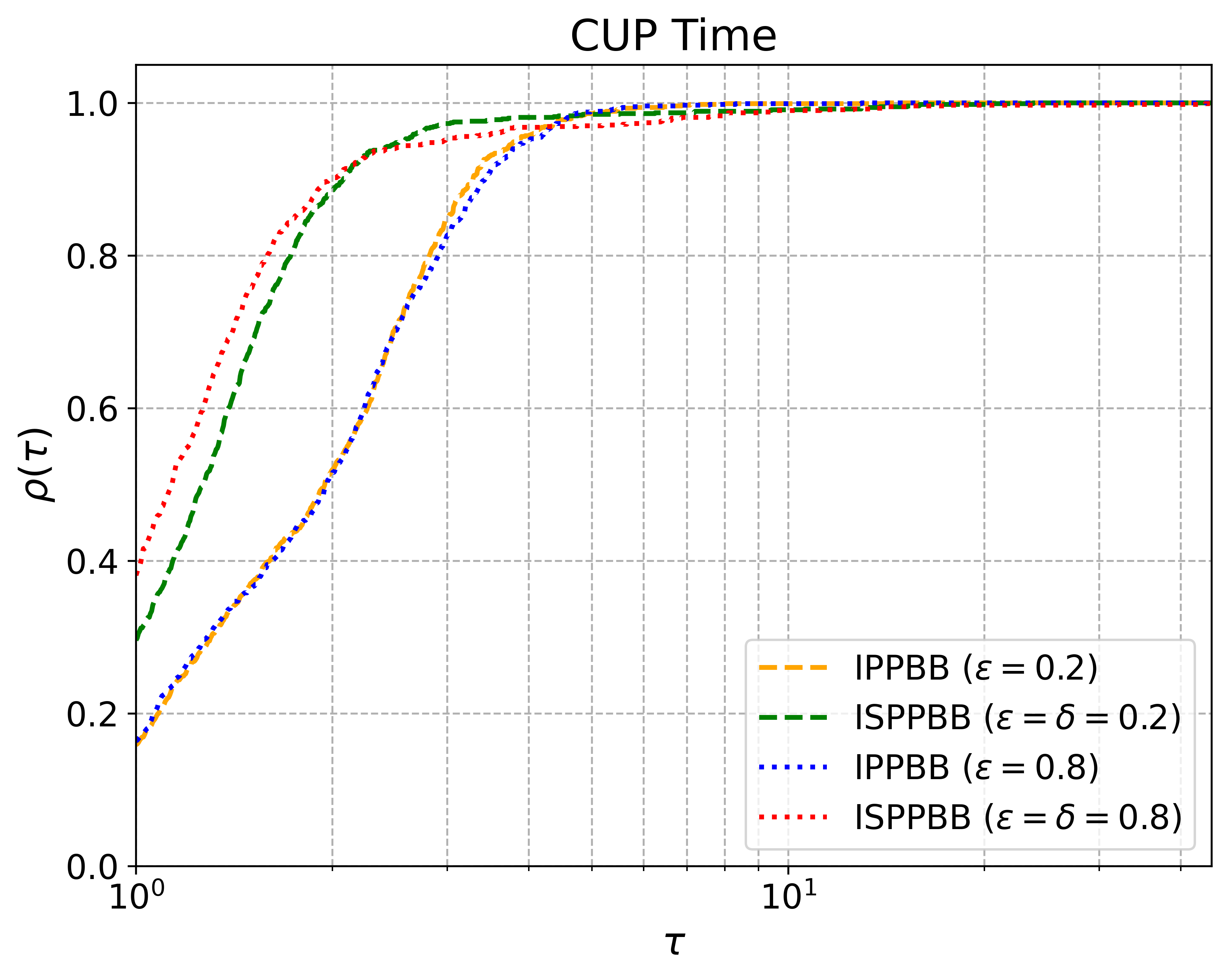} 
		\end{minipage}
	}
	\subfigure[Purity]
	{
		\begin{minipage}[H]{.3\linewidth}
			\centering
			\includegraphics[scale=0.21]{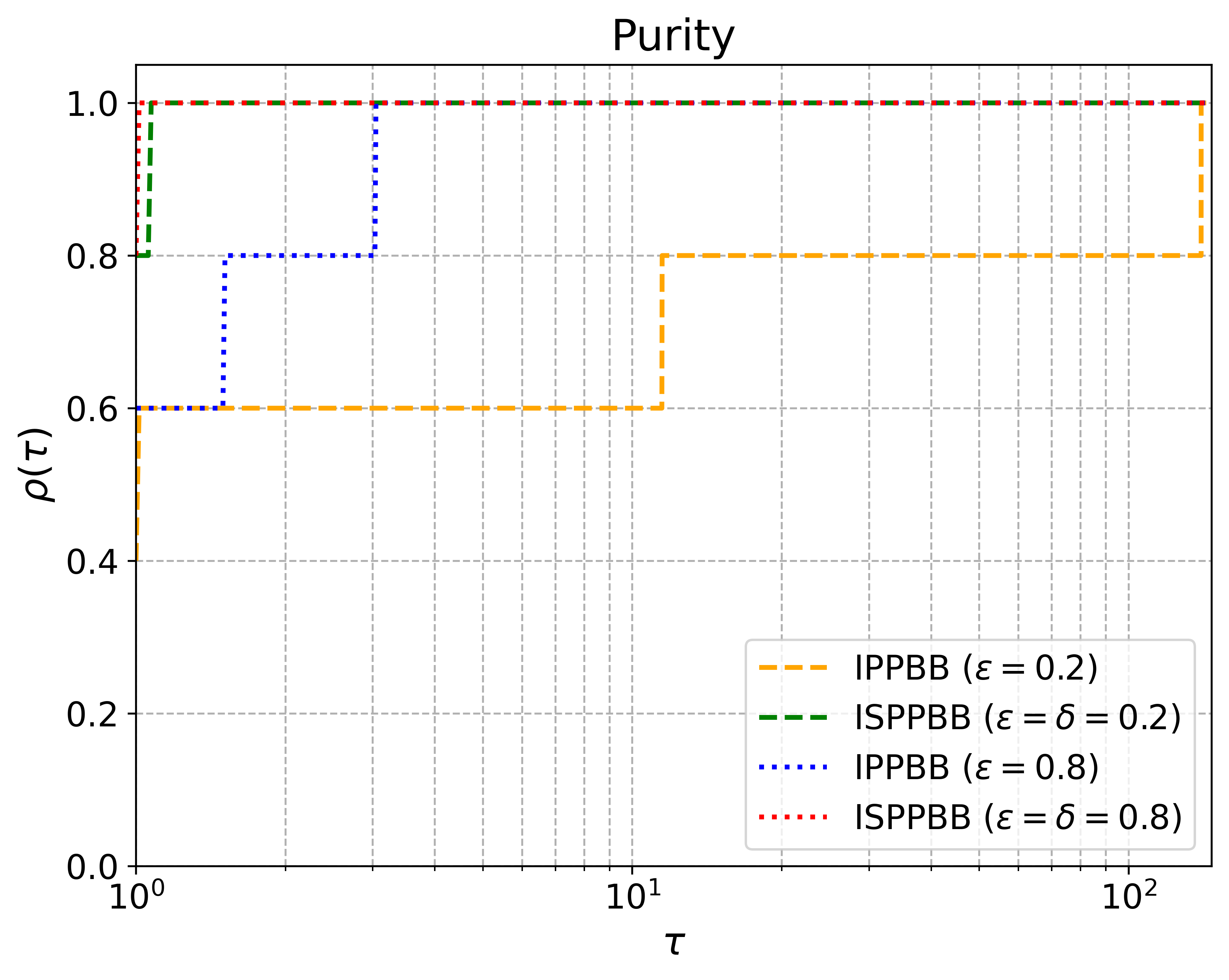} 
		\end{minipage}
	}
	\caption{Performance profiles and purity metric on linear constrained problems.}
	\label{fig5}
\end{figure}

\begin{figure}[H]
	\centering
	\subfigure[QPd]
	{
		\begin{minipage}[H]{.45\linewidth}
			\centering
			\includegraphics[scale=0.28]{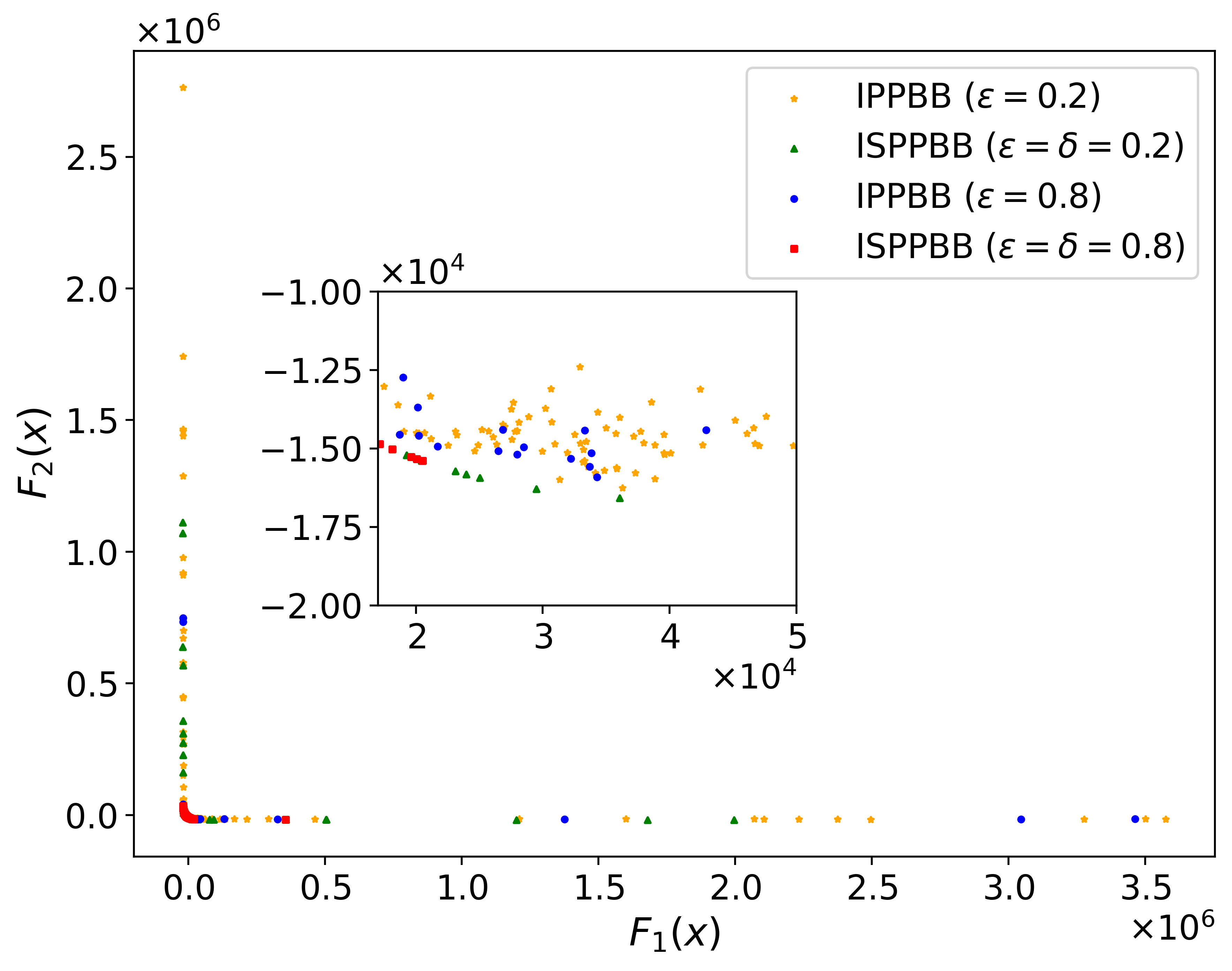} 
		\end{minipage}
	}
	\subfigure[QPe]
	{
		\begin{minipage}[H]{.45\linewidth}
			\centering
			\includegraphics[scale=0.28]{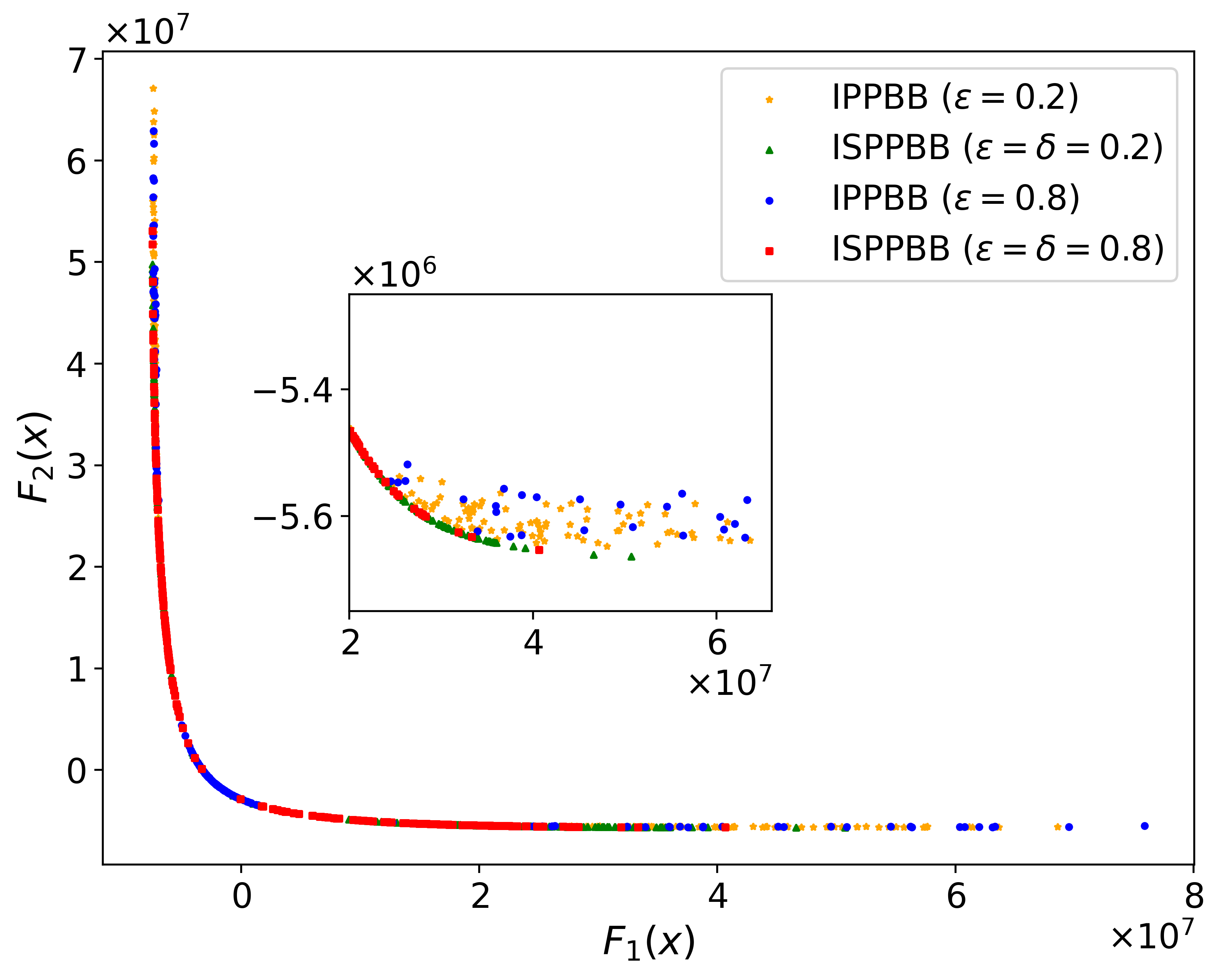} 
		\end{minipage}
	}
	
	\caption{Numerical results in value space obtained on linear constrained problems QPd and QPe.}\label{fig6}
\end{figure}

Overall, the three groups of experiments demonstrate that the proposed ISPPBB method effectively combines inexact dual solution, warm start, linear-operator-aware preconditioning, and subspace acceleration. Compared with the non-subspace variant IPPBB, ISPPBB consistently reduces the number of outer iterations and performs especially well on high-dimensional ill-conditioned problems. The results also show that the additional subspace dual subproblem is inexpensive in practice, since the relaxed inexact descent condition can often be satisfied within very few inner iterations.
\section{Conclusions}\label{sec7}
This paper developed a subspace second-order proximal framework for multiobjective composite optimization. The method is based on a preconditioned proximal Barzilai-Borwein model that combines a common metric for capturing aggregated curvature information with objective-wise Barzilai-Borwein scalings for reducing imbalance among objectives. 

To avoid evaluating complicated metric proximal mappings, we introduced a two-dimensional subspace model generated by a proximal-gradient-type direction and a projected historical direction. By using a conjugate basis with respect to the preconditioning metric, the subspace model can be reduced to tractable one-dimensional subproblems. We also extended the framework to problems with objectives of the form $f_i(x)+g_i(Ax)$. The proposed linear-operator-aware preconditioner separates the difficulty caused by the linear operator from the curvature approximation and leads to explicit proximal computations in the associated dual subproblems.

The convergence analysis was established for an inexact version of the method. The inexactness conditions are weak enough to allow inexpensive inner iterations but strong enough to guarantee sufficient descent. Under standard compactness and smoothness assumptions, every accumulation point is Pareto critical; under an additional global error-bound condition, a linear convergence rate is obtained. Numerical experiments on ill-conditioned quadratic problems with $\ell_1$-regularization, structured $\ell_1$-regularization, and linear constraints show that the subspace variant reduces the number of outer iterations compared with the non-subspace preconditioned method, while the warm-started spectral projected-gradient solver keeps the inner cost low.

Future work includes extensions to broader nonconvex composite settings and worst-case complexity analysis for the inexact subspace scheme.

\begin{acknowledgements}
	This work was funded by the Major Program of the National Natural Science Foundation of China [grant numbers 11991020, 11991024]; the Key Program of the National Natural Science Foundation of China [grant number 12431010]; the General Program of the National Natural Science Foundation of China [grant number 12171060]; NSFC-RGC (Hong Kong) Joint Research Program [grant number 12261160365]; the Team Project of Innovation Leading Talent in Chongqing [grant number CQYC20210309536]; the Natural Science Foundation of Chongqing [grant numbers ncamc2022-msxm01, CSTB2024NSCQ-LZX0140]; Major Project of Science and Technology Research Program of Chongqing Education Commission of China [grant number KJZD-M202300504]; the Science and Technology Research Program of Chongqing Education commission of China [grant number KJQN202400520]; the Chongqing Postdoctoral Research Project Special Grant [grant number 2024CQBSHTB1007] and Foundation of Chongqing Normal University [grant numbers 22XLB005, 22XLB006].
\end{acknowledgements}

\end{document}